\documentclass[11pt]{article}

\usepackage[english]{babel}

\usepackage[letterpaper,margin=0.7in,marginparwidth=0.7cm]{geometry}
\usepackage{amsmath}
\usepackage{amssymb}
\usepackage{amsthm}
\usepackage{graphicx}
\usepackage[colorlinks=true, allcolors=blue]{hyperref}
\usepackage{tikz}
\usepackage{todonotes}
\usepackage{hyperref}
\usepackage{diagbox}
\usetikzlibrary{arrows.meta}
\usetikzlibrary{patterns}
\usetikzlibrary{shapes.geometric}
\usepackage{MnSymbol}

\usetikzlibrary{calc, positioning, fit, shapes.geometric}
\pgfdeclarelayer{background}
\pgfsetlayers{background,main}

\newtheorem{theorem}{Theorem}     %[section]

\newtheorem{lemma}[theorem]{Lemma}

\newtheorem{conjecture}[theorem]{Conjecture}

\newtheorem{definition}[theorem]{Definition}

\newtheorem{claim}{Claim}

\newcommand*{\myproofname}{Proof}
\newenvironment{myproof}[1][\myproofname]{\begin{proof}[#1]}{\end{proof}}

\newcommand{\dom}[1]{\gamma\left(#1\right)}

\usepackage{xcolor}

\title{Domination of subcubic planar graphs with large girth}
\author{Eun-Kyung Cho$^{a}$\thanks{Email: \texttt{ekcho2020@gmail.com}} \and 
	Eric Culver$^{b}$\thanks{Email: \texttt{eric.culver@mathematics.byu.edu}} \and 
	Stephen G. Hartke$^{c}$\thanks{Email: \texttt{stephen.hartke@ucdenver.edu}} \and
	Vesna Iršič$^{d,e}$\thanks{Email: \texttt{vesna.irsic@fmf.uni-lj.si}}}

\begin{document}
	\maketitle
	
	\begin{center}
		$^a$ Hankuk University of Foreign Studies, Republic of Korea\\
		\medskip
		
		$^b$ Brigham Young University, USA\\
		\medskip
		
		$^c$ University of Colorado Denver, USA\\
		\medskip
		
		$^d$ University of Ljubljana, Slovenia\\
		\medskip
		
		$^e$ Institute of Mathematics, Physics and Mechanics, Slovenia\\
		\medskip
		
	\end{center}
	
	\begin{abstract}
		Since Reed conjectured in 1996 that the domination number of a connected cubic graph of order $n$ is at most $\lceil \frac13 n \rceil$, the domination number of cubic graphs has been extensively studied. It is now known that the conjecture is false in general, but Henning and Dorbec showed that it holds for graphs with girth at least $9$. 
		Zhu and Wu stated an analogous conjecture for 2-connected cubic planar graphs.
		
		In this paper, we present a new upper bound for the domination number of subcubic planar graphs: if $G$ is a subcubic planar graph with girth at least 8, then $\gamma(G) < n_0 + \frac{3}{4} n_1 + \frac{11}{20} n_2 +  \frac{7}{20} n_3$, where $n_i$ denotes the number of vertices in $G$ of degree $i$, for $i \in \{0,1,2,3\}$. We also prove that if $G$ is a subcubic planar graph with girth at least 9, then $\gamma(G) <  n_0 + \frac{13}{17} n_1 + \frac{9}{17} n_2 + \frac{6}{17} n_3$.
	\end{abstract}
	
	\section{Introduction}
	
	A \emph{dominating set} of $G$ is a set $D$ of vertices of $G$ such that every vertex of $G$ is either in $D$ or is adjacent to a vertex in $D$.
	The \emph{domination number} of $G$ is the size of the smallest dominating set of $G$, and is denoted by $\gamma(G)$.
	
	For a vertex $v$ in $G$, $v$ dominates the vertices in its closed neighborhood $N_G[v]$, so the larger $\deg_G(v)$, the greater the number of vertices dominated by $v$. 
	Thus, it is reasonable to infer that the greater $\delta(G)$, the smaller $\gamma(G)$.
	In 1962, Ore~\cite{ore1962theory} proved that a graph $G$ with $\delta(G) \ge 1$ satisfies $\gamma(G) \le \frac{1}{2}|V(G)|$.
	This bound is tight; for example, the equality is attained by an infinite family of graphs obtained by attaching a leaf to each vertex of an arbitrary graph. 
	In 1973, Blank~\cite{blank1973estimate} (and in 1989, independently McCuaig and Shepherd~\cite{McCuaig+1989}) proved that a connected graph with $\delta(G) \ge 2$ and $|V(G)| \ge 8$ satisfies $\gamma(G) \le \frac{2}{5}|V(G)|$.
	Furthermore, McCuaig and Shepherd~\cite{McCuaig+1989} also characterized those connected graphs $G$ with $\delta(G) \ge 2$ such that any dominating set contains at least $\frac{2}{5}|V(G)|$ vertices.
	In 1996, Reed~\cite{Reed1996} proved that a graph with $\delta(G) \ge 3$ satisfies $\gamma(G) \le \frac{3}{8}|V(G)|$.  
	Cubic graphs attaining the bound can be found in~\cite{goddard+2012, Reed1996}, but no arbitrarily large family of connected graphs that attain the bound has been found.
	
	In \cite{Reed1996}, Reed also conjectured that a connected cubic graph $G$ satisfies $\gamma(G) \le \lceil \frac{1}{3}|V(G)| \rceil$.
	This conjecture attracted much attention and has been a popular topic for a long time.
	However, in 2005, Kostochka and Stodolsky~\cite{Kostochka+2005} disproved the conjecture by constructing an infinite family of connected cubic graphs $\{G_k\}_{k=1}^{\infty}$ such that $|V(G_k)| = 46k$ and $\gamma(G) \ge 16k$, which implies $\frac{\gamma(G_k)}{|V(G_k)|} \ge \frac{1}{3} + \frac{1}{69}$ for all $k \ge 1$. The graphs in this family have girth $4$ and are not $2$-connected. However, Reed's conjecture is also false for $2$-connected cubic graphs; counterexamples were presented independently by Stodolsky~\cite{Stodolsky2008} and Kelmans~\cite{kelmans2006counterexamples}. Note that these counterexamples have girth $4$.
	These counterexamples raised two main questions: 
	\begin{itemize}
		\item Which is the correct bound for the domination number of cubic graphs? 
		\item Which bound holds if an additional condition on the girth of the graph is given?
	\end{itemize}
	
	For cubic graph with order greater than $8$, $\gamma(G) \le \frac{5}{14}|V(G)|$ was proved by Kostochka and Stocker~\cite{Kostochka+2009b}, and later improved by Cho, Choi, Kwon, and Park~\cite{cho2021tight} to independent domination when there is no $4$-cycle in a graph.
	These results are tight, as shown by a generalized Petersen graph on $14$ vertices.
	
	Obtaining good bounds for the domination number of planar graphs has been of separate interest. MacGillivray and Seyffarth~\cite{macgillivray+1996} showed that determining the domination number of a graph is NP-hard even for planar graphs. Upper bounds for the domination number of planar graphs of diameter $2$ and $3$ have been studied in~\cite{dorfling+2006, goddard+2002, macgillivray+1996}.
	Zhu and Wu~\cite{zhu+2015} made a conjecture on the domination number of cubic planar graphs.
	
	\begin{conjecture}[{\cite{zhu+2015}}]
		\label{conj:2concubicplanar}
		If $G$ is a $2$-connected cubic planar graph, then $\gamma(G) \leq \frac{|V(G)|}{3}$.
	\end{conjecture}
	
	We do not know of any counterexamples to the conjecture when the condition of 2-connectedness is removed. The conjecture might be true without any hypothesis on the graph connectivity.
	
	The problem has also been studied for total domination, and Henning and McCoy~\cite{henning+2009} showed that planar graphs of diameter $2$ have total domination number $3$. More general upper bounds were obtained for the prefect Italian domination number of planar graphs~\cite{lauri+2020}, the global domination number of planar graphs~\cite{enciso+2008}, the power domination number of maximal planar graphs~\cite{dorbec+2019}, and the independent domination number of $2$-connected planar graphs~\cite{abrishami+2019}.
	Moreover, Abrishami, Henning, and Rahbarnia~\cite{abrishami+2019} proposed a weaker version of Conjecture~\ref{conj:2concubicplanar}.
	Note that a bipartite cubic graph is $2$-connected.
	
	\begin{conjecture}[{\cite{abrishami+2019}}]
		\label{conj:theother}
		If $G$ is a bipartite cubic planar graph, then $\gamma(G) \leq \frac{|V(G)|}{3}$.
	\end{conjecture}
	
	Conjecture~\ref{conj:theother} is known to be best possible and is true for graphs on at most $24$ vertices~\cite{abrishami+2019}.
	
	Favaron~\cite{favaron92} showed that if $T$ is a tree on $n \geq 2$ vertices with $q$ leaves, then $\gamma(T) \leq \frac{n+q}{3}$. Applying this result to a subcubic tree $T$ gives the following bound:
	$$\gamma(T) \leq \frac23 n_1(T) + \frac13 n_2(T) + \frac13 n_3(T),$$
	where $n_i(G)$ denotes the number of vertices of degree $i$ in $G$.
	
	Finding a girth condition which implies Reed's upper bound has been of great interest to many mathematicians.
	In 2006, Kawarabayashi, Plummer, and Saito~\cite{Kawarabayashi+2006} proved that every $2$-edge-connected cubic graph with girth at least $g$ (where $g$ is divisible by $3$) satisfies $\gamma(G) \le (\frac{1}{3}+\frac{1}{3g+3})|V(G)|$.
	In 2009, Kostochka and Stodolsky~\cite{kostochka2009upper} proved that a cubic graph with order greater than $8$ and girth at least $g$ has $\gamma(G) \le (\frac{1}{3}+\frac{8}{3g^2})|V(G)|$.
	In 2008, L\"{o}wenstein and Rautenbach~\cite{Lowenstein+2008} proved that a cubic graph with girth at least $g \ge 6$ satisfies $\gamma(G) \le (\frac{44}{135}+\frac{82}{135g})|V(G)|$, which proves that the Reed's conjecture is true for graphs with girth at least $83$.
	
	Very recently, further progress towards solving Reed's conjecture for graphs of higher girth has been made. Dorbec and Henning~\cite{dorbec-henning-1/3} have proved that if $G$ is a cubic graph on $n$ vertices of girth at least $6$ with no $7$-cycle and no $8$-cycle, then $\gamma(G) \leq \frac13 n$. Additionally, they also prove that if $G$ is a cubic bipartite graph on $n$ vertices with no $4$-cycle and no $8$-cycle, then $\gamma(G) \leq \frac13 n$. However, we are not aware of progress made specifically for planar graphs or graphs containing $8$-cycles.
	
	In this paper, we made progress on bounding the domination number of subcubic planar graphs with girth at least $8$ by proving the following theorem.
	
	\begin{theorem}
		\label{thm:main}
		If $G$ is a subcubic planar graph with girth at least $8$, then
		\[ 20 \dom{G} < 20 n_0(G) + 15 n_1(G) + 11 n_2(G) + 7 n_3(G). \]
	\end{theorem}
	
	Note that since the domination number is an integer, $\gamma(G)$ can be bounded from above by the largest integer \emph{strictly smaller} than 
	\[ \frac{1}{20}\left(20 n_0(G) + 15 n_1(G) + 11 n_2(G) + 7 n_3(G)\right) = n_0(G) + \frac{3}{4} n_1(G) + \frac{11}{20} n_2(G) + \frac{7}{20} n_3(G). \]
	We present here two examples of subcubic planar graphs of girth at least $8$ whose domination number is equal to the largest integer strictly smaller than the bound of Theorem~\ref{thm:main}.
	Let $C_8^\star$ and $C_9^\star$ be graphs obtained by attaching a leaf to each vertex of an $8$-cycle and a $9$-cycle, respectively. Clearly, $\gamma(C_8^\star) = 8$, and $n_0(C_8^\star) + \frac{3}{4}n_1(C_8^\star) + \frac{11}{20}n_2(C_8^\star) + \frac{7}{20}n_3(C_8^\star)  =  \frac{44}{5} = 8.8$. Similarly, we can see that $\gamma(C_9^\star) = 9$, and Theorem~\ref{thm:main} gives an upper bound of $\frac{99}{10} = 9.9$.
	
	The rest of the paper is organized as follows. The notation needed is listed in Section~\ref{sec:notation}, and the general framework of the proof of Theorem~\ref{thm:main} is explained in Section~\ref{sec:framework}. For clarity, the proof itself is given in two parts, first we prove Theorem~\ref{thm:main} for graphs with girth at least 9 in Section~\ref{sec:girth9}, and only then we proceed with the more complicated proof for girth at least 8 in Section~\ref{sec:girth8}. In both cases, the proof consists of two parts, listing reducible configurations and finishing with a discharging argument. 
	
	In Section~\ref{sec:weights}, we present an alternative bound for the domination number of subcubic planar graphs, and in Section~\ref{sec:further}, we conclude the paper by discussing possibilities for an upper bound with a different girth condition.
	
	\section{Notation}
	\label{sec:notation}
	
	Throughout this paper, let $G$ be a finite, simple, undirected graph with vertex set $V(G)$ and edge set $E(G)$.
	For $v \in V(G)$, we use $\deg_G(v)$ and $N_G(v)$ to denote the {\em degree} and the {\em neighborhood}, respectively, of $v$ in $G$.
	Let $N_G[v] = N_G(v) \cup \{v\}$ be the {\em closed neighborhood} of $v$ in $G$.
	For $S = \{v_1, \ldots, v_n\}$, we will use $N_G[v_1, \ldots, v_n]$ instead of $N_G[\{v_1, \ldots, v_n\}]$ for simplicity.
	Let $\delta(G)$ be the {\em minimum degree} of $G$, which is the minimum degree among all the vertices in $G$.
	Let $g(G)$ be the {\em girth} of $G$, which is the length of the shortest cycle of $G$.
	When $G$ is a plane graph, we use $F(G)$ to denote the set of faces of $G$, and $\ell(f)$ to denote the length of $f \in F(G)$.
	
	For a positive integer $k$, we use $[k]$ to denote the set $\{1,\ldots, k\}$.
	A vertex of degree $k$ is called a \emph{$k$-vertex}, and a vertex of degree at least $k$ is called a \emph{$k^+$-vertex}. A face of length $k$ is called a \emph{$k$-face}, and a face of length at least $k$ is called a \emph{$k^+$-face}. Let $n_i(G)$ be the number of $i$-vertices in $G$. A \emph{facial walk} of a $k$-face $f$ is a walk in $G$ of length $k$ along the boundary of the face $f$. Note that a facial walk of $f$ visits each vertex incident with $f$ several times.
	For a positive integer $k$, a {\em path} $P_k$ in $G$ with $k$ vertices is a sequence of distinct vertices $v_1, \ldots, v_k$ such that $v_i v_{i+1} \in E(G)$ for $i \in [k-1]$. 
	
	For a graph $G$ and $S \subseteq V(G)$, let $G[S]$ be the graph induced by $S$.
	Also, let $G-S = G[V(G) \setminus S]$.
	
	Let $G$ be a subcubic graph. 
	A {\em sequence} $(i_1, \ldots, i_k)$ in $G$ is a path $v_1, \ldots, v_k$ in $G$ such that $\deg_G(v_j) = i_j$ for all $j \in [k]$.
	Additionally, a {\em sequence} $(i_1, \ldots, i_{k-1}, 3^q_p)$ in $G$ is a sequence $(i_1, \ldots, i_{k-1}, 3)$ formed by $v_1, \ldots, v_k$, where $N_G(v_k) = \{v_{k-1},x,y\}$, $\deg_G(x) = p$, and $\deg_G(y) = q$. 
	For an example, see Figure \ref{fig:seq-ex}.
	
	\begin{figure}[h!!]
		\centering
		\begin{tikzpicture}
			\begin{scope}[shift={(0,0)}]
				\filldraw[fill=black, draw=black] (-2,0) circle (0.05);
				\filldraw[fill=black, draw=black] (-1,0) circle (0.05);
				\filldraw[fill=black, draw=black] (0,0) circle (0.05);
				\filldraw[fill=black, draw=black] (1,0) circle (0.05);
				\filldraw[fill=black, draw=black] (2,0) circle (0.05);
				\draw (-2.5,0) -- (2,0);
				\draw (-1,0) -- (-1,-0.5);
				\draw (1,0) -- (1,-0.5);
				\node at (0,1) {$(2,3,2,3,1)$};
			\end{scope}
			
			\begin{scope}[shift={(7,0)}]
				\filldraw[fill=black, draw=black] (-2,0) circle (0.05);    
				\filldraw[fill=black, draw=black] (-1,0) circle (0.05);
				\filldraw[fill=black, draw=black] (0,0) circle (0.05);
				\filldraw[fill=black, draw=black] (-30:1) circle (0.05);
				\filldraw[fill=black, draw=black] (30:1) circle (0.05);
				\draw (-2.5,0) -- (0,0);
				\draw (0,0) -- (30:1.5);
				\draw (0,0) -- (-30:1);
				\begin{scope}[shift={(-30:1)}]
					\draw (0:0) -- (30:0.5);
					\draw (0:0) -- (-30:0.5);
				\end{scope}
				\node at (-0.5,1){$(2,2,3_3^2)$};
			\end{scope}
		\end{tikzpicture}
		\caption{An example of the sequence $(2,3,2,3,1)$ and an example of the sequence $(2,2,3_3^2)$.}
		\label{fig:seq-ex}
	\end{figure}
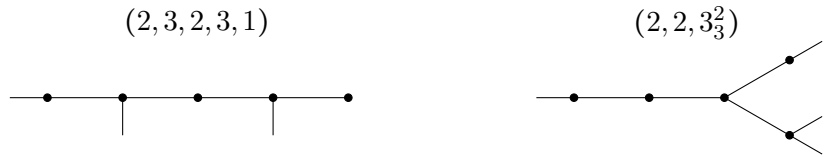
	
	\section{General framework}
	\label{sec:framework}
	
	The main idea of the proof of Theorem \ref{thm:main} is to first find several reducible configurations, and then use the discharging method. The working frame is developed in the rest of this section, the reducible configurations are presented in Sections \ref{sec:reducible9} and \ref{sec:reducible8}, and the discharging process is presented in Sections \ref{sec:discharging9} and \ref{sec:discharging8}.
	
	For each $v \in V(G)$, let $w_G(v)$ be the {\it weight} of $v$ on $G$, defined as follows:
	\[
	w_G(v) = 
	\begin{cases}
		20 &\text{if } \deg_G(v) = 0, \\
		15 &\text{if } \deg_G(v) = 1, \\
		11 &\text{if } \deg_G(v) = 2, \\
		7 &\text{if } \deg_G(v) = 3.
	\end{cases}
	\]
	
	For $\emptyset \neq S \subseteq V(G)$ and a subgraph $H$ of $G$, let $w_G(S) = \sum_{v \in S} w_G(v)$, $w_G(H) = w_G(V(H))$, and $c_G(S) = w_{G-S}(G-S)-w_G(G-S)$.
	Note that $c_G(S) \ge 0$ and $w_G(G) = 20n_0(G) + 15 n_1(G)+11n_2(G) + 7n_3(G)$.
	
	From now on, let $G$ be a minimal counterexample (with respect to the number of vertices) to the Theorem \ref{thm:main}. Thus, clearly $20 \gamma(G) \geq w_G(G)$, and $20 \gamma(H) < w_H(H)$ for every proper subgraph $H$ of $G$.
	
	\begin{lemma}[Key Lemma]
		\label{lem:key}
		For $\emptyset \neq S \subseteq V(G)$, $20 \gamma(G[S]) > w_G(S) - c_G(S)$.    
	\end{lemma}
	
	\begin{proof}
		By the minimality of $G$, we have $20\gamma(G-S) < w_{G-S}(G-S)$.
		Note also the following three facts:
		\begin{enumerate}
			\item $20\gamma(G) \geq w_G(G)$,
			\item $\gamma(G) \le \gamma(G-S) + \gamma(G[S])$,
			\item $c_G(S) = w_{G-S}(G-S) - w_G(G-S) = w_{G-S}(G-S) - (w_G(G) - w_G(S))$.
		\end{enumerate}
		
		These give us
		\begin{align*}
			w_{G-S}(G-S)-c_G(S)+w_G(S)=w_G(G) & \leq 20 \gamma(G) \\
			& \le 20\gamma(G-S) + 20 \gamma(G[S]) \\
			& < w_{G-S}(G-S) + 20\gamma(G[S]), 
		\end{align*}
		from which the desired inequality follows.
	\end{proof}
	
	The \nameref{lem:key} is our main tool when proving reducible configurations. The idea is to find a suitable set $S$, determine the $\gamma(G[S])$ (trivial since $S$ will be small), calculate the weight $w_G(S)$ of $S$ (by definition, we simply add up the respective weights that are in one-to-one correspondence with the degree of the vertices in $S$), and find a bound on the cost $c_G(S)$ of $S$. The cost represents how much we can afford to change the weight of the graph by removing $S$. Each vertex $v$ in $N_G(S) \setminus S$ will contribute to the cost of $S$, and this contribution is determined by observing the degree change of the vertex $v$. By definition, it equals $w_{G-S}(v)-w_G(v)$. Thus, it can be calculated from $\deg_G(v)$ and $|N_G(v) \cap S|$; see Table~\ref{tab:costs}. Observe that the change in weight is the largest when a 0-vertex is created. Thus, we also need to keep track of the change in degree for vertices in $G-S$. If for each $v \in N_G(S) \setminus S$, $|N_G(v) \cap S| = 1$ holds, then the cost of $S$ depends on the number of $1$-vertices and $2^+$-vertices in $N_G(S) \setminus S$. Note that the condition that girth is at least 8 will often imply that for each $v \in N_G(S) \setminus S$, $|N_G(v) \cap S| = 1$ for small sets $S$.
	
	\begin{table}[!ht]
		\centering
		\begin{tabular}{|c||*{4}{c|}}\hline
			\backslashbox{$\deg_G(v)$}{$|N_G(v) \cap S|$}
			&0&1&2&3\\\hline\hline
			0 & 0 & / & / & / \\\hline
			1 & 0 & 5 & / & / \\\hline
			2 & 0 & 4 & 9 & / \\\hline
			3 & 0 & 4 & 8 & 13 \\\hline
		\end{tabular}
		\caption{The amount a vertex $v$ contributes to $c_G(S)$.}
		\label{tab:costs}
	\end{table}
	
	Sometimes, we will bound the cost of $S$ by considering the edges leaving $S$ instead of considering the vertices in $N_G(S) \setminus S$. The possibility of several edges being incident to the same vertex in $N_G(S) \setminus S$ must be taken into account. This is formalized in the following lemma.
	
	\begin{lemma}
		\label{lem:cost-bound}
		Let $\emptyset \neq S \subseteq V(G)$ be such that there are $k$ edges between $S$ and $G-S$.
		If all vertices in $N_G(S) \setminus S$ are of degree 2 or 3, then $c_G(S) \leq 4k + \lfloor \frac{k}{2} \rfloor$.
		More precisely, if there there are at most $\ell$ edges between $S$ and $G-S$ that are incident with a $2$-vertex in $N_G(S) \setminus S$, then $c_G(S) \leq 4k + \max \left\{ \lfloor \frac{\ell}{2} \rfloor + \lfloor \frac{k-\ell}{3} \rfloor, \lfloor \frac{\ell-1}{2} \rfloor + \lfloor \frac{k-\ell+1}{3} \rfloor \right\}$.
	\end{lemma}
	
	\begin{proof}
		Let $x$ be the number of edges between $S$ and $G-S$ that are incident with a 2-vertex in $N_G(S) \setminus S$. Then $c_G(S) \leq 4 k + \lfloor \frac{x}{2} \rfloor + \lfloor \frac{k-x}{3} \rfloor$. 
		Note that each edge between $S$ and $G-S$ basically increases the cost of $N_G(S) \setminus S$ by $4$, and whenever there is an isolated vertex in $N_G(S) \setminus S$, there is an additional increase in cost by $1$.
		Observing that $\max \left\{ \lfloor \frac{x}{2} \rfloor + \lfloor \frac{k-x}{3} \rfloor; \; 0 \leq x \leq \ell \right\} = \max \left\{ \lfloor \frac{\ell}{2} \rfloor + \lfloor \frac{k-\ell}{3} \rfloor, \lfloor \frac{\ell-1}{2} \rfloor + \lfloor \frac{k-\ell+1}{3} \rfloor \right\} \leq \lfloor \frac{k}{2} \rfloor$ concludes the proof.
	\end{proof}

	Note that the more precise version of the lemma is not needed in the following, and we have included it for completeness only.
	
	\section{Graphs with girth at least 9}
	\label{sec:girth9}
	
	To make the proof more readable, we first present the proof of Theorem \ref{thm:main} only for graphs with girth at least 9. Reducible configurations are listed in Section \ref{sec:reducible9}, and the discharging method is presented in Section \ref{sec:discharging9}. The complete proof of the result (including graphs containing 8-faces) is presented in Section \ref{sec:girth8}.
	
	\subsection{Reducible configurations}
	\label{sec:reducible9}
	
	As above, $G$ is a minimal counterexample to the Theorem \ref{thm:main}. From now on, assume that $G$ is embedded in the plane and this embedding is fixed. From the minimality of $G$ we immediately infer that $G$ is connected.
	
	\begin{lemma}
		\label{lem:1-2}
		The graph $G$ does not contain a sequence $(1,2)$.
	\end{lemma}
	
	\begin{proof}
		Let $v, u$ be a sequence $(1,2)$ in $G$; see Figure \ref{fig:1-2}.
		
		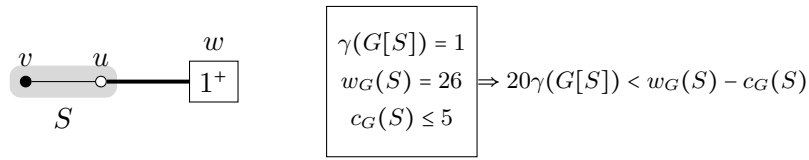
\begin{figure}[h!!]
			\centering
			\begin{tikzpicture}
				[
				vert/.style={circle,fill=black,draw=black, inner sep=0.05cm},
				s/.style={fill=black!15!white, draw=black!15!white, rounded corners},
				outvert/.style={rectangle,draw=black},
				outedge/.style={line width=1.5pt},
				dom_vert/.style={circle,draw=black,fill=white, inner sep=0.05cm}
				]
				\node[dom_vert, label={90:$u$}] (U) at (0,0) {};
				\node[vert, label={90:$v$}] (V) at (-1,0) {};
				\node[outvert, label={90:$w$}] (out) at (1.5,0) {$1^+$};
				\draw (V) -- (U);
				\draw[outedge] (U) -- (out);
				\begin{pgfonlayer}{background}
					\node[s, label={-90:$S$}, fit=(U) (V)] {};
				\end{pgfonlayer}
				\node at (4,0.5) {\footnotesize{$\gamma(G[S]) =1$}};
				\node at (4,0) {\footnotesize{$w_G(S)=26$}};
				\node at (4,-0.5) {\footnotesize{$c_G(S) \le 5$}};
				\node at (7.2,0) {\footnotesize{$\Rightarrow 20 \gamma(G[S]) < w_G(S)-c_G(S)$}};
				\draw (3,1) -- (5,1) -- (5,-1) -- (3,-1) --cycle;
			\end{tikzpicture}
			\caption{The reducible configuration $(1,2)$ from Lemma \ref{lem:1-2}.}
			\label{fig:1-2}
		\end{figure}
		
		Let $S = \{v, u\} \subseteq N_G[u]$. Clearly, $\gamma(G[S]) = 1$ and $w_G(S) = 11 + 15 = 26$. 
		Denote $N_G(u) = \{v, w\}$. Then $c_G(S) = w_{G-S}(G-S) - w_G(G-S) = \sum_{x \in V(G-S)} w_{G-S}(x) - \sum_{x \in V(G-S)} w_G(x) = w_{G-S} (w) - w_G(w) \in \{4, 5\}$, so $c_G(S) \leq 5$. (Note that since $|N_G(S) \setminus S|=1$, $c_G(S) = 5$ if and only if $\deg_G (w) = 1$; see Table \ref{tab:costs}.)
		
		Thus, by the \nameref{lem:key}, we obtain $20 = 20 \gamma(G[S]) > w_G(S) - c_G(S) \geq 26 - 5 = 21$, which is a contradiction.
	\end{proof}
	
	Before we continue to establish further reducible configurations, we consider Figure \ref{fig:1-2} again, since the same style will be used for the rest of the paper. Vertices forming set $S$ are marked with a grey shade, and the vertices dominating $G[S]$ are colored white (the remaining vertices of $S$ are given color black). The edges between $S$ and $G-S$ are drawn thicker. The vertices in $N_G(S) \setminus S$ are drawn as square boxes. The value $k^+$ (resp.\ $k$) inside the box means that this vertex has degree at least $k$ (resp.\ exactly $k$). Sometimes, even a degree of a vertex in $S$ is not unique -- in this case the edge to a possibly non-existing neighbor will be marked with a dashed line (see for example Figure \ref{fig:2-3-2}).
	
	\begin{lemma}
		\label{lem:1-3-1}
		The graph $G$ does not contain a sequence $(1,3,1)$.
	\end{lemma}
	
	\begin{proof}
		Let $v_1, u, v_2$ be a sequence $(1,3,1)$ in $G$; see Figure \ref{fig:1-3-1}.
		
		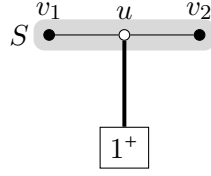
\begin{figure}[h!!]
			\centering
			\begin{tikzpicture}
				[
				vert/.style={circle,fill=black,draw=black, inner sep=0.05cm},
				s/.style={fill=black!15!white, draw=black!15!white, rounded corners},
				outvert/.style={rectangle,draw=black},
				outedge/.style={line width=1.5pt},
				dom_vert/.style={circle,draw=black,fill=white, inner sep=0.05cm}
				] 
				\filldraw[s] (-1.2,-0.2) rectangle (1.2, 0.2) {};    
				\node[dom_vert, label={90:$u$}] (U) at (0,0) {};
				\node[vert, label={90:$v_1$}] (V1) at (-1,0) {};
				\node[vert, label={90:$v_2$}] (V2) at (1,0) {};
				\node[outvert] (out) at (0,-1.5) {$1^+$};
				\draw (V1) -- (U) -- (V2);
				\draw[outedge] (U) -- (out);
				\node at (-1.4,0){$S$};
			\end{tikzpicture}
			\caption{The reducible configuration $(1,3,1)$ from Lemma \ref{lem:1-3-1}.}
			\label{fig:1-3-1}
		\end{figure}    
		
		Let $S= \{ u, v_1, v_2 \} \subseteq N_G[u]$. Clearly, $\gamma(G[S]) = 1$, $w_G(S) = 2 \cdot 15 + 7 = 37$, and $c_G(S) \leq 5$. By the \nameref{lem:key}, we get $20 > 37 - 5 = 32$, which is a contradiction.
	\end{proof}
	
	\begin{lemma}
		\label{lem:1-3-2}
		The graph $G$ does not contain a sequence $(1,3,2)$.
	\end{lemma}
	
	\begin{proof}
		Let $v, u, w$ be a sequence $(1, 3, 2)$ in $G$; see Figure~\ref{fig:1-3-2}.
		
		\begin{figure}[h!!]
			\centering
			\begin{tikzpicture}
				[
				vert/.style={circle, fill=black, draw=black, inner sep=0.05cm}, 
				s/.style={fill=black!15!white, draw=black!15!white, rounded corners},
				outvert/.style={rectangle, draw=black},
				outedge/.style={line width=1.5pt},
				dom_vert/.style={circle,draw=black,fill=white, inner sep=0.05cm}
				]
				\filldraw[s] (-0.2,-0.2) rectangle (2.2,0.2) {};
				
				\node[vert, label={90:$v$}] (V) at (0,0) {};
				\node[dom_vert, label={90:$u$}] (U) at (1,0) {};
				\node[vert, label={90:$w$}] (W) at (2,0) {};
				\node[outvert] (W1) at (3.5,0) {$1^+$};
				\node[outvert] (U1) at (1,-1.5) {$1^+$};
				\draw (V) -- (U) -- (W);
				\draw[outedge] (W) -- (W1);
				\draw[outedge] (U) -- (U1);
				
				\node at (-0.4,0) {$S$};
			\end{tikzpicture}
			
			\caption{The reducible configuration $(1,3,2)$ from Lemma \ref{lem:1-3-2}.}
			\label{fig:1-3-2}
		\end{figure}
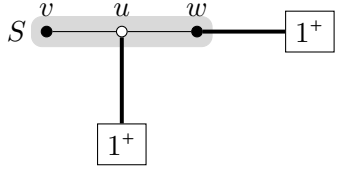

		Take $S = \{ v, u, w \} \subseteq N_G[u]$. We have $\gamma(G[S]) = 1$, $w_G(S) = 15 + 7 + 11 = 33$, and $c_G(S) \leq 2 \cdot 5 = 10$. By the \nameref{lem:key}, we obtain $20 > 33 - 10 = 23$, which is a contradiction.
	\end{proof}
	
	\begin{lemma}
		\label{lem:1-3-3-1}
		The graph $G$ does not contain a sequence $(1,3,3,1)$.
	\end{lemma}
	
	\begin{proof}
		Let $u_1, u, v, v_1$ be a sequence $(1,3,3,1)$ in $G$. Since $u, v$ are of degree $3$, they each have another neighbor, $u_2, v_2$, respectively. By Lemmas \ref{lem:1-3-1} and \ref{lem:1-3-2}, we know that $\deg_G(u_2) = \deg_G(v_2) = 3$; see Figure~\ref{fig:1-3-3-1}.
		
		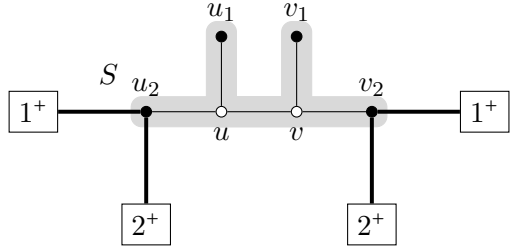
\begin{figure}[h!!]
			\centering
			\begin{tikzpicture}
				[
				vert/.style={circle,fill=black,draw=black, inner sep=0.05cm},
				s/.style={fill=black!15!white, draw=black!15!white, rounded corners},
				outvert/.style={rectangle,draw=black},
				outedge/.style={line width=1.5pt},
				dom_vert/.style={circle,draw=black,fill=white, inner sep=0.05cm}
				] 
				\filldraw[s] (-1.2,-0.2) rectangle (1.2,0.2) {};
				\filldraw[s, rotate=90] (-0.2,-0.2) rectangle (1.2,0.2){};
				\filldraw[s] (-0.2,-0.2) rectangle (2.2,0.2) {};
				\filldraw[s, shift={(1,0)}, rotate=90] (-0.2,-0.2) rectangle (1.2,0.2){};

				\node[vert, label={90:$u_2$}] (U2) at (-1,0) {};
				\node[dom_vert, label={-90:$u$}] (U) at (0,0) {};
				\node[dom_vert, label={-90:$v$}] (V) at (1,0) {};
				\node[vert, label={90:$v_2$}] (V2) at (2,0) {};
				\node[vert, label={90:$u_1$}] (U1) at (0,1) {};
				\node[vert, label={90:$v_1$}] (V1) at (1,1) {};
				\node[outvert] (U21) at (-2.5,0) {$1^+$};
				\node[outvert] (U22) at (-1,-1.5) {$2^+$};
				\node[outvert] (V21) at (3.5,0) {$1^+$};
				\node[outvert] (V22) at (2,-1.5) {$2^+$};
				\node at (-1.5,0.5) {$S$};
				\draw (U2) -- (U) -- (V)-- (V2);
				\draw (U) -- (U1);
				\draw (V) -- (V1);
				\draw[outedge] (U2)-- (U21);
				\draw[outedge] (U2) -- (U22);
				\draw[outedge] (V2) -- (V21);
				\draw[outedge] (V2) -- (V22);
			\end{tikzpicture}
			\caption{The reducible configuration $(1,3,3,1)$ from Lemma \ref{lem:1-3-3-1}.}
			\label{fig:1-3-3-1}
		\end{figure}
		
		Take $S = \{ u, u_1, u_2, v, v_1, v_2 \} = N_G[u,v]$. We have $\gamma(G[S]) = 2$ and $w_G(S) = 2 \cdot 15 + 4 \cdot 7 = 58$. 
		Since $G$ has girth at least $8$, $N_G(S) \setminus S$ has exactly $4$ vertices.
		By Lemma \ref{lem:1-3-1}, at most two of the vertices in $N_G(S) \setminus S$ are of degree $1$.
		Thus, $c_G(S) \leq 2 \cdot 5 + 2 \cdot 4 = 18$. By the \nameref{lem:key}, we obtain $40 = 2 \cdot 20 > 58 - 18 = 40$, which is a contradiction.
	\end{proof}
	
	\begin{lemma}
		\label{lem:2-2-2}
		The graph $G$ does not contain a sequence $(2,2,2)$.
	\end{lemma}
	
	\begin{proof}
		Let $u_1, u_2, u_3$ be a sequence $(2,2,2)$ in $G$; see Figure \ref{fig:2-2-2}.
		
		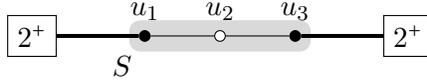
\begin{figure}[h!!]
			\centering
			\begin{tikzpicture}
				[
				vert/.style={circle,fill=black,draw=black, inner sep=0.05cm},
				s/.style={fill=black!15!white, draw=black!15!white, rounded corners},
				outvert/.style={rectangle,draw=black},
				outedge/.style={line width=1.5pt},
				dom_vert/.style={circle,draw=black,fill=white, inner sep=0.05cm}
				] 
				\filldraw[s] (-1.2,-0.2) rectangle (1.2,0.2) {};
				
				\node[vert, label={90:$u_1$}] (U1) at (-1,0) {};
				\node[dom_vert, label={90:$u_2$}] (U2) at (0,0) {};
				\node[vert, label={90:$u_3$}] (U3) at (1,0) {};
				\node[outvert] (U11) at (-2.5,0) {$2^+$};
				\node[outvert] (U31) at (2.5,0) {$2^+$};
				
				\node at (-1.3,-0.4) {$S$};
				\draw (U1) -- (U2) -- (U3);
				\draw[outedge] (U1)-- (U11);
				\draw[outedge] (U3) -- (U31);
			\end{tikzpicture}
			\caption{The reducible configuration $(2,2,2)$ from Lemma \ref{lem:2-2-2}.}
			\label{fig:2-2-2}
		\end{figure}
		
		Take $S = \{ u_1, u_2, u_3 \} = N_G[u_2]$. We have $\gamma(G[S]) = 1$ and $w_G(S) = 3 \cdot 11 = 33$. 
		Since $G$ has girth at least $8$, no vertex in $N_G(S) \setminus S$ has at least two neighbors in $S$, and we get $c_G(S) \leq 2 \cdot 5 = 10$.
		By the \nameref{lem:key}, we obtain $20 > 33-10 = 23$, which is a contradiction.
	\end{proof}
	
	\begin{lemma}
		\label{lem:no-deg-1}
		The graph $G$ has no 1-vertex.
	\end{lemma}
	
	\begin{proof}
		Suppose that $G$ contains a 1-vertex $v$ with a neighbor $u$. By Lemma \ref{lem:1-2}, $\deg_G(u) = 3$. 
		Let $w_1$ and $w_2$ be the other neighbors of $u$; see Figure~\ref{fig:1-333}.
		By Lemmas \ref{lem:1-3-1} and \ref{lem:1-3-2}, both $w_1$ and $w_2$ must be of degree $3$.         
		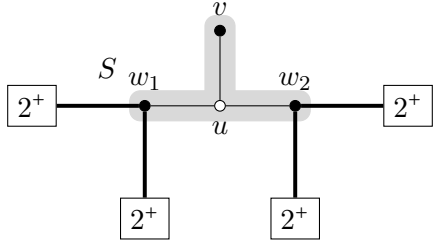
\begin{figure}[h!!]
			\centering
			\begin{tikzpicture}
				[
				vert/.style={circle,fill=black,draw=black, inner sep=0.05cm},
				s/.style={fill=black!15!white, draw=black!15!white, rounded corners},
				outvert/.style={rectangle,draw=black},
				outedge/.style={line width=1.5pt},
				dom_vert/.style={circle,draw=black,fill=white, inner sep=0.05cm}
				] 
				\filldraw[s] (-1.2,-0.2) rectangle (1.2,0.2) {};
				\filldraw[s, rotate=90] (-0.2,-0.2) rectangle (1.2,0.2){};
				
				\node[vert, label={90:$w_1$}] (w1) at (-1,0) {};
				\node[dom_vert, label={-90:$u$}] (u) at (0,0) {};
				\node[vert, label={90:$w_2$}] (w2) at (1,0) {};
				\node[vert, label={90:$v$}] (v) at (0,1) {};
				\node[outvert] (w11) at (-2.5,0) {$2^+$};
				\node[outvert] (w12) at (-1,-1.5) {$2^+$};
				\node[outvert] (w21) at (2.5,0) {$2^+$};
				\node[outvert] (w22) at (1,-1.5) {$2^+$};
				\node at (-1.5,0.5) {$S$};
				\draw (w1) -- (u) -- (w2);
				\draw (u) -- (v);
				\draw[outedge] (w1)-- (w11);
				\draw[outedge] (w1) -- (w12);
				\draw[outedge] (w2) -- (w21);
				\draw[outedge] (w2) -- (w22);
			\end{tikzpicture}
			\caption{The reducible configuration from Lemma \ref{lem:no-deg-1} that finalizes the argument that $G$ has no 1-vertex.}
			\label{fig:1-333}
		\end{figure}
		
		Take $S = \{ v, u, w_1, w_2 \} = N_G[u]$. We have $\gamma(G[S]) = 1$ and $w_G(S) = 15  + 3 \cdot 7 = 36$. Since $G$ has girth at least $8$, no vertex in $N_G(S) \setminus S$ has at least two neighbors in $S$. 
		By Lemma \ref{lem:1-3-3-1}, none of the vertices in $N_G(S) \setminus S$ are of degree $1$.
		Thus $c_G(S) \leq 4 \cdot 4 = 16$. By the \nameref{lem:key}, we obtain $20 > 36 - 16 =20$, which is a contradiction.
	\end{proof}
	
	\begin{lemma}
		\label{lem:2-3-2}
		The graph $G$ contains no sequence $(2,3,2)$.
	\end{lemma}
	
	\begin{proof}
		Let $v_1, v_2, v_3$ be a sequence $(2,3,2)$ in $G$; see Figure \ref{fig:2-3-2}. 
		Let $N_G(v_2) = \{v_1, w_2, v_3\}$.    
		
		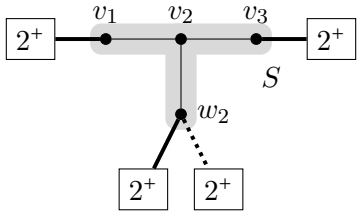
\begin{figure}[h!!]
			\centering
			\begin{tikzpicture}
				[
				vert/.style={circle,fill=black,draw=black, inner sep=0.05cm},
				s/.style={fill=black!15!white, draw = white, rounded corners},
				outvert/.style={rectangle,draw=black},
				outedge/.style={line width=1.5pt}
				] 
				\filldraw[fill=black!15!white, draw=black!15!white, rounded corners] (-1.2,-0.2) rectangle (1.2,0.2) {};
				\filldraw[fill=black!15!white, draw=black!15!white, rounded corners] (-0.2,-1.2) rectangle (0.2,0.2){};
				
				\node[vert, label={90:$v_1$}] (V1) at (-1,0) {};
				\node[vert, label={90:$v_2$}] (V2) at (0,0) {};
				\node[vert, label={90:$v_3$}] (V3) at (1,0) {};
				\node[vert, label={0:$w_2$}] (W2) at (0,-1) {};
				
				\node[outvert] (V11) at (-2,0) {$2^+$};
				\node[outvert] (V33) at (2,0) {$2^+$};
				\node[outvert] (W22) at (-0.5,-2) {$2^+$};
				\node[outvert] (W23) at (0.5,-2) {$2^+$};
				\node at (1.2,-0.5) {$S$};
				\draw (V1) -- (V2) -- (V3);
				\draw (V2) -- (W2);
				\draw[outedge] (V11)-- (V1);
				\draw[outedge] (V33)-- (V3);
				\draw[outedge, dotted] (W2)--(W23);
				\draw[outedge] (W2)--(W22);
			\end{tikzpicture}
			\caption{The reducible configuration $(2,3,2)$ from Lemma \ref{lem:2-3-2}.} 
			\label{fig:2-3-2}
		\end{figure}
		
		Take $S = \{ v_1, v_2, v_3, w_2 \} = N_G[v_2]$. We have $\gamma(G[S]) = 1$ and $w_G(S) \geq 2 \cdot 11  + 2 \cdot 7 = 36$. Since $G$ has girth at least $8$, no vertex in $N_G(S) \setminus S$ has at least two neighbors in $S$.
		Also, by Lemma~\ref{lem:no-deg-1}, all the vertices in $N_G(S) \setminus S$ have degree at least $2$, and we get $c_G(S) \leq 4 \cdot 4 = 16$. By the \nameref{lem:key}, we obtain $20 > w_G(S) - c_G(S) \geq 36 - 16 = 20$, which is a contradiction.
	\end{proof}
	
	\begin{lemma}
		\label{lem:6-path}
		Each $P_6$ in $G$ contains at most three $2$-vertices.
	\end{lemma}
	
	\begin{proof}
		By Lemma \ref{lem:2-2-2}, $G$ cannot contain a $P_6$ with three consecutive $2$-vertices. Thus $G$ can contain a $P_6$ with at most four $2$-vertices.
		
		Suppose that $G$ contains a $P_6$ of vertices $v_1, \ldots, v_6$, where exactly four of them are of degree $2$, and the remaining two are of degree $3$. 
		Applying Lemma~\ref{lem:2-3-2} to a short case-analysis, we obtain that the only possibility is that $\deg_G(v_1) = \deg_G(v_2) = \deg_G(v_5) = \deg_G(v_6) = 2$ and $\deg_G(v_3) = \deg_G(v_4)= 3$; see Figure \ref{fig:2-2-3-2-3-2}. 
		
		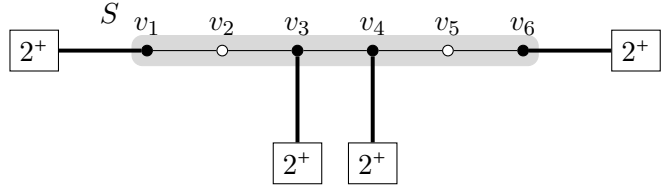
\begin{figure}[h!!]
			\centering
			\begin{tikzpicture}
				[
				vert/.style={circle,fill=black,draw=black, inner sep=0.05cm},
				s/.style={fill=black!15!white, draw=black!15!white, rounded corners},
				outvert/.style={rectangle,draw=black},
				outedge/.style={line width=1.5pt},
				dom_vert/.style={circle,draw=black,fill=white, inner sep=0.05cm}
				] 
				\filldraw[s] (-1.2,-0.2) rectangle (4.2,0.2) {};
				
				\node[vert, label={90:$v_1$}] (v1) at (-1,0) {};
				\node[dom_vert, label={90:$v_2$}] (v2) at (0,0) {};
				\node[vert, label={90:$v_3$}] (v3) at (1,0) {};
				\node[vert, label={90:$v_4$}] (v4) at (2,0) {};
				\node[dom_vert, label={90:$v_5$}] (v5) at (3,0) {};
				\node[vert, label={90:$v_6$}] (v6) at (4,0) {};
				\node[outvert] (v11) at (-2.5,0) {$2^+$};
				\node[outvert] (v31) at (1,-1.5) {$2^+$};
				\node[outvert] (v41) at (2,-1.5) {$2^+$};
				\node[outvert] (v61) at (5.5,0) {$2^+$};
				\node at (-1.5,0.5) {$S$};
				\draw (v1) -- (v2) -- (v3) -- (v4) -- (v5) -- (v6);
				\draw[outedge] (v1) -- (v11);
				\draw[outedge] (v3) -- (v31);
				\draw[outedge] (v4) -- (v41);
				\draw[outedge] (v6) -- (v61);
			\end{tikzpicture}
			\caption{An example of a reducible configuration $(2,2,3,3,2,2)$ from Lemma \ref{lem:6-path}.}
			\label{fig:2-2-3-2-3-2}
		\end{figure}
		
		Take $S = \{ v_1,\ldots, v_6 \} = N[v_2, v_5]$. 
		We have $\gamma(G[S]) = 2$ and $w_G(S) = 4 \cdot 11  + 2 \cdot 7 = 58$.
		Since $G$ has girth at least $8$, no vertex in $N_G(S) \setminus S$ has at least two neighbors in $S$.
		Also, by Lemma~\ref{lem:no-deg-1}, all the vertices in $N_G(S) \setminus S$ have degree at least $2$, and we get $c_G(S) \leq 4 \cdot 4 = 16$. By the \nameref{lem:key}, we obtain $40 > 58-16 = 42$, which is a contradiction.
	\end{proof}
	
	For the following, observe that since the girth of $G$ is at least 8, every consecutive $k$ vertices, $1 \leq k \leq 6$, along a facial walk in $G$ form a $P_k$.
	
	\begin{lemma}
		\label{lem:short-path}
		For $i \in [5]$, each face in $G$ contains a $P_i$ with at most $\lfloor\frac{i}{2}\rfloor$ $2$-vertices.
	\end{lemma}
	
	\begin{proof}
		Let $f$ be a face in $G$ of length $k$.
		Since $G$ has girth at least $8$, $k \geq 8$.
		By Lemma~\ref{lem:2-2-2}, $f$ contains a $P_1$ with zero 2-vertices, and a $P_2$ with at most one $2$-vertex. 
		By Lemma \ref{lem:2-3-2}, $f$ contains a $P_3$ with at most one 2-vertex. By Lemmas \ref{lem:2-2-2} and \ref{lem:2-3-2}, $f$ contains a $P_4$ with at most two 2-vertices. If $f$ contains no two consecutive 2-vertices, then $f$ has a $P_5$ with at most two 2-vertices. Otherwise, $f$ contains at least two consecutive 2-vertices, say $v_3$, $v_4$. Let $v_1, \ldots, v_6$ be a $P_6$ on $f$. By Lemmas \ref{lem:2-2-2} and \ref{lem:2-3-2}, $\deg_G(v_j) = 3$ for all $j \in \{1,2,5,6\}$. Thus $v_1, \ldots, v_5$ is a $P_5$ on $f$ with at most two 2-vertices.
	\end{proof}
	
	\begin{lemma}
		\label{lem:11+face}
		For $k \ge 11$, in the facial walk of each $k$-face in $G$ there are at most $\lfloor\frac{k}{2}\rfloor$ $2$-vertices.
	\end{lemma}
	
	\begin{proof}
		Let $f$ be a face in $G$ of length $k \ge 11$.
		Let $k \equiv i \pmod{6}$, $i \in [6]$. By Lemma \ref{lem:6-path} or \ref{lem:short-path}, the face $f$ contains a $P_i$ with at most $\lfloor \frac{i}{2} \rfloor$ 2-vertices. The remaining vertices on the facial walk can be covered by $\frac{k-i}{6}$ $P_6$s, and by Lemma \ref{lem:6-path}, each of them contains at most three 2-vertices. Altogether, the number of 2-vertices appearing on the facial walk of $f$ is at most $\lfloor \frac{i}{2} \rfloor + \frac{k-i}{6} \cdot 3 \in \{\frac{k-1}{2}, \frac{k}{2}\}$. Thus, the facial walk of $f$ contains at most $\lfloor \frac{k}{2} \rfloor$ 2-vertices.
	\end{proof}
	
	Note that if $G$ is 2-connected, we can replace ``facial walk of the $k$-face $f$'' above with simply ``$k$-face $f$''. But if $G$ is not 2-connected, it is possible that some vertices are incident with only the face $f$, and thus these vertices appear on the facial walk of $f$ twice.
	
	\begin{lemma}
		\label{lem:10-face}
		Every $10$-face in $G$ is incident with at most four $2$-vertices.
	\end{lemma}
	
	\begin{proof}
		Let $f$ be a $10$-face $v_1, \ldots, v_{10}$ in $G$. If $f$ contains no two consecutive $2$-vertices, then by Lemma \ref{lem:2-3-2}, $f$ contains at most three $2$-vertices. For example, see Figure \ref{fig:10-face-2}.
		
		\begin{figure}[h!!]
			\centering
			
			\begin{tikzpicture}
				[scale=0.7,
				vert/.style={circle,fill=black,draw=black, inner sep=0.05cm},
				s/.style={fill=black!15!white, draw=black!15!white, rounded corners},
				outvert/.style={rectangle,draw=black},
				outedge/.style={line width=1.5pt},
				dom_vert/.style={circle,draw=black,fill=white, inner sep=0.05cm}
				] 
				\pgfmathtruncatemacro{\N}{10}
				\pgfmathtruncatemacro{\R}{2}
				\pgfmathtruncatemacro{\D}{3}
				
				\begin{scope}
					\foreach \x in {1,...,\N}
					\node[vert] (\x) at (18 + \x*360/\N:\R cm) {};
					\foreach \x [remember=\x as \lastx (initially 1)] in {1,...,\N,1}
					\path (\x) edge (\lastx);
					
					\foreach \y in {1,3,4,6,7,8,10}
					\draw (\y) -- (18 + \y*360/\N:\D cm);
					\node[fill=white] at (0,0) {$10$-face};
					
				\end{scope}
				
			\end{tikzpicture}
			\caption{An example of a $10$-face with no two consecutive $2$-vertices.}
			\label{fig:10-face-2}
		\end{figure}
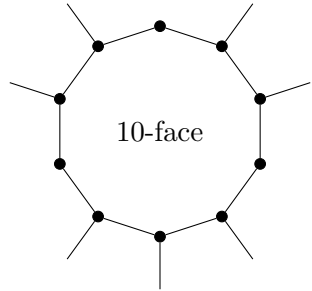
		
		Let $f$ contain two consecutive $2$-vertices; without loss of generality, let $\deg_G(v_1) = \deg_G(v_2) = 2$. By Lemmas \ref{lem:2-2-2} and \ref{lem:2-3-2}, vertices $v_3, v_4, v_{10}, v_9$ are all of degree $3$. By Lemma \ref{lem:6-path}, at most one of $v_5, v_6$ is a $2$-vertex, and at most one of $v_8, v_7$ is a $2$-vertex. Thus, $f$ contains at most four $2$-vertices. For example, see Figure \ref{fig:10-face-22}.
	\end{proof}
	
	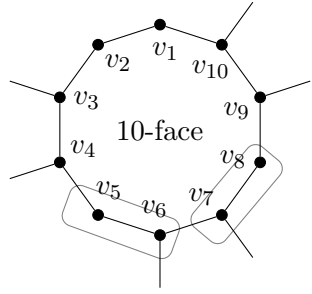
\begin{figure}[h!!]
		\centering
		
		\begin{tikzpicture}
			[scale=0.7,
			vert/.style={circle,fill=black,draw=black, inner sep=0.05cm},
			s/.style={fill=black!15!white, draw=black!15!white, rounded corners},
			s2/.style={fill=white, draw=black!50!white, rounded corners},
			outvert/.style={rectangle,draw=black},
			outedge/.style={line width=1.5pt},
			dom_vert/.style={circle,draw=black,fill=white, inner sep=0.05cm}
			] 
			\pgfmathtruncatemacro{\N}{10}
			\pgfmathtruncatemacro{\R}{2}
			\pgfmathtruncatemacro{\D}{3}
			
			\begin{scope}
				\filldraw[s2, rotate=160] (-1,1.5) rectangle (1.2,2.3){};
				\filldraw[s2, rotate=230] (-1,1.5) rectangle (1,2.3){};
				
				\foreach \x in {1,...,\N}
				\node[vert] (\x) at (18 + \x*360/\N:\R cm) {};
				\foreach \x [remember=\x as \lastx (initially 1)] in {1,...,\N,1}
				\path (\x) edge (\lastx);
				
				\foreach \y in {4,5,7,8,10,1}
				\draw (\y) -- (18 + \y*360/\N:\D cm);
				
				\draw \foreach \x in {1,...,10} {(50+\x*360/\N:1.5) node {$v_{\x}$}};
				
				\node[fill=white] at (0,0) {$10$-face};
				
			\end{scope}
		\end{tikzpicture}
		
		\caption{An example of a $10$-face with two consecutive $2$-vertices.} 
		\label{fig:10-face-22}
	\end{figure}
	
	Note that since the girth of $G$ is at least 8 and there are no 1-vertices, the facial walk of a 10-face is always a cycle.
	
	\begin{lemma}
		\label{lem:9-face}
		Every $9$-face in $G$ is incident with at most three $2$-vertices.
	\end{lemma}
	
	\begin{proof}
		Let $f$ be a $9$-face $v_1, \ldots, v_9$ in $G$. If $f$ has no two consecutive $2$-vertices, then by Lemma \ref{lem:2-3-2}, $f$ contains at most three $2$-vertices (if $\deg_G(v_1) = 2$, then $\deg_G(v_2) = \deg_G(v_3) = \deg_G(v_9) = \deg_G(v_8) = 3$, and since there are no two consecutive $2$-vertices on $f$, at most two of the $v_4, v_5, v_6, v_7$ can be 2-vertices). For example, see Figure \ref{fig:9-face-2}.
		
		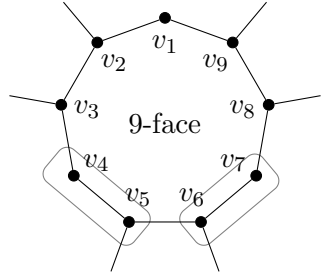
\begin{figure}[h!!]
			\centering
			\begin{tikzpicture}
				[scale=0.7,
				vert/.style={circle,fill=black,draw=black, inner sep=0.05cm},
				s/.style={fill=black!15!white, draw=black!15!white, rounded corners},
				s2/.style={fill=white, draw=black!50!white, rounded corners},
				outvert/.style={rectangle,draw=black},
				outedge/.style={line width=1.5pt},
				dom_vert/.style={circle,draw=black,fill=white, inner sep=0.05cm}
				] 
				\pgfmathtruncatemacro{\N}{9}
				\pgfmathtruncatemacro{\R}{2}
				\pgfmathtruncatemacro{\D}{3}
				
				\begin{scope}
					\filldraw[s2, rotate=140] (-1,1.5) rectangle (1.2,2.3){};
					\filldraw[s2, rotate=220] (-1,1.5) rectangle (1,2.3){};
					
					\foreach \x in {1,...,\N}
					\node[vert] (\x) at (50 + \x*360/\N:\R cm) {};
					\foreach \x [remember=\x as \lastx (initially 1)] in {1,...,\N,1}
					\path (\x) edge (\lastx);
					\draw \foreach \x in {1,...,\N} {(50+\x*360/\N:1.5) node {$v_{\x}$}};
					
					\foreach \y in {2,3,5,6,8,9}
					\draw (\y) -- (50 + \y*360/\N:\D cm);
					\node[fill=white] at (0,0) {$9$-face};
				\end{scope}
			\end{tikzpicture}
			
			\caption{An example of a 9-face with no two consecutive 2-vertices.}
			\label{fig:9-face-2}
		\end{figure}
		
		Suppose that $f$ contains two consecutive $2$-vertices; without loss of generality, let $\deg_G(v_1) = \deg_G(v_2) = 2$. 
		Then by Lemmas~\ref{lem:2-2-2} and~\ref{lem:2-3-2}, $\deg_G(v_3) = \deg_G(v_4) = \deg_G(v_9) = \deg_G(v_8) = 3$. 
		By Lemma \ref{lem:6-path}, at most one of $v_5, v_6$ can be a 2-vertex and at most one of $v_6, v_7$ can be a $2$-vertex. 
		But since $v_5$ and $v_7$ cannot both be $2$-vertices by Lemma \ref{lem:2-3-2}, at most one of $v_5, v_6, v_7$ is a 2-vertex. This implies that $f$ again contains at most three $2$-vertices. For example, see Figure \ref{fig:9-face-22}.
	\end{proof}
	
	\begin{figure}[h!!]
		\centering
		\begin{tikzpicture}
			[scale=0.7,
			vert/.style={circle,fill=black,draw=black, inner sep=0.05cm},
			s/.style={fill=black!15!white, draw=black!15!white, rounded corners},
			s2/.style={fill=white, draw=black!50!white, rounded corners},
			outvert/.style={rectangle,draw=black},
			outedge/.style={line width=1.5pt},
			dom_vert/.style={circle,draw=black,fill=white, inner sep=0.05cm}
			] 
			\pgfmathtruncatemacro{\N}{9}
			\pgfmathtruncatemacro{\R}{2}
			\pgfmathtruncatemacro{\D}{3}
			
			\begin{scope}
				\filldraw[s2, rotate=200] (-2,1.2) rectangle (2,2.5){};
				
				\foreach \x in {1,...,\N}
				\node[vert] (\x) at (50 + \x*360/\N:\R cm) {};
				\foreach \x [remember=\x as \lastx (initially 1)] in {1,...,\N,1}
				\path (\x) edge (\lastx);
				
				\draw \foreach \x in {1,...,\N} {(50+\x*360/\N:1.5) node {$v_{\x}$}};
				
				\foreach \y in {3,4,5,6,8,9}
				\draw (\y) -- (50 + \y*360/\N:\D cm);

				\node[fill=white] at (0,0) {$9$-face};
			\end{scope}
		\end{tikzpicture}
		
		\caption{An example of a 9-face with two consecutive 2-vertices.}\label{fig:9-face-22}
	\end{figure}
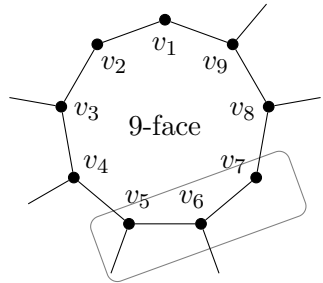
	
	\subsection{Discharging}
	\label{sec:discharging9}
	
	For $x \in V(G) \cup F(G)$, we define the following initial charge:
	$$\mu(x) = \begin{cases}
		2 \deg_G(x) - 6 & \text{if } x \in V(G),\\
		\ell(x) - 6 & \text{if } x \in F(G).
	\end{cases}$$
	
	Since $G$ is planar, Euler's formula implies that the initial charge of $G$ is $$\mu(G) = \sum_{x \in V(G) \cup F(G)} \mu(x) = -12.$$ 
	
	We use the following discharging rule: 
	\begin{itemize}
		\item Every face $f$ sends charge 1 to each of the 2-vertices on the facial walk of $f$, counted with multiplicities.
	\end{itemize}
	Let $\mu^*(x)$ denote the charge of vertices and faces after applying the discharging rule. 
	
	\begin{lemma}
		\label{lem:charge-positive-1}
		For every $x \in V(G) \cup F(G)$, $\mu^*(x) \geq 0$.
	\end{lemma}
	
	\begin{proof}
		By Lemma~\ref{lem:no-deg-1}, $G$ has no vertex of degree $1$.
		A 2-vertex that is incident with two different faces receives charge 1 from each of them. A 2-vertex that is incident with only one face $f$ appears on the facial walk of $f$ twice, so it receives charge 2 from $f$.
		So $\mu^*(v) = (2 \cdot 2 -6) + 2 \cdot 1 = 0$. If $v$ is a vertex of degree 3, then $\mu^*(v) = \mu(v) = 0$. 
		
		If $f$ is a $k$-face, $k \geq 11$, then by Lemma \ref{lem:11+face}, the facial walk of $f$ contains at most $\left \lfloor \frac{k}{2} \right \rfloor$ 2-vertices. Thus, $$\mu^*(f) \geq (k-6) - \left \lfloor \frac{k}{2} \right \rfloor = \left \lceil \frac{k}{2} \right \rceil - 6 \geq 0$$ since $k \geq 11$.
		
		If $f$ is a $10$-face, then by Lemma \ref{lem:10-face}, $f$ is incident with at most four $2$-vertices. Hence, $\mu^*(f) \geq (10-6) - 4 = 0$. If $f$ is a $9$-face, then by Lemma \ref{lem:9-face}, $f$ is incident with at most three $2$-vertices. Hence, $\mu^*(f) \geq (9-6) - 3 = 0$.
	\end{proof}
	
	Clearly it holds that $$-12 = \mu(G) = \mu^*(G) \geq 0,$$ which is a contradiction. This concludes the proof of Theorem~\ref{thm:main} for graphs with girth at least $9$.
	
	\section{Graphs with girth at least 8}
	\label{sec:girth8}
	
	In this section, we give the full proof of Theorem \ref{thm:main}, so proving the desired bound for all subcubic planar graphs with girth at least $8$. The structure of the proof is the same as for graphs with girth at least $9$, but additional care is needed to deal with $8$-faces. The discharging method used in the last step of the proof is also similar as in the case of graphs with girth at least $9$, but we include it in full for completeness.
	
	\subsection{Reducible configurations}
	\label{sec:reducible8}
	
	A minimal counterexample $G$ in this case satisfies all reducible configurations presented in Section \ref{sec:reducible9}. Thus, Lemmas \ref{lem:1-2}-\ref{lem:9-face} still hold. In the rest of this section, we present additional reducible configurations with respect to 8-faces.
	
	\begin{definition}
		A face $f$ in $G$ is a \emph{bad $8$-face} if it is an 8-face $v_1, \ldots, v_8$ with vertices of degrees $(2$, $2$, $3$, $3$, $2$, $3$, $3$, $3)$ (see Figure \ref{fig:bad-8-face}). If the 8-face is not bad, we say it is a \emph{good $8$-face}. The two consecutive $3$-vertices on a bad 8-face $f$ that both have a neighboring $2$-vertex on $f$ (vertices $v_3$ and $v_4$ in the above notation) are called \emph{contributing vertices} of $f$. A face adjacent to the bad $8$-face $f$ that contains the two contributing vertices of $f$ is called a \emph{contributing face} of $f$. 
	\end{definition}
	
	Recall that a vertex $v$ is incident with a face $f$ (or, $v$ belongs to $f$) if $v$ is one of the vertices on $f$. 
	For a contributing face $f$, let the \emph{set of bad 8-faces incident with $f$}, $\mathcal{C}(f)$, be the set of all bad $8$-faces whose two contributing vertices are on $f$, i.e.\ $g \in \mathcal{C}(f)$ means that $g$ is a bad $8$-face, and the two contributing vertices of $g$ belong to $f$. Shortly, we will say that a bad 8-face $g$ is incident with its contributing face $f$.
	
	\begin{figure}[!ht]
		\centering
		\begin{tikzpicture}
			[scale=0.7,
			vert/.style={circle,fill=black,draw=black, inner sep=0.05cm},
			s/.style={fill=black!15!white, draw=black!15!white, rounded corners},
			s2/.style={fill=white, draw=black!50!white, rounded corners},
			outvert/.style={rectangle,draw=black},
			outedge/.style={line width=1.5pt},
			dom_vert/.style={circle,draw=black,fill=white, inner sep=0.05cm}
			] 
			\pgfmathtruncatemacro{\N}{8}
			\pgfmathtruncatemacro{\R}{2}
			\pgfmathtruncatemacro{\D}{3}
			
			\clip (-7.5,-4.3) rectangle (3,3); % to remove the white space above and below
			
			\begin{scope}[>=Stealth]
				\foreach \x in {1,...,\N}
				\node[vert] (\x) at (23 + \x*360/\N:\R cm) {};
				
				\filldraw[pattern=dots] (23+360/8:2) -- (23+2*360/8:2) -- (23+ 3*360/8:2) -- (23+4*360/8:2) -- (23+5*360/8:2) -- (23+6*360/8:2) -- (23+7*360/8:2) -- (23+8*360/8:2) -- (23+360/8:2);
				
				\draw \foreach \x in {1,...,\N}{(23+\x*360/\N:1.5) node {$v_{\x}$}};
				
				\foreach \y in {6,7,8}
				\draw (\y) -- (23 + \y*360/\N:\D cm);
				
				\draw    (3) to[out=135,in=-135,distance=10cm] (4);
				
				\node (l1) at (0,0) {\small{bad 8-face $f$}};
				\node (l2) at (-4.6,0) {\small{contributing face of $f$}};
				\node (l4) at (-4,-4) {\small{contributing vertices of $f$}};
				
				\draw[->] (l4) -- (3);
				\draw[->] (l4) -- (4);
				
			\end{scope}
		\end{tikzpicture}
		\caption{A bad $8$-face $f$. The dotted background in the figures in this section is used only to help the reader recognize the face of the main focus.}
		\label{fig:bad-8-face}
	\end{figure}
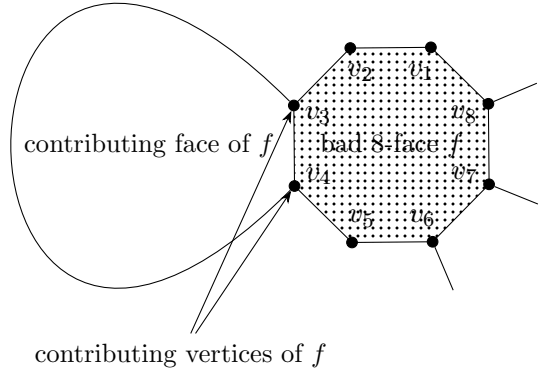
	
	\begin{lemma}
		\label{lem:8-face-weak}
		Every $8$-face in $G$ is either incident with at most two $2$-vertices or it is a bad $8$-face.
	\end{lemma}
	
	\begin{proof}
		Let $f$ be an $8$-face $v_1, \ldots, v_8$ in $G$. If $f$ has no two consecutive $2$-vertices, then by Lemma \ref{lem:2-3-2}, $f$ contains at most two $2$-vertices (if $\deg_G(v_1) = 2$, then $\deg_G(v_2) = \deg_G(v_3) = \deg_G(v_8) = \deg_G(v_7) = 3$, and since there are no two consecutive $2$-vertices on $f$, at most one of $v_4, v_5, v_6$ can be a $2$-vertex). 
		
		Suppose that $f$ contains two consecutive $2$-vertices; without loss of generality, $\deg_G(v_1) = \deg_G(v_2) = 2$. Then by Lemma \ref{lem:2-3-2}, $\deg_G(v_3) = \deg_G(v_4) = \deg_G(v_8) = \deg_G(v_7) = 3$. By Lemma \ref{lem:6-path}, at most one of $v_5, v_6$ can be a $2$-vertex. 
		Thus, $f$ contains at most three $2$-vertices. 
		If $f$ contains exactly three $2$-vertices, then by the above argument, $f$ is a bad $8$-face.
	\end{proof}
	
	Note that from now on, even though the girth of the graph is at least 8, it may happen that for some $v \in N_G(S) \setminus S$, $|N_G(v) \cap S| \geq 2$. Thus planarity of $G$ and $g(G) \geq 8$ will both be used to limit the number of common neighbors of vertices from $S$.
	
	\begin{lemma}
		\label{lem:good-3333}
		Both contributing vertices of a bad 8-face are adjacent to a 3-vertex on a contributing face. Additionally, one of the two contributing vertices of a bad $8$-face is adjacent to a 3-vertex whose neighbors are all 3-vertices.
	\end{lemma}
	
	\begin{proof}
		Let $v_1, \ldots, v_8$ be a bad 8-face on $G$ with $\deg_G(v_i) = 2$ for $i \in \{1,2,5\}$, so $v_3$ and $v_4$ are its contributing vertices. Let $w_3$, $w_4$ be the other neighbors of $v_3$, $v_4$, respectively.
		By Lemma \ref{lem:2-3-2}, $\deg_G(w_3) = \deg_G(w_4) =3$.
		Suppose that both $w_3$ and $w_4$ have a neighbor of degree $2$, say $x_3$ and $x_4$, respectively. See Figure~\ref{fig:good-3333} for example.
		
		\begin{figure}[h!!]
			\centering
			\begin{tikzpicture}
				[
				vert/.style={circle,fill=black,draw=black, inner sep=0.05cm},
				s/.style={fill=black!15!white, draw=black!15!white, rounded corners},
				outvert/.style={rectangle,draw=black},
				outedge/.style={line width=1.5pt},
				dom_vert/.style={circle,draw=black,fill=white, inner sep=0.05cm}
				] 
				
				\draw[pattern=dots] (0,1) rectangle (3,0);
				
				\node at (1.5,0.5) {bad $8$-face};
				
				\filldraw[s] (0.8,-0.2) rectangle (3.2,0.2) {};
				\filldraw[s] (0.8,0.8) rectangle (3.2,1.2) {};
				
				\filldraw[s] (-2.2,-0.2) rectangle (0.2,0.2){};
				\filldraw[s] (-2.2,0.8) rectangle (0.2,1.2){};
				
				\filldraw[s] (-1.2,-1.2) rectangle (-0.8,0.2) {};
				\filldraw[s] (-1.2,0.8) rectangle (-0.8,2.2) {};
				
				\node[vert, label={-90:$v_4$}] (v4) at (0,0) {};
				\node[vert, label={-90:$v_5$}] (v5) at (1,0) {};
				\node[dom_vert, label={-45:$v_6$}] (v6) at (2,0) {};
				\node[vert, label={-90:$v_7$}] (v7) at (3,0) {};
				\node[vert, label={90:$v_3$}] (v3) at (0,1) {};
				\node[vert, label={90:$v_2$}] (v2) at (1,1) {};
				\node[dom_vert, label={90:$v_1$}] (v1) at (2,1) {};
				\node[vert, label={90:$v_8$}] (v8) at (3,1) {};
				\node[dom_vert, label={45:$w_3$}] (w3) at (-1,1) {};
				\node[dom_vert, label={-45:$w_4$}] (w4) at (-1,0) {};
				\node[vert, label={90:$x_3$}] (x3) at (-2,1) {};
				\node[vert] (x3') at (-1,2) {};
				\node[vert, label={-90:$x_4$}] (x4) at (-2,0) {};
				\node[vert] (x4') at (-1,-1) {};
				
				\node[outvert] (v81) at (4.5,1) {$3$};
				\node[outvert] (v71) at (4.5,0) {$2^+$};
				\node[outvert] (v61) at (2,-1.5) {$3$};
				\node[outvert] (x31) at (-3.5,1) {$2^+$};
				\node[outvert] (x41) at (-3.5,0) {$2^+$};
				\node[outvert] (x3'1) at (-1.8,3.2) {$2^+$};
				\node[outvert] (x3'2) at (-0.2,3.2) {$3$};
				\node[outvert] (x4'1) at (-1.8,-2.2) {$2^+$};
				\node[outvert] (x4'2) at (-0.2,-2.2) {$3$};
				
				\node at (-0.5,0.5) {$S$};
				\draw (v4) -- (v5) -- (v6) -- (v7) -- (v8) -- (v1) -- (v2) -- (v3) -- (v4);
				\draw (v4) -- (w4) -- (x4);
				\draw (w4) -- (x4');
				\draw (v3) -- (w3) -- (x3);
				\draw (w3) -- (x3');
				
				\draw[outedge] (v8)-- (v81);
				\draw[outedge] (v7)-- (v71);
				\draw[outedge] (v6)-- (v61);
				\draw[outedge] (x3)-- (x31);
				\draw[outedge] (x4) -- (x41);
				\draw[outedge] (x3'1) -- (x3') -- (x3'2);
				\draw[outedge] (x4'1) -- (x4') -- (x4'2);
				
			\end{tikzpicture}
			
			\caption{The configuration from Lemma \ref{lem:good-3333}.}
			\label{fig:good-3333}
		\end{figure}
		
		Let $S = N_G[v_1, w_3, w_4] \cup \{v_5, v_6, v_7\}  \subseteq N_G[v_1, v_6, w_3, w_4]$.
		There are at most nine edges between $S$ and $G - S$. We have $\gamma(G[S]) = 4$, $w_G(S) \geq 5 \cdot 11  + 9 \cdot 7 = 118$, and $c_G(S) \leq 7 \cdot 4 + 9 = 37$ (by planarity and the girth condition, at most one pair of vertices in $S$ can have a common neighbor in $N_G(S) \setminus S$; for example, $x_3$ and $v_7$, or $x_4$ and $v_8$, can have a common neighbor in $N_G(S) \setminus S$). By the \nameref{lem:key}, we obtain $80 > 118-37 = 81$, which is a contradiction.
	\end{proof}
	
	\begin{lemma}
		\label{lem:adj-bad-faces}
		No two adjacent bad $8$-faces of $G$ share a contributing vertex.
	\end{lemma}
	
	\begin{proof}
		Let $v_1, \ldots, v_8$ be a bad $8$-face $f$ on $G$ with $\deg_G(v_i) = 2$ for $i \in \{1,2,5\}$. Recall that by Lemma~\ref{lem:good-3333}, two bad $8$-faces cannot have both $v_3$ and $v_4$ in common.
		
		First suppose that the vertex $v_3$ belongs to another bad $8$-face $f'$ of $G$. 
		Let its vertices be $v_1$, $v_2$, $v_3$, $w_4$, $w_5$, $w_6$, $w_7$, $v_8$.
		See Figure~\ref{fig:adj-bad-faces-1} for example.
		Since $f'$ is a bad $8$-face and $v_3$ is one of its contributing vertices, $w_4$ is also a contributing vertex, and thus $\deg_G(w_5)=2$. Let $z$ be the other neighbor of $w_6$. By Lemma \ref{lem:2-3-2}, $\deg_G(z) = 3$. Let the other neighbor of $v_6$ be denoted by $y$. By Lemma \ref{lem:2-3-2}, $\deg_G(y) = 3$. 
		
		\begin{figure}[h!!]
			\centering
			\begin{tikzpicture}
				[
				vert/.style={circle,fill=black,draw=black, inner sep=0.05cm},
				s/.style={fill=black!15!white, draw=black!15!white, rounded corners},
				outvert/.style={rectangle,draw=black},
				outedge/.style={line width=1.5pt},
				dom_vert/.style={circle,draw=black,fill=white, inner sep=0.05cm}
				] 
				
				\draw[pattern=dots] (0,1) rectangle (3,0);
				
				\node at (1.5,0.4) {\small{bad $8$-face $f$}};
				\node[fill=white] at (1.5,1.4) {\small{bad $8$-face $f'$}};
				
				\filldraw[s] (0.8,-0.2) rectangle (3.2,0.2) {};
				\filldraw[s] (1.8,-1.2) rectangle (2.2,0.2){};
				
				\filldraw[s] (0.8,0.8) rectangle (3.2,1.2) {};
				
				\filldraw[s] (0.8,1.8) rectangle (3.2,2.2) {};
				\filldraw[s] (1.8,1.8) rectangle (2.2,3.2){};
				
				\filldraw[s] (-0.2,-0.2) rectangle (0.2,2.2){};

				\node[vert, label={-90:$v_4$}] (v4) at (0,0) {};
				\node[vert, label={-90:$v_5$}] (v5) at (1,0) {};
				\node[dom_vert, label={-45:$v_6$}] (v6) at (2,0) {};
				\node[vert, label={-90:$v_7$}] (v7) at (3,0) {};
				\node[dom_vert, label={180:$v_3$}] (v3) at (0,1) {};
				\node[vert, label={-90:$v_2$}] (v2) at (1,1) {};
				\node[dom_vert, label={-90:$v_1$}] (v1) at (2,1) {};
				\node[vert, label={0:$v_8$}] (v8) at (3,1) {};
				\node[vert, label={90:$w_4$}] (w4) at (0,2) {};
				\node[vert, label={45:$w_5$}] (w5) at (1,2) {};
				\node[dom_vert, label={-90:$w_6$}] (w6) at (2,2) {};
				\node[vert, label={90:$w_7$}] (w7) at (3,2) {};
				\node[vert, label={0:$z$}] (z) at (2,3) {};
				\node[vert, label={0:$y$}] (y) at (2,-1) {};
				
				\node[outvert] (v41) at (-1.5,0) {$3$};
				\node[outvert] (w41) at (-1.5,2) {$3$};
				\node[outvert] (v71) at (4.5,0) {$2^+$};
				\node[outvert] (w71) at (4.5,2) {$2^+$};
				\node[outvert] (y1) at (1.2,-2.2) {$3$};
				\node[outvert] (y2) at (2.8,-2.2) {$2^+$};
				\node[outvert] (z1) at (1.2,4.2) {$3$};
				\node[outvert] (z2) at (2.8,4.2) {$2^+$};
				
				\node at (-0.5,0.5) {$S$};
				\draw (v4) -- (v5) -- (v6) -- (v7) -- (v8) -- (v1) -- (v2) -- (v3) -- (v4);
				\draw (v6) -- (y);
				\draw (v3) -- (w4) -- (w5) -- (w6) -- (w7) -- (v8);
				\draw (w6) -- (z);
				\draw[outedge] (v4)-- (v41);
				\draw[outedge] (v7)-- (v71);
				\draw[outedge] (w4)-- (w41);
				\draw[outedge] (w7)-- (w71);
				\draw[outedge] (y)-- (y1);
				\draw[outedge] (y)-- (y2);
				\draw[outedge] (z)-- (z1);
				\draw[outedge] (z)-- (z2);
			\end{tikzpicture}
			\caption{Two bad $8$-faces have vertex $v_3$ in common.}
			\label{fig:adj-bad-faces-1}
		\end{figure}
		
		Take $S = \{ v_1,\ldots, v_8, w_4, \ldots, w_7, y, z \} = N_G[v_1,v_3,v_6,w_6]$. By above, there are exactly eight edges between $S$ and $G - S$. We have $\gamma(G[S]) = 4$, $w_G(S) = 4 \cdot 11  + 10 \cdot 7 = 114$, and $c_G(S) \leq 6 \cdot 4 + 9 = 33$ (since $y$ and $z$ could have a common neighbor). By the \nameref{lem:key}, we obtain $80 > 114-33 = 81$, which is a contradiction.
		
		Next suppose that the vertex $v_4$ belongs to another bad 8-face $f''$ of $G$. 
		Let its vertices be $w_1$, $w_2$, $w_3$, $v_4$, $v_5$, $v_6$, $w_7$, $w_8$. 
		See Figure~\ref{fig:adj-bad-faces-2} for example.
		Since $f''$ is a bad 8-face, $w_1$ and $w_2$ are 2-vertices, and $v_4$ and $w_3$ are contributing vertices of $f''$.
		
		\begin{figure}[h!!]
			\centering
			
			\begin{tikzpicture}
				[
				vert/.style={circle,fill=black,draw=black, inner sep=0.05cm},
				s/.style={fill=black!15!white, draw=black!15!white, rounded corners},
				outvert/.style={rectangle,draw=black},
				outedge/.style={line width=1.5pt},
				dom_vert/.style={circle,draw=black,fill=white, inner sep=0.05cm}
				] 
				
				\draw[pattern=dots] (0,1) rectangle (3,0);
				\node at (1.5,0.5) {\small{bad $8$-face $f$}};
				\node at (1,-0.8) {\small{bad}}; 
				\node at (1,-1.3) {\small{$8$-face $f''$}};
				
				\filldraw[s] (-0.2,-0.2) rectangle (3.2,0.2) {};
				\filldraw[s] (1.8,-1.2) rectangle (2.2,0.2){};
				
				\filldraw[s] (0.8,0.8) rectangle (3.2,1.2) {};
				
				\filldraw[s] (-0.2,-1.2) rectangle (0.2,1.2){};
				
				\filldraw[s] (-0.2,-2.2) rectangle (2.2,-1.8) {};
				
				\node[dom_vert, label={180:$v_4$}] (v4) at (0,0) {};
				\node[vert, label={-90:$v_5$}] (v5) at (1,0) {};
				\node[dom_vert, label={-45:$v_6$}] (v6) at (2,0) {};
				\node[vert, label={-90:$v_7$}] (v7) at (3,0) {};
				\node[vert, label={90:$v_3$}] (v3) at (0,1) {};
				\node[vert, label={90:$v_2$}] (v2) at (1,1) {};
				\node[dom_vert, label={90:$v_1$}] (v1) at (2,1) {};
				\node[vert, label={90:$v_8$}] (v8) at (3,1) {};
				\node[vert, label={-135:$w_3$}] (w3) at (0,-1) {};
				\node[vert, label={-90:$w_2$}] (w2) at (0,-2) {};
				\node[dom_vert, label={-90:$w_1$}] (w1) at (1,-2) {};
				\node[vert, label={-45:$w_7$}] (w7) at (2,-1) {};
				\node[vert, label={-90:$w_8$}] (w8) at (2,-2) {};

				\node[outvert] (v31) at (-1.5,1) {$3$};
				\node[outvert] (w31) at (-1.5,-1) {$3$};
				\node[outvert] (v71) at (4.5,0) {$2^+$};
				\node[outvert] (v81) at (4.5,1) {$3$};
				\node[outvert] (w71) at (3.5,-1) {$2^+$};
				\node[outvert] (w81) at (3.5,-2) {$2^+$};
				
				\node at (-0.5,0.5) {$S$};
				\draw (v4) -- (v5) -- (v6) -- (v7) -- (v8) -- (v1) -- (v2) -- (v3) -- (v4);
				\draw (v4) -- (w3) -- (w2) -- (w1) -- (w8) -- (w7) -- (v6);
				
				\draw[outedge] (v3)-- (v31);
				\draw[outedge] (v7)-- (v71);
				\draw[outedge] (w3)-- (w31);
				\draw[outedge] (w7)-- (w71);
				\draw[outedge] (w8)-- (w81);
				\draw[outedge] (v8)-- (v81);
				
			\end{tikzpicture}
			
			\caption{Two bad 8-faces have vertex $v_4$ in common.}
			\label{fig:adj-bad-faces-2}
		\end{figure}
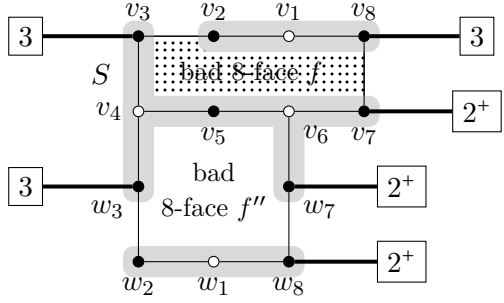
		
		Take $S = \{ v_1,\ldots, v_8, w_1, w_2, w_3, w_7, w_8 \} = N_G[v_1,v_4,v_6,w_1]$. Clearly there are exactly six edges between $S$ and $G - S$. We have $\gamma(G[S]) = 4$, $w_G(S) = 5 \cdot 11  + 8 \cdot 7 = 111$, and $c_G(S) \leq 6 \cdot 4 = 24$. By the \nameref{lem:key}, we obtain $80 > 111-24 = 87$, which is a contradiction.
	\end{proof}

	\begin{lemma}
		\label{lem:good-11+}
		Let $k \geq 11$ and let $f$ be a $k$-face. If a facial walk of $f$ contains $x$ 2-vertices and $f$ is incident with $y$ bad 8-faces, then $x+y \leq \left \lfloor \frac{k}{2} \right \rfloor$.
	\end{lemma}
	
	\begin{proof}
		Observe that vertices incident with $f$ are 2-vertices, contributing 3-vertices, or non-contributing $3$-vertices. Let $A$ be the multiset of 2-vertices on $f$, $B$ be the multiset of contributing 3-vertices on $f$, and $C$ be the multiset of the remaining 3-vertices of $f$. Note that $A$, $B$ and $C$ are multisets since a vertex is included multiple times if it appears more than once in the facial walk of $f$. Clearly, $|A| + |B| + |C| = k$. 
		
		Since each bad 8-face incident with $f$ has two contributing 3-vertices in common with $f$, $y \le \frac{|B|}{2}$. Clearly, $x = |A|$. By definition of a contributing 3-vertex and by Lemmas~\ref{lem:2-3-2} and~\ref{lem:adj-bad-faces}, the neighbors of contributing 3-vertices on $f$ are 3-vertices that are non-contributing. 
		Thus by Lemmas~\ref{lem:2-2-2} and~\ref{lem:2-3-2} again, there is an injection from $A$ to $C$, which maps each $v \in A$ to $w \in C \cap N_G(v)$, so $|A| \leq |C|$. 
		
		Hence, $$x + y \le |A| + \frac{|B|}{2} \leq \frac{|A| + |B| + |C|}{2} = \frac{k}{2},$$
		thus also $x + y \leq \left \lfloor \frac{k}{2} \right \rfloor$.
	\end{proof}
	
	Note that this bound holds for all $k$-faces, $k \geq 8$, but the obtained bound is not strong enough for our later use in the discharging process. Thus we consider cases $k \in \{8,9,10\}$ separately.
	
	\begin{lemma}
		\label{lem:good-8}
		If a contributing $8$-face $f$ is incident with $x$ $2$-vertices and $y$ bad $8$-faces, then $x+y \leq 2$.
	\end{lemma}
	
	\begin{proof}
		Since $f$ is a contributing $8$-face, $x \leq 2$ by Lemma \ref{lem:8-face-weak}. Suppose that $x + y \geq 3$. We distinguish between the following cases.
		
		\begin{description}
			\item[Case 1.] $x = 2$ and $y \geq 1$.\\
			Let the vertices of $f$ be $v_1, \ldots, v_8$. Since $ y \geq 1$, we may assume that $v_1$ and $v_2$ are the contributing vertices that $f$ shares with a bad $8$-face $f'$. Let the vertices of $f'$ be $v_1, v_2, u_3, \ldots, u_8$, where $u_3, u_4, u_8$ are $2$-vertices; see Figure~\ref{fig:good-8-x2-y1}. 
			
			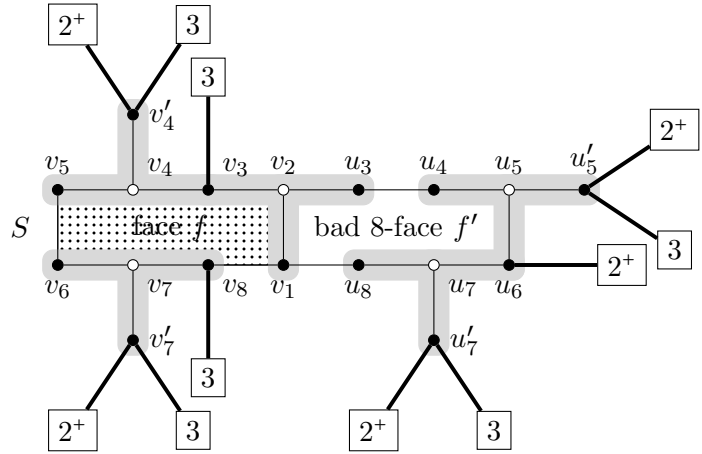
\begin{figure}[h!!]
				\centering
				
				\begin{tikzpicture}
					[
					vert/.style={circle,fill=black,draw=black, inner sep=0.05cm},
					s/.style={fill=black!15!white, draw=black!15!white, rounded corners},
					outvert/.style={rectangle,draw=black},
					outedge/.style={line width=1.5pt},
					dom_vert/.style={circle,draw=black,fill=white, inner sep=0.05cm}
					] 
					
					\draw[pattern=dots] (0,0) rectangle (-3,1) {};
					\filldraw[s] (0.8,-0.2) rectangle (3.2,0.2) {};
					\filldraw[s] (1.8,0.8) rectangle (4.2,1.2) {};
					
					\filldraw[s] (-3.2,-0.2) rectangle (-0.8,0.2){};
					\filldraw[s] (-3.2,0.8) rectangle (-0.8,1.2){};
					
					\filldraw[s] (-2.2,-1.2) rectangle (-1.8,0.2) {};
					\filldraw[s] (-2.2,0.8) rectangle (-1.8,2.2) {};
					
					\filldraw[s] (-1.2,0.8) rectangle (1.2,1.2) {};
					\filldraw[s] (-0.2,-0.2) rectangle (0.2,1.2) {};

					\filldraw[s] (1.8,0.2) rectangle (2.2,-1.2){};
					\filldraw[s] (2.8,-0.2) rectangle (3.2,1.2) {};

					\node[vert, label={-90:$v_1$}] (v1) at (0,0) {};
					\node[dom_vert, label={90:$v_2$}] (v2) at (0,1) {};
					\node[vert, label={45:$v_3$}] (v3) at (-1,1) {};
					\node[dom_vert, label={45:$v_4$}] (v4) at (-2,1) {};
					\node[vert, label={90:$v_5$}] (v5) at (-3,1) {};
					\node[vert, label={-90:$v_6$}] (v6) at (-3,0) {};
					\node[dom_vert, label={-45:$v_7$}] (v7) at (-2,0) {};
					\node[vert, label={-45:$v_8$}] (v8) at (-1,0) {};
					\node[vert, label={90:$u_3$}] (u3) at (1,1) {};
					\node[vert, label={90:$u_4$}] (u4) at (2,1) {};
					\node[dom_vert, label={90:$u_5$}] (u5) at (3,1) {};
					\node[vert, label={-90:$u_6$}] (u6) at (3,0) {};
					\node[dom_vert, label={-45:$u_7$}] (u7) at (2,0) {};
					\node[vert, label={-90:$u_8$}] (u8) at (1,0) {};
					\node[vert,label={0:$v_4'$}] (v4') at (-2,2) {};
					\node[vert,label={0:$v_7'$}] (v7') at (-2,-1) {};
					\node[vert,label={90:$u_5'$}] (u5') at (4,1) {};
					\node[vert,label={0:$u_7'$}] (u7') at (2,-1) {};
					
					\node[outvert] (v4'1) at (-2.8,3.2) {$2^+$};
					\node[outvert] (v4'2) at (-1.2,3.2) {$3$};
					\node[outvert] (v7'1) at (-2.8,-2.2) {$2^+$};
					\node[outvert] (v7'2) at (-1.2,-2.2) {$3$};
					\node[outvert] (v31) at (-1,2.5) {$3$};
					\node[outvert] (v81) at (-1,-1.5) {$3$};
					\node[outvert] (u61) at (4.5,0) {$2^+$};
					\node[outvert] (u7'1) at (1.2,-2.2) {$2^+$};
					\node[outvert] (u7'2) at (2.8,-2.2) {$3$};
					\node[outvert] (u5'1) at (5.2,1.8) {$2^+$};
					\node[outvert] (u5'2) at (5.2,0.2) {$3$};
					
					\node at (-3.5,0.5) {$S$};
					\node at (-1.5,0.5) {face $f$};
					\node at (1.5,0.5) {bad 8-face $f'$};
					
					\draw (v4) -- (v5) -- (v6) -- (v7) -- (v8) -- (v1) -- (v2) -- (v3) -- (v4);
					\draw (v2) -- (u3) -- (u4) -- (u5) -- (u6) -- (u7) -- (u8) -- (v1);
					\draw (v4) -- (v4');
					\draw (v7) -- (v7');
					\draw (u5) -- (u5');
					\draw (u7) -- (u7');

					\draw[outedge] (v4'1) -- (v4') -- (v4'2);
					\draw[outedge] (v7'1) -- (v7') -- (v7'2);
					\draw[outedge] (v3) -- (v31);
					\draw[outedge] (v8) -- (v81);
					\draw[outedge] (u6) -- (u61);
					\draw[outedge] (u5'1) -- (u5') -- (u5'2);
					\draw[outedge] (u7'1) -- (u7') -- (u7'2);
				\end{tikzpicture}
				
				\caption{The configuration from Lemma \ref{lem:good-8}, Case 1a.}
				\label{fig:good-8-x2-y1}
			\end{figure} 
			
			By Lemma \ref{lem:2-3-2}, $\deg_G(v_3) = \deg_G(v_8) = 3$. 
			By Lemma \ref{lem:6-path}, at most one of $v_4, v_5$ is a $2$-vertex. 
			By Lemma~\ref{lem:good-3333}, at most one of $v_4, v_7$ is a $2$-vertex.
			
			If $\deg_G(v_4) = 2$, then $\deg_G(v_5) =\deg_G(v_7)= 3$.
			By Lemma \ref{lem:2-3-2}, $\deg_G(v_6) = 3$.
			This is a contradiction that $x=2$.
			Thus $\deg_G(v_4) = 3$.
			
			\begin{description}
				\item[Case 1a.] $\deg_G(v_5) = 2$.\\
				Since $x=2$, exactly one of $v_6$ and $v_7$ is a $2$-vertex. By Lemma~\ref{lem:2-3-2}, $\deg_G(v_6) =2$ and $\deg_G(v_7) = 3$.
				Let the neighbor of a $3$-vertex $v_i$ that lies on neither $f$ nor $f'$ be denoted by $v_i'$, and similarly, the other neighbor of $u_i$ is $u_i'$. See Figure \ref{fig:good-8-x2-y1}.
				
				Let $S =\{v_1, \ldots,v_8, u_3,\ldots,u_8, v_4',v_7',u_5',u_7'\}= N_G[v_4, v_7, v_2, u_5, u_7]$. 
				\medskip
				
				\begin{claim}\label{clm:girth8-case1-1}
					$G-S$ has at most two isolated vertices.
				\end{claim}
				\begin{myproof}
					Suppose $\deg_{G-S}(w)=0$ for some $w \in V(G) \setminus S$.
					Since $G$ has girth at least $8$, $N_G(w)$ can contain at most one vertex from $\{v_3, v_4', v_7', v_8\}$, and at most one vertex from $\{u_5',u_6,u_7'\}$.
					Thus, $\deg_G(w)=2$.
					Also, since $G$ has girth at least $8$, the following are the only possibilities for $N_G(w)$.
					\begin{enumerate}
						\item[(1)] $N_G(w)=\{v_4',u_6\}$.
						\item[(2)] $N_G(w)=\{v_7',u_6\}$.
						\item[(3)] $N_G(w)=\{v_4',u_7'\}$.
						\item[(4)] $N_G(w)=\{v_7',u_7'\}$.
						\item[(5)] $N_G(w)=\{v_4',u_5'\}$.
						\item[(6)] $N_G(w)=\{v_7',u_5'\}$.
						\item[(7)] $N_G(w)=\{v_8,u_5'\}$.
					\end{enumerate}
					Suppose on the contrary that $\deg_{G-S}(w_1)=\deg_{G-S}(w_2)=\deg_{G-S}(w_3)=0$ for $w_1, w_2, w_3 \in V(G) \setminus S$.
					
					Suppose $w_1$ satisfies (1) so that $N_G(w_1)=\{v_4',u_6\}$.
					Since $\deg_G(u_6)=3$, $w_2$ and $w_3$ cannot satisfy (2).
					By Lemma~\ref{lem:2-3-2}, $w_2$ and $w_3$ cannot satisfy (3) and (5).
					By the planarity of $G$, $w_2$ and $w_3$ cannot satisfy (6) and (7).
					Thus, $w_2$ and $w_3$ both satisfy (4), which is a contradiction.
					By symmetry, none of $w_1, w_2, w_3$ satisfy (1).
					
					Suppose $w_1$ satisfies (2) so that $N_G(w_1)=\{v_7',u_6\}$.
					By the planarity of $G$, $w_2$ and $w_3$ cannot satisfy (3) and (6).
					By the girth condition of $G$, $w_2$ and $w_3$ cannot satisfy (4) and (7).
					Thus, $w_2$ and $w_3$ both satisfy (5), which is a contradiction.
					By symmetry, none of $w_1, w_2, w_3$ satisfy (2).
					
					Suppose $w_1$ satisfies (3) so that $N_G(w_1) = \{v_4',u_7'\}$.
					By Lemma~\ref{lem:2-3-2}, $w_2$ and $w_3$ cannot satisfy (4) and (5).
					By the planarity of $G$, $w_2$ and $w_3$ cannot satisfy (6) and (7).
					Thus, $w_2$ and $w_3$ satisfy none of the conditions (1) through (7), which is a contradiction.
					By symmetry, none of $w_1, w_2, w_3$ satisfy (3).
					
					Suppose $w_1$ satisfies (4) so that $N_G(w_1)=\{v_7',u_7'\}$.
					By Lemma~\ref{lem:2-3-2}, $w_2$ and $w_3$ cannot satisfy (6).
					By the planarity of $G$, $w_2$ and $w_3$ cannot satisfy (7).
					Thus, $w_2$ and $w_3$ both satisfy (5), which is a contradiction.
					By symmetry, none of $w_1, w_2, w_3$ satisfy (4).
					
					Suppose $w_1$ satisfies (5) so that $N_G(w_1)=\{v_4',u_5'\}$.
					By Lemma~\ref{lem:2-3-2}, $w_2$ and $w_3$ cannot satisfy (6) and (7).
					Thus, $w_2$ and $w_3$ satisfy none of the conditions (1) through (7), which is a contradiction.
					By symmetry, none of $w_1, w_2, w_3$ satisfy (5).
					
					Now, since $w_1, w_2, w_3$ cannot satisfy conditions (1) through (5), 
					all of $w_1, w_2, w_3$ satisfy (6) or (7), which is a contradiction.
				\end{myproof}
				
				By Claim~\ref{clm:girth8-case1-1}, $c_G(S) \leq 11 \cdot 4 + 2 = 46$.
				Since $w_G(S) \geq 5 \cdot 11  + 13 \cdot 7 = 146$ and $\gamma(G[S]) = 5$, by the~\nameref{lem:key}, $5 \cdot 20 = 100 > 146 - 46 = 100$, which is a contradiction.    
				
				\item[Case 1b.] $\deg_G(v_5) = 3$.\\
				Since $x=2$, $v_6, v_7$ are 2-vertices; see Figure \ref{fig:good-8-1b}.
				
				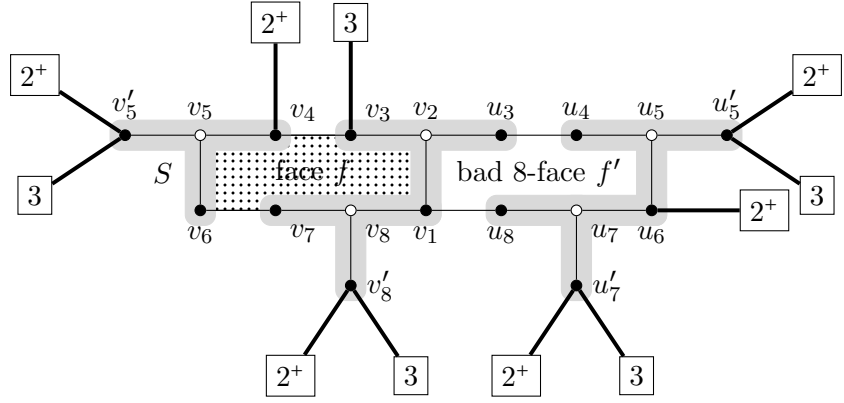
\begin{figure}[h!!]
					\centering
					
					\begin{tikzpicture}
						[
						vert/.style={circle,fill=black,draw=black, inner sep=0.05cm},
						s/.style={fill=black!15!white, draw=black!15!white, rounded corners},
						outvert/.style={rectangle,draw=black},
						outedge/.style={line width=1.5pt},
						dom_vert/.style={circle,draw=black,fill=white, inner sep=0.05cm}
						] 
						\draw[pattern=dots] (0,0) rectangle (-3,1) {};
						
						\filldraw[s] (0.8,-0.2) rectangle (3.2,0.2) {};
						\filldraw[s] (1.8,0.8) rectangle (4.2,1.2) {};
						
						\filldraw[s] (-2.2,-0.2) rectangle (0.2,0.2){};
						\filldraw[s] (-4.2,0.8) rectangle (-1.8,1.2){};
						
						\filldraw[s] (-1.2,-1.2) rectangle (-0.8,0.2) {};
						\filldraw[s] (-3.2,-0.2) rectangle (-2.8,1.2) {};
						
						\filldraw[s] (-1.2,0.8) rectangle (1.2,1.2) {};
						\filldraw[s] (-0.2,-0.2) rectangle (0.2,1.2) {};
						
						\filldraw[s] (1.8,0.2) rectangle (2.2,-1.2){};
						\filldraw[s] (2.8,-0.2) rectangle (3.2,1.2) {};
						
						\node[vert, label={-90:$v_1$}] (v1) at (0,0) {};
						\node[dom_vert, label={90:$v_2$}] (v2) at (0,1) {};
						\node[vert, label={45:$v_3$}] (v3) at (-1,1) {};
						\node[vert, label={45:$v_4$}] (v4) at (-2,1) {};
						\node[dom_vert, label={90:$v_5$}] (v5) at (-3,1) {};
						\node[vert, label={-90:$v_6$}] (v6) at (-3,0) {};
						\node[vert, label={-45:$v_7$}] (v7) at (-2,0) {};
						\node[dom_vert, label={-45:$v_8$}] (v8) at (-1,0) {};
						\node[vert, label={90:$u_3$}] (u3) at (1,1) {};
						\node[vert, label={90:$u_4$}] (u4) at (2,1) {};
						\node[dom_vert, label={90:$u_5$}] (u5) at (3,1) {};
						\node[vert, label={-90:$u_6$}] (u6) at (3,0) {};
						\node[dom_vert, label={-45:$u_7$}] (u7) at (2,0) {};
						\node[vert, label={-90:$u_8$}] (u8) at (1,0) {};
						\node[vert,label={90:$v_5'$}] (v5') at (-4,1) {};
						\node[vert,label={0:$v_8'$}] (v8') at (-1,-1) {};
						\node[vert,label={90:$u_5'$}] (u5') at (4,1) {};
						\node[vert,label={0:$u_7'$}] (u7') at (2,-1) {};
						
						\node[outvert] (v5'1) at (-5.2,1.8) {$2^+$};
						\node[outvert] (v5'2) at (-5.2,0.2) {$3$};
						\node[outvert] (v8'1) at (-1.8,-2.2) {$2^+$};
						\node[outvert] (v8'2) at (-0.2,-2.2) {$3$};
						\node[outvert] (v31) at (-1,2.5) {$3$};
						\node[outvert] (v41) at (-2,2.5) {$2^+$};
						\node[outvert] (u61) at (4.5,0) {$2^+$};
						\node[outvert] (u7'1) at (1.2,-2.2) {$2^+$};
						\node[outvert] (u7'2) at (2.8,-2.2) {$3$};
						\node[outvert] (u5'1) at (5.2,1.8) {$2^+$};
						\node[outvert] (u5'2) at (5.2,0.2) {$3$};
						
						\node at (-3.5,0.5) {$S$};
						\node at (-1.5,0.5) {face $f$};
						\node at (1.5,0.5) {bad 8-face $f'$};
						
						\draw (v4) -- (v5) -- (v6) -- (v7) -- (v8) -- (v1) -- (v2) -- (v3) -- (v4);
						\draw (v2) -- (u3) -- (u4) -- (u5) -- (u6) -- (u7) -- (u8) -- (v1);
						\draw (v5) -- (v5');
						\draw (v8) -- (v8');
						\draw (u5) -- (u5');
						\draw (u7) -- (u7');

						\draw[outedge] (v5'1) -- (v5') -- (v5'2);
						\draw[outedge] (v8'1) -- (v8') -- (v8'2);
						\draw[outedge] (v3) -- (v31);
						\draw[outedge] (v4) -- (v41);
						\draw[outedge] (u6) -- (u61);
						\draw[outedge] (u5'1) -- (u5') -- (u5'2);
						\draw[outedge] (u7'1) -- (u7') -- (u7'2);
					\end{tikzpicture}
					
					\caption{The configuration from Lemma \ref{lem:good-8}, Case 1b.}
					\label{fig:good-8-1b}
				\end{figure} 
				
				Let $S =\{v_1,\ldots,v_8,u_3,\ldots,u_8, v_5',v_8',u_5',u_7'\}= N_G[v_2, v_5, v_8, u_5, u_7]$.  
				\medskip
				
				\begin{claim}\label{clm:girth8-case1-2}
					$G-S$ has at most two isolated vertices.    
				\end{claim}
				\begin{myproof}
					Suppose $\deg_{G-S}(w)=0$ for some $w \in V(G) \setminus S$.
					Since $G$ has girth at least $8$, $N_G(w)$ can contain at most one vertex from $\{v_3,v_4, v_5', v_8'\}$, and at most one vertex from $\{u_5',u_6,u_7'\}$.
					Thus $\deg_G(w)=2$.
					Also, since $G$ has girth at least $8$, the following are the only possibilities for $N_G(w)$.
					\begin{itemize}
						\item[(1)] $N_G(w)=\{v_8',u_5'\}$.
						\item[(2)] $N_G(w)=\{v_4,u_5'\}$.
						\item[(3)] $N_G(w)=\{v_5',u_5'\}$.
						\item[(4)] $N_G(w)=\{v_4,u_7'\}$.
						\item[(5)] $N_G(w)=\{v_5',u_7'\}$.
						\item[(6)] $N_G(w)=\{v_4,u_6\}$.
						\item[(7)] $N_G(w)=\{v_5',u_6\}$.
					\end{itemize}
					Suppose on the contrary that $\deg_{G-S}(w_1) = \deg_{G-S}(w_2) = \deg_{G-S}(w_3)=0$ for $w_1, w_2, w_3 \in V(G) \setminus S$.
					
					Suppose $w_1$ satisfies (1) so that $N_G(w_1)=\{v_8',u_5'\}$.
					By Lemma~\ref{lem:2-3-2}, $w_2$ and $w_3$ cannot satisfy (2) and (3).
					By the planarity of $G$, $w_2$ and $w_3$ cannot satisfy (4), (5), (6), and (7).
					Thus, $w_2$ and $w_3$ satisfy none of the conditions (1) through (7), which is a contradiction.
					By symmetry, none of $w_1$, $w_2$, $w_3$ satisfy (1).
					
					Suppose $w_1$ satisfies (2) so that $N_G(w_1)=\{v_4,u_5'\}$.
					By Lemma~\ref{lem:2-3-2}, $w_2$ and $w_3$ cannot satisfy (3).
					Since $\deg_G(v_4)=3$, $w_2$ and $w_3$ cannot satisfy (4) and (6).
					Thus, $w_2$ and $w_3$ satisfy (5) and (7) each, which is a contradiction to Lemma~\ref{lem:2-3-2}.
					By symmetry, none of $w_1$, $w_2$, $w_3$ satisfy (2).
					
					Suppose $w_1$ satisfies (3) so that $N_G(w_1)=\{v_5',u_5'\}$.
					By Lemma~\ref{lem:2-3-2}, $w_2$ and $w_3$ cannot satisfy (5) and (7).
					By the planarity of $G$, $x_2$ and $w_3$ cannot satisfy (4) and (6).
					Thus, $w_2$ and $w_3$ satisfy none of the conditions (1) through (7), which is a contradiction.
					By symmetry, none of $w_1$, $w_2$, $w_3$ satisfy (3).
					
					Suppose $w_1$ satisfies (4) so that $N_G(w_1)=\{v_4,u_7'\}$.
					By Lemma~\ref{lem:2-3-2}, $w_2$ and $w_3$ cannot satisfy (5).
					Since $\deg_G(v_4)=3$, $w_2$ and $w_3$ cannot satisfy (6).
					Thus, $w_2$ and $w_3$ both satisfy (7), which is a contradiction.
					By symmetry, none of $w_1$, $w_2$, $w_3$ satisfy (4).
					
					Suppose $w_1$ satisfies (5) so that $N_G(w_1)=\{v_5',u_7'\}$.
					By Lemma~\ref{lem:2-3-2}, $w_2$ and $w_3$ cannot satisfy (7).
					Thus, $w_2$ and $w_3$ both satisfy (6), which is a contradiction.
					By symmetry, none of $w_1$, $w_2$, $w_3$ satisfy (5).
					
					Now, since $w_1$, $w_2$, and $w_3$ cannot satisfy conditions (1) through (5), 
					all of $w_1$, $w_2$ and $w_3$ satisfy (6) or (7), which is a contradiction.
				\end{myproof}
				By Claim~\ref{clm:girth8-case1-2}, $c_G(S) \le 11 \cdot 4 + 2 = 46$. Since $w_G(S) = 5 \cdot 11 + 13 \cdot 7 = 146$ and $\gamma(G[S]) = 5$, by the \nameref{lem:key}, $100> 146-46=100$, which is a contradiction.
				
			\end{description}
			
			\item[Case 2.] $x = 1$ and $y \geq 2$.\\
			Let the vertices of $f$ be $v_1, \ldots, v_8$. Let $\deg_G(v_1) = 2$. Since $x = 1$, all other vertices on $f$ are of degree $3$. By Lemma \ref{lem:good-3333}, there are only four possible edges that $f$ can have in common with a bad 8-face ($v_3 v_4$, $v_4 v_5$, $v_5 v_6$, $v_6 v_7$). Since bad 8-faces cannot share a contributing vertex by Lemma \ref{lem:adj-bad-faces}, there are only two different configurations to consider. However, if $v_3, v_4$ and $v_5, v_6$ are contributing vertices of $f$, then neither of $v_3, v_4$ has a neighbor on $f$ that is a 3-vertex with only degree 3 neighbors; see Figure \ref{fig:good-8-2a}. Since this contradicts Lemma \ref{lem:good-3333}, there is only one configuration left to consider.
			
			\begin{figure}[h!!]
				\centering
				
				\begin{tikzpicture}
					[
					vert/.style={circle,fill=black,draw=black, inner sep=0.05cm},
					s/.style={fill=black!15!white, draw=black!15!white, rounded corners},
					outvert/.style={rectangle,draw=black},
					outedge/.style={line width=1.5pt},
					dom_vert/.style={circle,draw=black,fill=white, inner sep=0.05cm}
					] 
					\draw[pattern=dots] (0,0) rectangle (2.5, 1)  {};
					
					\node[vert, label={-90:$v_1$}] (v1) at (1,1) {};
					\node[vert, label={-90:$v_2$}] (v2) at (0.5,1) {};
					\node[vert, label={90:$v_3$}] (v3) at (0,1) {};
					\node[vert, label={-90:$v_4$}] (v4) at (0,0) {};
					\node[vert, label={-90:$v_5$}] (v5) at (2.5,0) {};
					\node[vert, label={90:$v_6$}] (v6) at (2.5,1) {};
					\node[vert, label={-90:$v_7$}] (v7) at (2,1) {};
					\node[vert, label={-90:$v_8$}] (v8) at (1.5,1) {};
					
					\node (l1) at (1.25,0.3) {face $f$};
					\node[fill=white] (l2) at (-1.5,0.5) {bad 8-face};
					\node[fill=white] (l5) at (4,0.5) {bad 8-face};
					
					\foreach \x [remember=\x as \lastx (initially 1)] in {1,...,8,1}
					\path (v\x) edge (v\lastx);
					
					\node[vert] (u1) at (-1,1) {};
					\node[vert] (u2) at (-2,1) {};
					\node[vert] (u3) at (-3,1) {};
					\node[vert] (u4) at (-3,0) {};
					\node[vert] (u5) at (-2,0) {};
					\node[vert] (u6) at (-1,0) {};
					
					\foreach \x [remember=\x as \lastx (initially 1)] in {2,...,6}
					\path (u\x) edge (u\lastx);
					
					\draw (v3) -- (u1);
					\draw (v4) -- (u6);
					
					\node[vert] (u1') at (3.5,1) {};
					\node[vert] (u2') at (4.5,1) {};
					\node[vert] (u3') at (5.5,1) {};
					\node[vert] (u4') at (5.5,0) {};
					\node[vert] (u5') at (4.5,0) {};
					\node[vert] (u6') at (3.5,0) {};
					
					\foreach \x [remember=\x as \lastx (initially 1)] in {2,...,6}
					\path (u\x') edge (u\lastx');
					
					\draw (v6) -- (u1');
					\draw (v5) -- (u6');
					
					\draw (v2) -- (0.5, 2);
					\draw (v8) -- (1.5, 2);
					\draw (v7) -- (2, 2);
					\draw (u3) -- (-4,1);
					\draw (u4) -- (-4,0);
					\draw (u5) -- (-2,-1);
					\draw (u3') -- (6.5,1);
					\draw (u4') -- (6.5,0);
					\draw (u5') -- (4.5,-1);
					
				\end{tikzpicture}
				
				\caption{An example of the impossible configuration from Lemma \ref{lem:good-8}, Case 2.}
				\label{fig:good-8-2a}
			\end{figure}
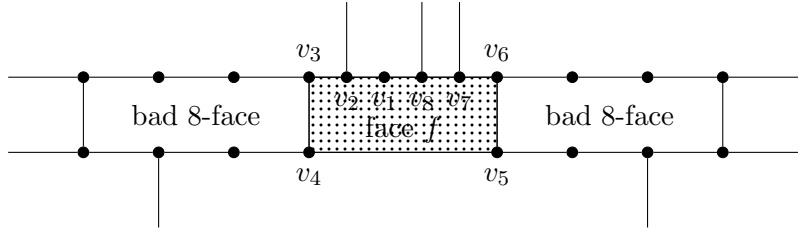  
			
			The remaining configuration is sketched in Figure \ref{fig:good-8-2b}. 
			
			\begin{figure}[h!!]
				\centering
				
				\begin{tikzpicture}
					[
					vert/.style={circle,fill=black,draw=black, inner sep=0.05cm},
					s/.style={fill=black!15!white, draw=black!15!white, rounded corners},
					outvert/.style={rectangle,draw=black},
					outedge/.style={line width=1.5pt},
					dom_vert/.style={circle,draw=black,fill=white, inner sep=0.05cm}
					]
					\draw[pattern=dots] (0,0) rectangle (2, 1)  {};
					
					\draw[s] (-0.2,-0.2) rectangle (2.2,0.2);
					\draw[s] (0.8,0.2) rectangle (1.2,-1.2);
					
					\draw[s] (-3.2,-0.2) rectangle (-0.8,0.2);
					\draw[s] (-2.2,0.2) rectangle (-1.8,-1.2);
					
					\draw[s] (2.8,-0.2) rectangle (5.2,0.2);
					\draw[s] (3.8,0.2) rectangle (4.2,-1.2);
					
					\draw[s] (-3.2,0.8) rectangle (-0.8,1.2);
					
					\draw[s] (2.8,0.8) rectangle (5.2,1.2);
					
					\draw[s] (-0.2,0.8) rectangle (2.2,1.2);
					\draw[s] (0.3,0.8) rectangle (0.7,2.2);
					\draw[s] (1.3,0.8) rectangle (1.7,2.2);
					
					\node (S) at (-2.5,-0.5) {$S$};
					
					\node[vert, label={-90:$v_1$}] (v1) at (1,1) {};
					\node[dom_vert, label={-90:$v_2$}] (v2) at (0.5,1) {};
					\node[vert, label={90:$v_3$}] (v3) at (0,1) {};
					\node[vert, label={-90:$v_4$}] (v4) at (0,0) {};
					\node[dom_vert, label={-45:$v_5$}] (v5) at (1,0) {};
					\node[vert, label={-90:$v_6$}] (v6) at (2,0) {};
					\node[vert, label={90:$v_7$}] (v7) at (2,1) {};
					\node[dom_vert, label={-90:$v_8$}] (v8) at (1.5,1) {};
					\node[vert] (v2s) at (0.5,2) {};
					\node[vert] (v8s) at (1.5,2) {};
					\node[vert] (v5s) at (1,-1) {};
					
					\node (l1) at (1,0.3) {face $f$};
					\node[fill=white] (l2) at (-1.5,0.5) {bad 8-face};
					\node[fill=white] (l5) at (3.5,0.5) {bad 8-face};
					
					\foreach \x [remember=\x as \lastx (initially 1)] in {1,...,8,1}
					\path (v\x) edge (v\lastx);
					
					\node[vert, label={90:$w_1$}] (u1) at (-1,1) {};
					\node[dom_vert, label={90:$w_2$}] (u2) at (-2,1) {};
					\node[vert, label={90:$w_3$}] (u3) at (-3,1) {};
					\node[vert, label={-90:$w_4$}] (u4) at (-3,0) {};
					\node[dom_vert, label={-45:$w_5$}] (u5) at (-2,0) {};
					\node[vert, label={-90:$w_6$}] (u6) at (-1,0) {};
					\node[vert] (u5s) at (-2,-1) {};
					
					\foreach \x [remember=\x as \lastx (initially 1)] in {2,...,6}
					\path (u\x) edge (u\lastx);
					
					\draw (v3) -- (u1);
					\draw (v4) -- (u6);
					
					\node[vert, label={90:$u_1$}] (u1') at (3,1) {};
					\node[dom_vert, label={90:$u_2$}] (u2') at (4,1) {};
					\node[vert, label={90:$u_3$}] (u3') at (5,1) {};
					\node[vert, label={-90:$u_4$}] (u4') at (5,0) {};
					\node[dom_vert, label={-45:$u_5$}] (u5') at (4,0) {};
					\node[vert, label={-90:$u_6$}] (u6') at (3,0) {};
					\node[vert] (u5's) at (4,-1) {};
					
					\foreach \x [remember=\x as \lastx (initially 1)] in {2,...,6}
					\path (u\x') edge (u\lastx');
					
					\draw (v7) -- (u1');
					\draw (v5) -- (u6');
					\draw (v2) -- (v2s);
					\draw (v8) -- (v8s);
					\draw (v5) -- (v5s);
					\draw (u5) -- (u5s);
					\draw (u5') -- (u5's);
					
					\node[outvert] (w31) at (-4.5,1) {$3$};
					\node[outvert] (w41) at (-4.5,0) {$2^+$};
					\node[outvert] (u31) at (6.5,1) {$3$};
					\node[outvert] (u41) at (6.5,0) {$2^+$};
					\node[outvert] (w51) at (-2.8,-2.2) {$3$};
					\node[outvert] (w52) at (-1.2,-2.2) {$2^+$};
					\node[outvert] (v51) at (0.2,-2.2) {$3$};
					\node[outvert] (v52) at (1.8,-2.2) {$2^+$};
					\node[outvert] (u51) at (3.2,-2.2) {$3$};
					\node[outvert] (u52) at (4.8,-2.2) {$2^+$};
					\node[outvert] (v21) at (-0.5,3.2) {$3$};
					\node[outvert] (v22) at (0.5,3.2) {$2^+$};
					\node[outvert] (v81) at (1.5,3.2) {$3$};
					\node[outvert] (v82) at (2.5,3.2) {$2^+$};
					
					\draw[outedge] (v21) -- (v2s) -- (v22);
					\draw[outedge] (v81) -- (v8s) -- (v82);
					\draw[outedge] (v51) -- (v5s) -- (v52);
					\draw[outedge] (w51) -- (u5s) -- (w52);
					\draw[outedge] (u51) -- (u5's) -- (u52);
					\draw[outedge] (u3) -- (w31);
					\draw[outedge] (u4) -- (w41);
					\draw[outedge] (u3') -- (u31);
					\draw[outedge] (u4') -- (u41);
					
				\end{tikzpicture}
				
				\caption{The remaining configuration from Lemma \ref{lem:good-8}, Case 2.}
				\label{fig:good-8-2b}
			\end{figure}
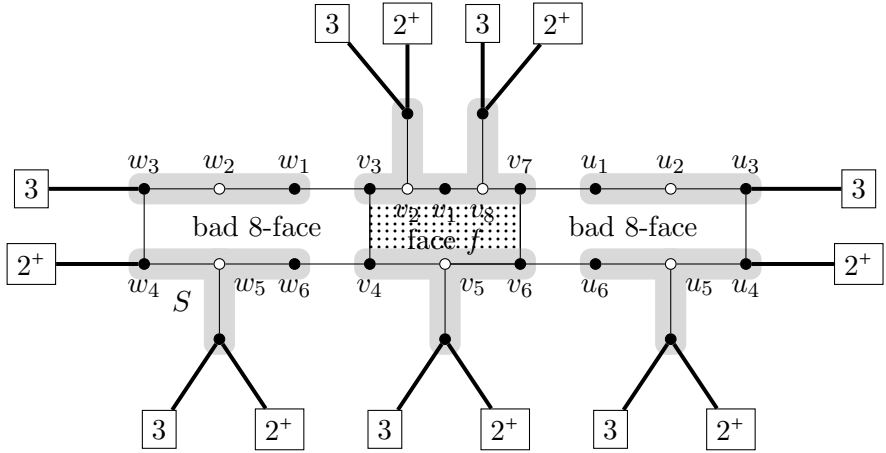 
			
			Let $S = N_G[v_2, v_5, v_8, u_2, u_5, w_2, w_5]$. Observe that there are seven 2-vertices and eighteen 3-vertices in $S$, and that the selection of $S$ is independent of the positions of 2-vertices on bad 8-faces. 
			Without loss of generality, let $\deg_G(w_5)=\deg_G(u_5)=3$. By Lemma~\ref{lem:2-3-2}, we see that the neighbors of $w_5$ or $u_5$ that are not on a bad face are of degree $3$.
			Note that $v_5$ and all of its neighbors are of degree 3  by Lemma \ref{lem:good-3333}. 
			Thus $\gamma(G[S]) = 7$ and $w_G(S) \geq 7 \cdot 11  + 18 \cdot 7 = 203$. There are 14 edges between $S$ and $G - S$. Thus, by Lemma~\ref{lem:cost-bound}, $c_G(S) \leq 14 \cdot 4 + \lfloor \frac{14}{2} \rfloor = 63$. The \nameref{lem:key} gives $7 \cdot 20 = 140 > 203 - 63 = 140$, which is a contradiction. 
			
			\item[Case 3.] $x = 0$ and $y \geq 3$.\\
			Let the vertices of $f$ be $v_1, \ldots, v_8$. Since $x = 0$, all vertices on $f$ are of degree $3$. Lemmas \ref{lem:good-3333} and \ref{lem:adj-bad-faces} imply that in fact $y \leq 3$, and there is only one possible configuration; see Figure \ref{fig:good-8-3}.
			
			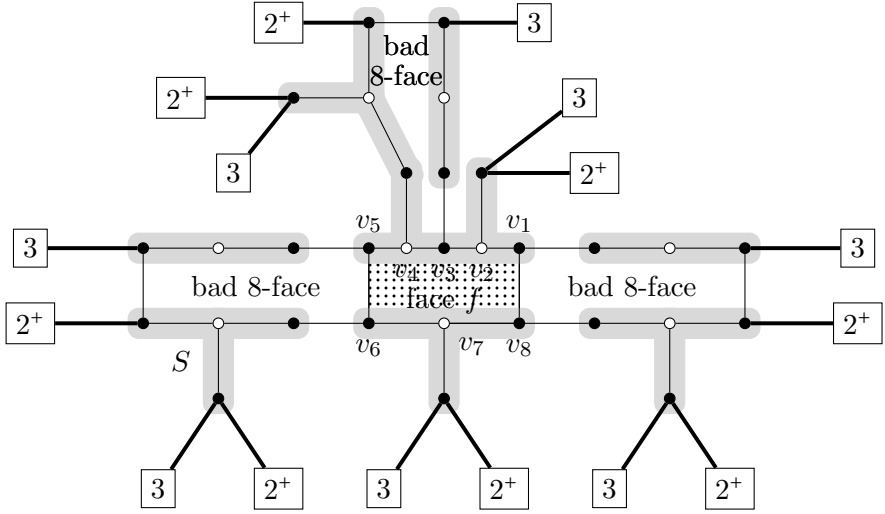
\begin{figure}[h!!]
				\centering
				
				\begin{tikzpicture}
					[
					vert/.style={circle,fill=black,draw=black, inner sep=0.05cm},
					s/.style={fill=black!15!white, draw=black!15!white, rounded corners},
					outvert/.style={rectangle,draw=black},
					outedge/.style={line width=1.5pt},
					dom_vert/.style={circle,draw=black,fill=white, inner sep=0.05cm}
					]
					\draw[pattern=dots] (0,0) rectangle (2, 1)  {};
					
					\draw[s] (-0.2,-0.2) rectangle (2.2,0.2);
					\draw[s] (0.8,0.2) rectangle (1.2,-1.2);
					
					\draw[s] (-3.2,-0.2) rectangle (-0.8,0.2);
					\draw[s] (-2.2,0.2) rectangle (-1.8,-1.2);
					
					\draw[s] (2.8,-0.2) rectangle (5.2,0.2);
					\draw[s] (3.8,0.2) rectangle (4.2,-1.2);
					
					\draw[s] (-3.2,0.8) rectangle (-0.8,1.2);
					
					\draw[s] (2.8,0.8) rectangle (5.2,1.2);
					
					\draw[s] (-0.2,0.8) rectangle (2.2,1.2);
					\draw[s] (0.3,0.8) rectangle (0.7,2.2);
					\draw[s] (1.3,0.8) rectangle (1.7,2.2);
					
					\draw[s] (0.8,1.8) rectangle (1.2,4.2);
					\draw[s] (-0.2,2.8) rectangle (0.2,4.2);
					\draw[s] (-1.2,2.8) rectangle (0.2,3.2);
					\draw[s, shift={(0.5,2)}, rotate=117] (-0.2,-0.2) rectangle (1.2,0.2);
					
					\node (S) at (-2.5,-0.5) {$S$};
					
					\node (l1) at (1,0.3) {face $f$};
					\node[fill=white] (l2) at (-1.5,0.5) {bad 8-face};
					\node[fill=white] (l5) at (3.5,0.5) {bad 8-face};
					\node (l81) at (0.5,3.7) {bad};
					\node (l82) at (0.5,3.3) {8-face};
					
					\node[vert, label={-90:$v_3$}] (v3) at (1,1) {};
					\node[dom_vert, label={-90:$v_4$}] (v4) at (0.5,1) {};
					\node[vert, label={90:$v_5$}] (v5) at (0,1) {};
					\node[vert, label={-90:$v_6$}] (v6) at (0,0) {};
					\node[dom_vert, label={-45:$v_7$}] (v7) at (1,0) {};
					\node[vert, label={-90:$v_8$}] (v8) at (2,0) {};
					\node[vert, label={90:$v_1$}] (v1) at (2,1) {};
					\node[dom_vert, label={-90:$v_2$}] (v2) at (1.5,1) {};
					\node[vert] (v8s) at (1.5,2) {};
					\node[vert] (v5s) at (1,-1) {};
					
					\node[vert] (x1) at (0.5,2) {};
					\node[dom_vert] (x2) at (0,3) {};
					\node[vert] (x3) at (0,4) {};
					\node[vert] (x4) at (1,4) {};
					\node[dom_vert] (x5) at (1,3) {};
					\node[vert] (x6) at (1,2) {};
					
					\node[vert] (x2s) at (-1,3) {};
					\draw (x2) -- (x2s);
					
					\draw (v4) -- (x1) -- (x2) -- (x3) -- (x4) -- (x5) -- (x6) -- (v3);
					
					\foreach \x [remember=\x as \lastx (initially 1)] in {1,...,8,1}
					\path (v\x) edge (v\lastx);
					
					\node[vert] (u1) at (-1,1) {};
					\node[dom_vert] (u2) at (-2,1) {};
					\node[vert] (u3) at (-3,1) {};
					\node[vert] (u4) at (-3,0) {};
					\node[dom_vert] (u5) at (-2,0) {};
					\node[vert] (u6) at (-1,0) {};
					\node[vert] (u5s) at (-2,-1) {};
					
					\foreach \x [remember=\x as \lastx (initially 1)] in {2,...,6}
					\path (u\x) edge (u\lastx);
					
					\draw (v5) -- (u1);
					\draw (v6) -- (u6);
					
					\node[vert] (u1') at (3,1) {};
					\node[dom_vert] (u2') at (4,1) {};
					\node[vert] (u3') at (5,1) {};
					\node[vert] (u4') at (5,0) {};
					\node[dom_vert] (u5') at (4,0) {};
					\node[vert] (u6') at (3,0) {};
					\node[vert] (u5's) at (4,-1) {};
					
					\foreach \x [remember=\x as \lastx (initially 1)] in {2,...,6}
					\path (u\x') edge (u\lastx');
					
					\draw (v1) -- (u1');
					\draw (v7) -- (u6');
					\draw (v2) -- (v8s);
					\draw (v7) -- (v5s);
					\draw (u5) -- (u5s);
					\draw (u5') -- (u5's);
					
					\node (l81) at (0.5,3.7) {bad};
					\node (l82) at (0.5,3.3) {8-face};
					
					\node[outvert] (w31) at (-4.5,1) {$3$};
					\node[outvert] (w41) at (-4.5,0) {$2^+$};
					\node[outvert] (u31) at (6.5,1) {$3$};
					\node[outvert] (u41) at (6.5,0) {$2^+$};
					\node[outvert] (w51) at (-2.8,-2.2) {$3$};
					\node[outvert] (w52) at (-1.2,-2.2) {$2^+$};
					\node[outvert] (v51) at (0.2,-2.2) {$3$};
					\node[outvert] (v52) at (1.8,-2.2) {$2^+$};
					\node[outvert] (u51) at (3.2,-2.2) {$3$};
					\node[outvert] (u52) at (4.8,-2.2) {$2^+$};
					\node[outvert] (v81) at (2.8,3) {$3$};
					\node[outvert] (v82) at (3,2) {$2^+$};
					
					\node[outvert] (vx21) at (-2.5,3) {$2^+$};
					\node[outvert] (vx22) at (-1.8,2) {$3$};
					\node[outvert] (vx3) at (-1.2,4) {$2^+$};
					\node[outvert] (vx4) at (2.2,4) {$3$};
					
					\draw[outedge] (v81) -- (v8s) -- (v82);
					\draw[outedge] (v51) -- (v5s) -- (v52);
					\draw[outedge] (w51) -- (u5s) -- (w52);
					\draw[outedge] (u51) -- (u5's) -- (u52);
					\draw[outedge] (u3) -- (w31);
					\draw[outedge] (u4) -- (w41);
					\draw[outedge] (u3') -- (u31);
					\draw[outedge] (u4') -- (u41);
					
					\draw[outedge] (vx21) -- (x2s) -- (vx22);
					\draw[outedge] (vx3) -- (x3);
					\draw[outedge] (vx4) -- (x4);
					
				\end{tikzpicture}
				
				\caption{The configuration from Lemma \ref{lem:good-8}, Case 3.}
				\label{fig:good-8-3}
			\end{figure}
			
			Let $S$ be as in the Figure \ref{fig:good-8-3}, and notice that it does not depend on the positions of the 2-vertices on the bad 8-faces. Thus $\gamma(G[S]) = 9$ and $w_G(S) \geq 9 \cdot 11  + 22 \cdot 7 = 253$. There are 16 edges between $S$ and $G - S$. Thus, by Lemma, \ref{lem:cost-bound} $c_G(S) \leq 16 \cdot 4 + \lfloor \frac{16}{2} \rfloor = 72$. The \nameref{lem:key} gives $9 \cdot 20 = 180 > 253 - 72 = 181$, which is a contradiction.
		\end{description}
		This concludes the proof of Lemma~\ref{lem:good-8}.
	\end{proof}
	
	\begin{lemma}
		\label{lem:good-9}
		If a 9-face $f$ is incident with $x$ 2-vertices and $y$ bad 8-faces, then $x+y \leq 3$.
	\end{lemma}
	
	\begin{proof}
		Since $f$ is a 9-face, $x \leq 3$ by Lemma \ref{lem:9-face}. Suppose that $x + y \geq 4$. We distinguish between the following cases. Let the vertices of $f$ be $v_1, \ldots, v_9$.
		
		\begin{description}
			\item[Case 1.] $x = 3$ and $y \geq 1$.\\
			Since $y \geq 1$, we may assume that $v_1, v_2$ are contributing vertices of $f$ that belong to a bad 8-face; see Figure \ref{fig:good-9-1}. By Lemma \ref{lem:good-3333}, both $v_3$ and $v_9$ are 3-vertices, and (without loss of generality) every neighbor of $v_3$ has degree $3$, so $\deg_G(v_4) = 3$. Now $f$ has at least three 2-vertices among $v_5, v_6, v_7, v_8$, which is not possible by Lemmas \ref{lem:2-2-2} and \ref{lem:2-3-2}.
			
			\begin{figure}[h!!]
				\centering
				\begin{tikzpicture}
					[scale=0.7,
					vert/.style={circle,fill=black,draw=black, inner sep=0.05cm},
					s/.style={fill=black!15!white, draw=black!15!white, rounded corners},
					s2/.style={fill=white, draw=black!50!white, rounded corners},
					outvert/.style={rectangle,draw=black},
					outedge/.style={line width=1.5pt},
					dom_vert/.style={circle,draw=black,fill=white, inner sep=0.05cm}
					] 
					\pgfmathtruncatemacro{\N}{9}
					\pgfmathtruncatemacro{\R}{2}
					\pgfmathtruncatemacro{\D}{3}
					
					\clip (-2.75,-2.75) rectangle (4.75, 2.5); % to remove the white space above and below
					
					\begin{scope}[scale=0.8]
						\filldraw[s2, rotate=32, shift={(-2.5,-2.3)}] (0,0) rectangle (2,4.5){};
						
						\foreach \x in {1,...,\N}
						\node[vert] (\x) at (-50 + \x*360/\N:\R cm) {};
						\foreach \x [remember=\x as \lastx (initially 1)] in {1,...,\N,1}
						\path (\x) edge (\lastx);
						
						\draw \foreach \x in {1,...,\N} {(-50 + \x*360/\N:1.5) node {$v_{\x}$}};
						
						\foreach \y in {3,4,9}
						\draw (\y) -- (-50 + \y*360/\N:\D cm);
						
						\draw    (1) to[out=-20,in=40,distance=6cm] (2);
						\node (l) at (3.7,0.7) {bad 8-face};
						
					\end{scope}
				\end{tikzpicture}
				
				\caption{The impossible configuration from Case 1 of Lemma \ref{lem:good-9}.}
				\label{fig:good-9-1}
			\end{figure}
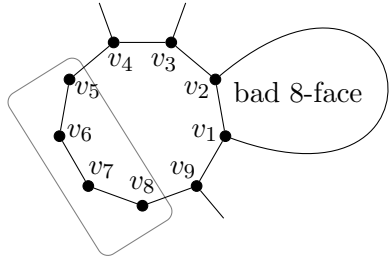  
			
			\item[Case 2.] $x = 2$ and $y \geq 2$.\\
			Since $y \geq 2$, $f$ is incident with at least two bad 8-faces. If $v_1, v_2$ and $v_3, v_4$ are contributing vertices, then by Lemma \ref{lem:good-3333}, $v_5, v_6, v_9, v_8$ are all 3-vertices. Thus $f$ can have at most one 2-vertex, $v_7$; see the left of Figure \ref{fig:good-9-2-not}. Similarly, if $v_1, v_2$ and $v_5, v_6$ are contributing vertices, then by Lemma \ref{lem:2-3-2}, $v_3, v_4, v_7, v_9$ are all 3-vertices, and $f$ can again have at most one 2-vertex, $v_8$; see the right of Figure \ref{fig:good-9-2-not}.
			
			\begin{figure}[h!!]
				\centering
				
				\begin{tikzpicture}
					[scale=0.7,
					vert/.style={circle,fill=black,draw=black, inner sep=0.05cm},
					s/.style={fill=black!15!white, draw=black!15!white, rounded corners},
					s2/.style={fill=white, draw=black!50!white, rounded corners},
					outvert/.style={rectangle,draw=black},
					outedge/.style={line width=1.5pt},
					dom_vert/.style={circle,draw=black,fill=white, inner sep=0.05cm}
					] 
					\pgfmathtruncatemacro{\N}{9}
					\pgfmathtruncatemacro{\R}{2}
					\pgfmathtruncatemacro{\D}{3}
					
					\clip (-3,-2.75) rectangle (15.25, 5); % to remove the white space above and below
					
					\begin{scope}[scale=0.8]
						\filldraw[s2, rotate=0, shift={(-1.8,-2)}] (0,0) rectangle (1.1,1.1){};
						
						\foreach \x in {1,...,\N}
						\node[vert] (\x) at (-50 + \x*360/\N:\R cm) {};
						\foreach \x [remember=\x as \lastx (initially 1)] in {1,...,\N,1}
						\path (\x) edge (\lastx);
						
						\draw \foreach \x in {1,...,\N} {(-50 + \x*360/\N:1.5) node {$v_{\x}$}};
						
						\foreach \y in {5,9}
						\node[outvert] (o\y) at (-50 + \y*360/\N:3.5 cm) {$3$};
						
						\foreach \y in {5,9}
						\draw (\y) -- (o\y);
						
						\foreach \y in {6,8}
						\draw (\y) -- (-50 + \y*360/\N:\D cm);
						
						\draw    (1) to[out=-20,in=40,distance=6cm] (2);
						\node (l1) at (3.7,0.7) {bad 8-face};
						
						\draw    (3) to[out=65,in=125,distance=6cm] (4);
						\node (l2) at (0,4.5) {bad};
						\node (l3) at (0,3.8) {8-face};
						
					\end{scope}
					
					\begin{scope}[scale=0.8, xshift=13cm]
						\filldraw[s2, rotate=0, shift={(-0.5,-2.3)}] (0,0) rectangle (1.1,1.1){};
						
						\foreach \x in {1,...,\N}
						\node[vert] (\x) at (-50 + \x*360/\N:\R cm) {};
						\foreach \x [remember=\x as \lastx (initially 1)] in {1,...,\N,1}
						\path (\x) edge (\lastx);
						
						\draw \foreach \x in {1,...,\N} {(-50 + \x*360/\N:1.5) node {$v_{\x}$}};
						
						\foreach \y in {3,4,7,9}
						\draw (\y) -- (-50 + \y*360/\N:\D cm);
						
						\draw (1) to[out=-20,in=40,distance=6cm] (2);
						\node (l1) at (3.7,0.7) {bad 8-face};
						
						\draw (5) to[out=140,in=-160,distance=6cm] (6);
						\node (l2) at (-3.7,0.7) {bad 8-face};
						
					\end{scope}
				\end{tikzpicture}
				
				\caption{The impossible configurations from Case 2 of Lemma \ref{lem:good-9}.}
				\label{fig:good-9-2-not}
			\end{figure}
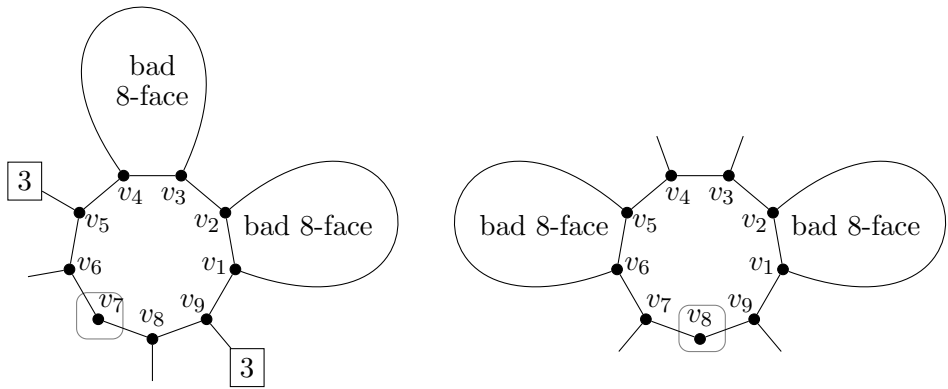
			
			By symmetry, the only remaining option is that $v_1, v_2$ and $v_4, v_5$ are contributing vertices. By Lemma \ref{lem:good-3333}, $v_3, v_6, v_9$ are 3-vertices. Since $x = 2$, it follows that $\deg_G(v_7) = \deg_G(v_8) = 2$. Thus by Lemma \ref{lem:2-3-2}, $v_6$ and $v_9$ have a neighbor of degree 3 outside $f$, and by Lemma \ref{lem:good-3333}, $v_3$ has a neighbor of degree 3 outside $f$. Let $S$ be as in Figure \ref{fig:good-9-2}. Notice that $S$ does not depend on the position of 2-vertices in the bad 8-faces. (Lemma \ref{lem:6-path} actually forces the position as in the figure, but we do not need this.)
			
			\begin{figure}[h!!]
				\centering
				
				\begin{tikzpicture}
					[
					vert/.style={circle,fill=black,draw=black, inner sep=0.05cm},
					s/.style={fill=black!15!white, draw=black!15!white, rounded corners},
					outvert/.style={rectangle,draw=black},
					outedge/.style={line width=1.5pt},
					dom_vert/.style={circle,draw=black,fill=white, inner sep=0.05cm}
					] 
					\draw[pattern=dots] (0,0) rectangle (2.5,1) {};
					
					\draw[s] (-0.2,-0.2) rectangle (2.7,0.2);
					\draw[s] (1.05,0.2) rectangle (1.45,-1.2);
					
					\draw[s] (-3.2,-0.2) rectangle (-0.8,0.2);
					
					\draw[s] (3.2,-0.2) rectangle (5.7,0.2);
					
					\draw[s] (-3.2,0.8) rectangle (-0.8,1.2);
					\draw[s] (-2.2,0.8) rectangle (-1.8,2.2);
					
					\draw[s] (3.2,0.8) rectangle (5.7,1.2);
					\draw[s] (4.3,0.8) rectangle (4.7,2.2);
					
					\draw[s] (-0.2,0.8) rectangle (1.2,1.2);
					\draw[s] (1.3,0.8) rectangle (2.7,1.2);
					\draw[s] (0.3,0.8) rectangle (0.7,2.2);
					\draw[s] (1.8,0.8) rectangle (2.2,2.2);
					
					\node (S) at (-2.5,-0.5) {$S$};
					
					\node[vert, label={-90:$v_8$}] (v1) at (1,1) {};
					\node[dom_vert, label={-90:$v_9$}] (v2) at (0.5,1) {};
					\node[vert, label={90:$v_1$}] (v3) at (0,1) {};
					\node[vert, label={-90:$v_2$}] (v4) at (0,0) {};
					\node[dom_vert, label={-45:$v_3$}] (v5) at (1.25,0) {};
					\node[vert, label={-90:$v_4$}] (v6) at (2.5,0) {};
					\node[vert, label={90:$v_5$}] (v7) at (2.5,1) {};
					\node[dom_vert, label={-90:$v_6$}] (v8) at (2,1) {};
					\node[vert, label={-90:$v_7$}] (v9) at (1.5,1) {};
					\node[vert] (v2s) at (0.5,2) {};
					\node[vert] (v8s) at (2,2) {};
					\node[vert] (v5s) at (1.25,-1) {};
					
					\node (l1) at (1.25,0.3) {face $f$};
					\node[fill=white] (l2) at (-1.5,0.5) {bad 8-face};
					\node[fill=white] (l5) at (4,0.5) {bad 8-face};
					
					\foreach \x [remember=\x as \lastx (initially 1)] in {1,...,9,1}
					\path (v\x) edge (v\lastx);
					
					\node[vert] (u1) at (-1,1) {};
					\node[dom_vert] (u2) at (-2,1) {};
					\node[vert] (u3) at (-3,1) {};
					\node[vert] (u4) at (-3,0) {};
					\node[dom_vert] (u5) at (-2,0) {};
					\node[vert] (u6) at (-1,0) {};
					\node[vert] (u5s) at (-2,2) {};
					
					\foreach \x [remember=\x as \lastx (initially 1)] in {2,...,6}
					\path (u\x) edge (u\lastx);
					
					\draw (v3) -- (u1);
					\draw (v4) -- (u6);
					
					\node[vert] (u1') at (3.5,1) {};
					\node[dom_vert] (u2') at (4.5,1) {};
					\node[vert] (u3') at (5.5,1) {};
					\node[vert] (u4') at (5.5,0) {};
					\node[dom_vert] (u5') at (4.5,0) {};
					\node[vert] (u6') at (3.5,0) {};
					\node[vert] (u5's) at (4.5,2) {};
					
					\foreach \x [remember=\x as \lastx (initially 1)] in {2,...,6}
					\path (u\x') edge (u\lastx');
					
					\draw (v7) -- (u1');
					\draw (v5) -- (u6');
					\draw (v2) -- (v2s);
					\draw (v8) -- (v8s);
					\draw (v5) -- (v5s);
					\draw (u2) -- (u5s);
					\draw (u2') -- (u5's);
					
					\node[outvert] (w31) at (-4.5,1) {$2^+$};
					\node[outvert] (w41) at (-4.5,0) {$3$};
					\node[outvert] (u31) at (7,1) {$2^+$};
					\node[outvert] (u41) at (7,0) {$3$};
					\node[outvert] (w51) at (-2.8,3.2) {$3$};
					\node[outvert] (w52) at (-1.2,3.2) {$2^+$};
					\node[outvert] (v51) at (0.45,-2.2) {$3$};
					\node[outvert] (v52) at (2.05,-2.2) {$2^+$};
					\node[outvert] (u51) at (3.7,3.2) {$3$};
					\node[outvert] (u52) at (5.3,3.2) {$2^+$};
					\node[outvert] (v21) at (-0.25,3.2) {$3$};
					\node[outvert] (v22) at (0.75,3.2) {$2^+$};
					\node[outvert] (v81) at (2,3.2) {$3$};
					\node[outvert] (v82) at (3,3.2) {$2^+$};
					
					\draw[outedge] (v21) -- (v2s) -- (v22);
					\draw[outedge] (v81) -- (v8s) -- (v82);
					\draw[outedge] (v51) -- (v5s) -- (v52);
					\draw[outedge] (w51) -- (u5s) -- (w52);
					\draw[outedge] (u51) -- (u5's) -- (u52);
					\draw[outedge] (u3) -- (w31);
					\draw[outedge] (u4) -- (w41);
					\draw[outedge] (u3') -- (u31);
					\draw[outedge] (u4') -- (u41);
					
				\end{tikzpicture}
				
				\caption{The configuration from Lemma \ref{lem:good-9}, Case 3.}
				\label{fig:good-9-2}
			\end{figure}
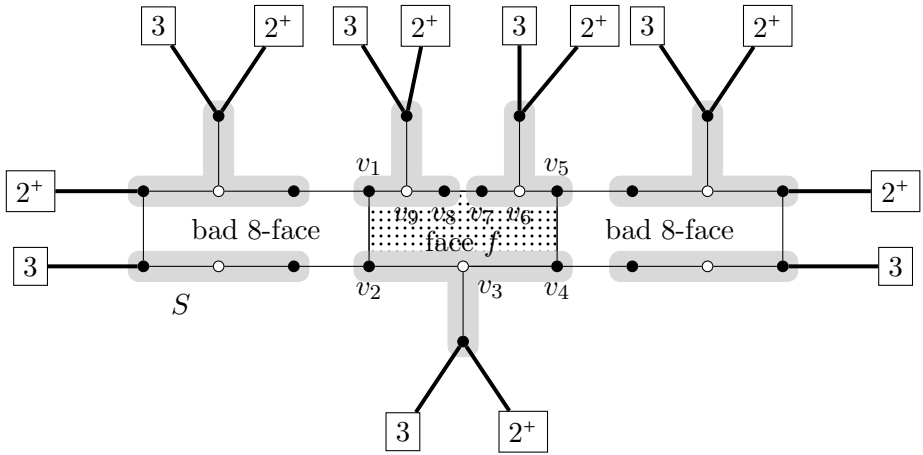
			
			Clearly, $\gamma(G[S]) = 7$ and $w_G(S) \geq 8 \cdot 11  + 18 \cdot 7 = 214$. There are 14 edges between $S$ and $V(G) \setminus S$. 
			Thus, by Lemma \ref{lem:cost-bound}, $c_G(S) \leq 14 \cdot 4 + \lfloor \frac{14}{2} \rfloor = 63$. The \nameref{lem:key} gives $7 \cdot 20 = 140 > 214 - 63 = 151$, which is a contradiction. 
			
			\item[Case 3.] $x = 1$ and $y \geq 3$.\\
			Since $x=1$, we may assume that $\deg_G(v_1) = 2$, and this is the only 2-vertex on $f$. By Lemma \ref{lem:2-3-2}, $v_2$ and $v_9$ are adjacent to a 3-vertex outside $f$. Thus only pairs $\{v_3, v_4\}$, $\{v_4, v_5\}$, $\{v_5, v_6\}$, $\{v_6, v_7\}$, and $\{v_7, v_8\}$ can form a pair of contributing vertices of $f$. But by Lemma \ref{lem:adj-bad-faces} and since $y \geq 3$, the unique option is that $\{v_3, v_4\}$, $\{v_5, v_6\}$, and $\{v_7, v_8\}$ are pairs of contributing vertices. But this gives a contradiction with Lemma \ref{lem:good-3333} since for example $v_2$ and $v_5$ are both adjacent to some 2-vertex; see Figure \ref{fig:good-9-3}.
			
			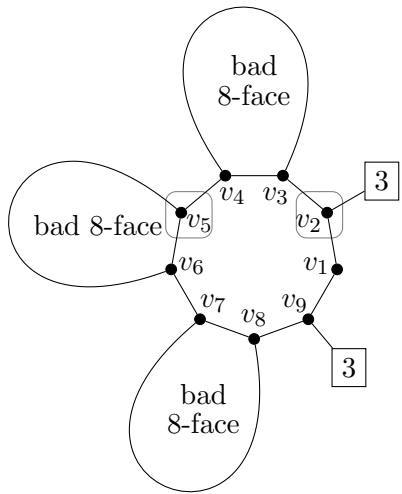
\begin{figure}[h!!]
				\centering
				
				\begin{tikzpicture}
					[scale=0.7,
					vert/.style={circle,fill=black,draw=black, inner sep=0.05cm},
					s/.style={fill=black!15!white, draw=black!15!white, rounded corners},
					s2/.style={fill=white, draw=black!50!white, rounded corners},
					outvert/.style={rectangle,draw=black},
					outedge/.style={line width=1.5pt},
					dom_vert/.style={circle,draw=black,fill=white, inner sep=0.05cm}
					] 
					\pgfmathtruncatemacro{\N}{9}
					\pgfmathtruncatemacro{\R}{2}
					\pgfmathtruncatemacro{\D}{3}
					
					\clip (-4.75,-5) rectangle (3, 5); % to remove the white space above and below
					
					\begin{scope}[scale=0.8]
						\filldraw[s2, rotate=0, shift={(-2.1,0.4)}] (0,0) rectangle (1.1,1.1){};
						
						\filldraw[s2, rotate=0, shift={(1,0.4)}] (0,0) rectangle (1.1,1.1){};
						
						\foreach \x in {1,...,\N}
						\node[vert] (\x) at (-50 + \x*360/\N:\R cm) {};
						\foreach \x [remember=\x as \lastx (initially 1)] in {1,...,\N,1}
						\path (\x) edge (\lastx);
						
						\draw \foreach \x in {1,...,\N} {(-50 + \x*360/\N:1.5) node {$v_{\x}$}};
						
						\foreach \y in {2,9}
						\node[outvert] (o\y) at (-50 + \y*360/\N:3.5 cm) {$3$};
						
						\foreach \y in {2,9}
						\draw (\y) -- (o\y);

						\draw    (3) to[out=65,in=125,distance=6cm] (4);
						\node (l1) at (0,4.5) {bad};
						\node (l11) at (0,3.8) {8-face};
						
						\draw (5) to[out=140,in=-160,distance=6cm] (6);
						\node (l2) at (-3.7,0.7) {bad 8-face};
						
						\draw (7) to[out=-140,in=-80,distance=6cm] (8);
						\node (l3) at (-1.2,-3.3) {bad};
						\node (l4) at (-1.2,-4) {8-face};
						
					\end{scope}
					
				\end{tikzpicture}
				
				\caption{The impossible configuration from Case 3 of Lemma \ref{lem:good-9}.}
				\label{fig:good-9-3}
			\end{figure}
			
			\item[Case 4.] $x = 0$ and $y \geq 4$.\\
			Since $y \geq 4$ and Lemma \ref{lem:adj-bad-faces}, there is at least one pair of bad 8-faces such that there is only one edge of $f$ between the bad 8-faces. We may thus assume that $\{v_1, v_2\}$ and $\{v_3, v_4\}$ are pairs of contributing vertices. By Lemma \ref{lem:good-3333}, vertices $v_5$ and $v_9$ are adjacent to a 3-vertex outside of $f$; see Figure \ref{fig:good-9-4}. Thus only $\{v_6, v_7\}$ and $\{v_7, v_8\}$ remain as possible pairs of contributing vertices, but by Lemma \ref{lem:adj-bad-faces}, this is impossible. Thus $y \leq 3$, which is a contradiction.
			
			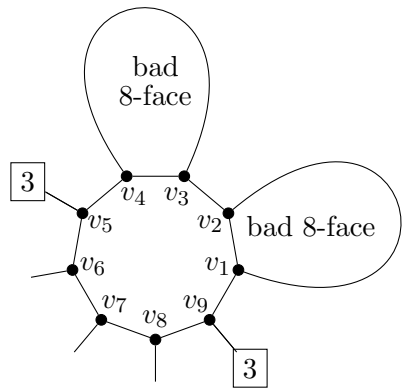
\begin{figure}[h!!]
				\centering
				
				\begin{tikzpicture}
					[scale=0.7,
					vert/.style={circle,fill=black,draw=black, inner sep=0.05cm},
					s/.style={fill=black!15!white, draw=black!15!white, rounded corners},
					s2/.style={fill=white, draw=black!50!white, rounded corners},
					outvert/.style={rectangle,draw=black},
					outedge/.style={line width=1.5pt},
					dom_vert/.style={circle,draw=black,fill=white, inner sep=0.05cm}
					] 
					\pgfmathtruncatemacro{\N}{9}
					\pgfmathtruncatemacro{\R}{2}
					\pgfmathtruncatemacro{\D}{3}
					
					\clip (-3,-2.75) rectangle (5, 5); % to remove the white space above and below
					
					\begin{scope}[scale=0.8]
						\foreach \x in {1,...,\N}
						\node[vert] (\x) at (-50 + \x*360/\N:\R cm) {};
						\foreach \x [remember=\x as \lastx (initially 1)] in {1,...,\N,1}
						\path (\x) edge (\lastx);
						
						\draw \foreach \x in {1,...,\N} {(-50 + \x*360/\N:1.5) node {$v_{\x}$}};
						
						\foreach \y in {5,9}
						\node[outvert] (o\y) at (-50 + \y*360/\N:3.5 cm) {$3$};
						
						\foreach \y in {5,9}
						\draw (\y) -- (o\y);
						
						\foreach \y in {5,6,7,8,9}
						\draw (\y) -- (-50 + \y*360/\N:\D cm);

						\draw    (1) to[out=-20,in=40,distance=6cm] (2);
						\node (l1) at (3.7,0.7) {bad 8-face};
						
						\draw    (3) to[out=65,in=125,distance=6cm] (4);
						\node (l2) at (0,4.5) {bad};
						\node (l3) at (0,3.8) {8-face};
						
					\end{scope}
				\end{tikzpicture}
				
				\caption{The impossible configuration from Case 4 of Lemma \ref{lem:good-9}.}
				\label{fig:good-9-4}
			\end{figure}    
		\end{description}
		This concludes the proof of Lemma~\ref{lem:good-9}.
	\end{proof}
	
	\begin{lemma}
		\label{lem:good-10}
		If a 10-face $f$ is incident with $x$ 2-vertices and $y$ bad 8-faces, then $x+y \leq 4$.
	\end{lemma}
	
	\begin{proof}
		Let the vertices of $f$ be $v_1, \ldots, v_{10}$. Since $f$ is a 10-face, $x \leq 4$ by Lemma \ref{lem:10-face}. Suppose that $x + y \geq 5$. We distinguish between the following cases.
		\begin{description}
			\item[Case 1.] $x = 4$ and $y \geq 1$.\\
			Since $y \geq 1$, we may assume that $\{v_1, v_2\}$ is a pair of contributing vertices. Thus, by Lemma \ref{lem:good-3333}, $v_3$ and $v_{10}$ are 3-vertices; see Figure \ref{fig:good-10-1}. Among the remaining six consecutive vertices of $f$, four must be 2-vertices (since $x = 4$). But this contradicts Lemma \ref{lem:6-path}.
			
			\begin{figure}[h!!]
				\centering
				
				\begin{tikzpicture}
					[scale=0.7,
					vert/.style={circle,fill=black,draw=black, inner sep=0.05cm},
					s/.style={fill=black!15!white, draw=black!15!white, rounded corners},
					s2/.style={fill=white, draw=black!50!white, rounded corners},
					outvert/.style={rectangle,draw=black},
					outedge/.style={line width=1.5pt},
					dom_vert/.style={circle,draw=black,fill=white, inner sep=0.05cm}
					] 
					\pgfmathtruncatemacro{\N}{10}
					\pgfmathtruncatemacro{\R}{2}
					\pgfmathtruncatemacro{\D}{3}
					\pgfmathtruncatemacro{\r}{35}
					
					%\draw[gray,step=0.25] (-2.25,-2.25) grid (2.25, 3.5);
					\clip (-2.25,-2.25) rectangle (2.25, 3.5); % to remove the white space above and below
					
					\begin{scope}[scale=0.8]
						\filldraw[s2] (-2.5,-2.5) rectangle (2.5,0.5){};
						
						\foreach \x in {1,...,\N}
						\node[vert] (\x) at (\r + \x*360/\N:\R cm) {};
						\foreach \x [remember=\x as \lastx (initially 1)] in {1,...,\N,1}
						\path (\x) edge (\lastx);
						
						\draw \foreach \x in {1,...,\N} {(\r + \x*360/\N:1.5) node {$v_{\x}$}};
						
						\foreach \y in {3,10}
						\draw (\y) -- (\r + \y*360/\N:\D cm);
						
						\draw    (1) to[out=72,in=108,distance=3cm] (2);
						\node (l) at (0,3) {bad};
						
					\end{scope}
				\end{tikzpicture}
				
				\caption{The contradictory configuration from Case 1 of Lemma \ref{lem:good-10}. To keep the figure clearer, we write ``bad'' instead of ``bad 8-face''.}
				\label{fig:good-10-1}
			\end{figure}
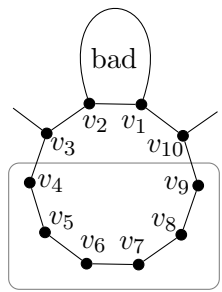
			
			\item[Case 2.] $x = 3$ and $y \geq 2$.\\
			Since $y \geq 2$, we may assume that $\{v_1, v_2\}$ is a pair of contributing vertices. Thus, by Lemma \ref{lem:good-3333}, $v_3$ and $v_{10}$ are 3-vertices, as is one of $v_4, v_9$. Without loss of generality, let $v_4$ be a 3-vertex; see Figure \ref{fig:good-10-2}. Since $x=3$, there are three 2-vertices among $v_5, \ldots, v_9$, and since $y \geq 2$, there is another bad 8-face incident with $f$. The only possible contributing pairs are $\{v_i, v_{i+1}\}$, $i \in \{4, \ldots, 9\}$.
			
			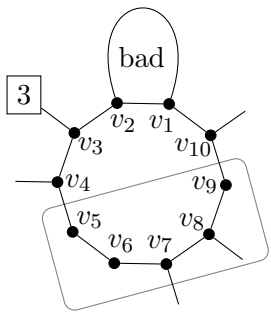
\begin{figure}[h!!]
				\centering
				
				\begin{tikzpicture}
					[scale=0.7,
					vert/.style={circle,fill=black,draw=black, inner sep=0.05cm},
					s/.style={fill=black!15!white, draw=black!15!white, rounded corners},
					s2/.style={fill=white, draw=black!50!white, rounded corners},
					outvert/.style={rectangle,draw=black},
					outedge/.style={line width=1.5pt},
					dom_vert/.style={circle,draw=black,fill=white, inner sep=0.05cm}
					] 
					\pgfmathtruncatemacro{\N}{10}
					\pgfmathtruncatemacro{\R}{2}
					\pgfmathtruncatemacro{\D}{3}
					\pgfmathtruncatemacro{\r}{35}
					
					\clip (-2.75,-2.5) rectangle (2.5, 3.5); % to remove the white space above and below
					
					\begin{scope}[scale=0.8]
						\filldraw[s2, rotate=15] (-2.5,-2.5) rectangle (2.5,0){};
						
						\foreach \x in {1,...,\N}
						\node[vert] (\x) at (\r + \x*360/\N:\R cm) {};
						\foreach \x [remember=\x as \lastx (initially 1)] in {1,...,\N,1}
						\path (\x) edge (\lastx);
						
						\draw \foreach \x in {1,...,\N} {(\r + \x*360/\N:1.5) node {$v_{\x}$}};
						
						\foreach \y in {3}
						\node[outvert] (o\y) at (\r + \y*360/\N:3.5 cm) {$3$};
						
						\foreach \y in {3}
						\draw (\y) -- (o\y);
						
						\foreach \y in {4,7,8,10}
						\draw (\y) -- (\r + \y*360/\N:\D cm);
						
						\draw    (1) to[out=72,in=108,distance=3cm] (2);
						\node (l) at (0,3) {bad};
						
					\end{scope}
				\end{tikzpicture}
				
				\caption{The contradictory configuration from Case 2 of Lemma \ref{lem:good-10}.}
				\label{fig:good-10-2}
			\end{figure}
			
			It follows from Lemmas \ref{lem:2-2-2} and \ref{lem:2-3-2} that $(2,2,3,3,2)$ and $(2,3,3,2,2)$ are the only possible degree sequences of vertices $v_5, \ldots, v_9$. So $\deg_G(v_5) = \deg_G(v_9) = 2$, $\deg_G(v_7) = 3$, and one of $v_6, v_8$ is a 2-vertex and the other a 3-vertex. The following argument is analogous for both cases, so we assume that $\deg_G(v_6) = 2$. Thus, the only remaining possible pair of contributing vertices is $\{v_7, v_8\}$, but since both of these vertices have a 2-vertex as a neighbor, they cannot be contributing vertices. Thus $y = 1$, which is a contradiction.
			
			\item[Case 3.] $x = 2$ and $y \geq 3$.\\
			Since $x=2$, the two 2-vertices can either be adjacent or not. Note that by Lemma \ref{lem:2-3-2} in the second case, the 2-vertices are at distance at least 3.
			
			\begin{description}
				\item[Case 3.a.] The 2-vertices are adjacent.\\
				Without loss of generality, let $v_1$ and $v_2$ be the 2-vertices of $f$. By Lemma \ref{lem:2-3-2}, all neighbors of $v_3$ and $v_{10}$ are 3-vertices. Thus, there is only five remaining possible contributing pairs: $\{v_i, v_{i+1}\}$, $i \in \{ 4, \ldots, 8 \}$; see Figure \ref{fig:good-10-3a}. Since $y \geq 3$ and by Lemma \ref{lem:adj-bad-faces}, the position of the bad 8-faces must be on $\{v_4, v_5\}, \{v_6, v_7\}, \{v_8, v_9\}$. However, this gives a contradiction with Lemma \ref{lem:good-3333}, since for example, both neighbors of the pair $\{v_6, v_7\}$ on $f$ are adjacent to a 2-vertex.
				
				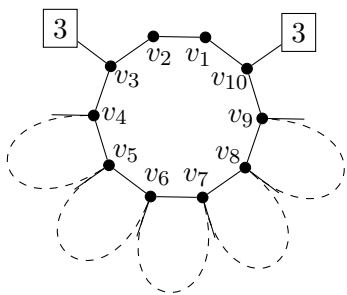
\begin{figure}[h!!]
					\centering
					
					\begin{tikzpicture}
						[scale=0.7,
						vert/.style={circle,fill=black,draw=black, inner sep=0.05cm},
						s/.style={fill=black!15!white, draw=black!15!white, rounded corners},
						s2/.style={fill=white, draw=black!50!white, rounded corners},
						outvert/.style={rectangle,draw=black},
						outedge/.style={line width=1.5pt},
						dom_vert/.style={circle,draw=black,fill=white, inner sep=0.05cm}
						] 
						\pgfmathtruncatemacro{\N}{10}
						\pgfmathtruncatemacro{\R}{2}
						\pgfmathtruncatemacro{\D}{3}
						\pgfmathtruncatemacro{\r}{35}
						
						\clip (-3.5,-3.5) rectangle (3.5, 2.25); % to remove the white space above and below
						
						\begin{scope}[scale=0.8]
							
							\foreach \x in {1,...,\N}
							\node[vert] (\x) at (\r + \x*360/\N:\R cm) {};
							\foreach \x [remember=\x as \lastx (initially 1)] in {1,...,\N,1}
							\path (\x) edge (\lastx);
							
							\draw \foreach \x in {1,...,\N} {(\r + \x*360/\N:1.5) node {$v_{\x}$}};
							
							\foreach \y in {3,10}
							\node[outvert] (o\y) at (\r + \y*360/\N:3.5 cm) {$3$};
							
							\foreach \y in {3,10}
							\draw (\y) -- (o\y);
							
							\foreach \y in {4,5,6,7,8,9}
							\draw (\y) -- (\r + \y*360/\N:\D cm);
							
							\draw[dashed] (4) to[out=180,in=212,distance=3cm] (5);
							\draw[dashed] (5) to[out=212,in=248,distance=3cm] (6);
							\draw[dashed] (6) to[out=248,in=284,distance=3cm] (7);
							\draw[dashed] (7) to[out=284,in=320,distance=3cm] (8);
							\draw[dashed] (8) to[out=320,in=356,distance=3cm] (9);
							
						\end{scope}
					\end{tikzpicture}
					
					\caption{The contradictory configuration from Case 3.a of Lemma \ref{lem:good-10}.}
					\label{fig:good-10-3a}
				\end{figure}
				
				\item[Case 3.b.] The 2-vertices are not adjacent.\\
				Without loss of generality, let $v_1$ be a 2-vertex. By Lemma \ref{lem:2-3-2}, all other neighbors of $v_2$ and $v_{10}$ are 3-vertices. Thus, using symmetry, we only need to consider the cases when the other 2-vertex is one of $v_4, v_5, v_6$. 
				
				If $\deg_G(v_4)=2$, then by Lemma \ref{lem:2-3-2}, all other neighbors of $v_5$ are 3-vertices and thus, the only possible contributing pairs are $\{v_6, v_7\}, \{v_7, v_8\}, \{v_8, v_9\}$, so by Lemma \ref{lem:adj-bad-faces}, $y \leq 2$, which is a contradiction; see Figure \ref{fig:good-10-3b}(i). 
				
				\begin{figure}[h!!]
					\centering
					
					\begin{tikzpicture}
						[scale=0.7,
						vert/.style={circle,fill=black,draw=black, inner sep=0.05cm},
						s/.style={fill=black!15!white, draw=black!15!white, rounded corners},
						s2/.style={fill=white, draw=black!50!white, rounded corners},
						outvert/.style={rectangle,draw=black},
						outedge/.style={line width=1.5pt},
						dom_vert/.style={circle,draw=black,fill=white, inner sep=0.05cm}
						] 
						\pgfmathtruncatemacro{\N}{10}
						\pgfmathtruncatemacro{\R}{2}
						\pgfmathtruncatemacro{\D}{3}
						\pgfmathtruncatemacro{\r}{35}

						\begin{scope}[scale=0.8]
							
							\foreach \x in {1,...,\N}
							\node[vert] (\x) at (\r + \x*360/\N:\R cm) {};
							\foreach \x [remember=\x as \lastx (initially 1)] in {1,...,\N,1}
							\path (\x) edge (\lastx);
							
							\draw \foreach \x in {1,...,\N} {(\r + \x*360/\N:1.5) node {$v_{\x}$}};
							
							\foreach \y in {2,5,10}
							\node[outvert] (o\y) at (\r + \y*360/\N:3.5 cm) {$3$};
							
							\foreach \y in {2,5,10}
							\draw (\y) -- (o\y);
							
							\foreach \y in {3,6,7,8,9}
							\draw (\y) -- (\r + \y*360/\N:\D cm);
							
							\draw[dashed] (6) to[out=248,in=284,distance=3cm] (7);
							\draw[dashed] (7) to[out=284,in=320,distance=3cm] (8);
							\draw[dashed] (8) to[out=320,in=356,distance=3cm] (9);
							
							\node (i) at (0,-5) {(i)};
							
						\end{scope}
						
						\begin{scope}[scale=0.8, xshift=9cm]
							
							\foreach \x in {1,...,\N}
							\node[vert] (\x) at (\r + \x*360/\N:\R cm) {};
							\foreach \x [remember=\x as \lastx (initially 1)] in {1,...,\N,1}
							\path (\x) edge (\lastx);
							
							\draw \foreach \x in {1,...,\N} {(\r + \x*360/\N:1.5) node {$v_{\x}$}};
							
							\foreach \y in {2,4,6,10}
							\node[outvert] (o\y) at (\r + \y*360/\N:3.5 cm) {$3$};
							
							\foreach \y in {2,4,6,10}
							\draw (\y) -- (o\y);
							
							\foreach \y in {3,7,8,9}
							\draw (\y) -- (\r + \y*360/\N:\D cm);
							
							\draw[dashed] (7) to[out=284,in=320,distance=3cm] (8);
							\draw[dashed] (8) to[out=320,in=356,distance=3cm] (9);
							
							\node (i) at (0,-5) {(ii)};
							
						\end{scope}
						
						\begin{scope}[scale=0.8, xshift=18cm]
							
							\foreach \x in {1,...,\N}
							\node[vert] (\x) at (\r + \x*360/\N:\R cm) {};
							\foreach \x [remember=\x as \lastx (initially 1)] in {1,...,\N,1}
							\path (\x) edge (\lastx);
							
							\draw \foreach \x in {1,...,\N} {(\r + \x*360/\N:1.5) node {$v_{\x}$}};
							
							\foreach \y in {2,5,7,10}
							\node[outvert] (o\y) at (\r + \y*360/\N:3.5 cm) {$3$};
							
							\foreach \y in {2,5,7,10}
							\draw (\y) -- (o\y);
							
							\foreach \y in {3,4,8,9}
							\draw (\y) -- (\r + \y*360/\N:\D cm);
							
							\draw[dashed] (3) to[out=146,in=180,distance=3cm] (4);
							\draw[dashed] (8) to[out=320,in=356,distance=3cm] (9);
							
							\node (i) at (0,-5) {(iii)};
							
						\end{scope}
					\end{tikzpicture}
					
					\caption{The contradictory configurations from Case 3.b of Lemma \ref{lem:good-10}.}
					\label{fig:good-10-3b}
				\end{figure}
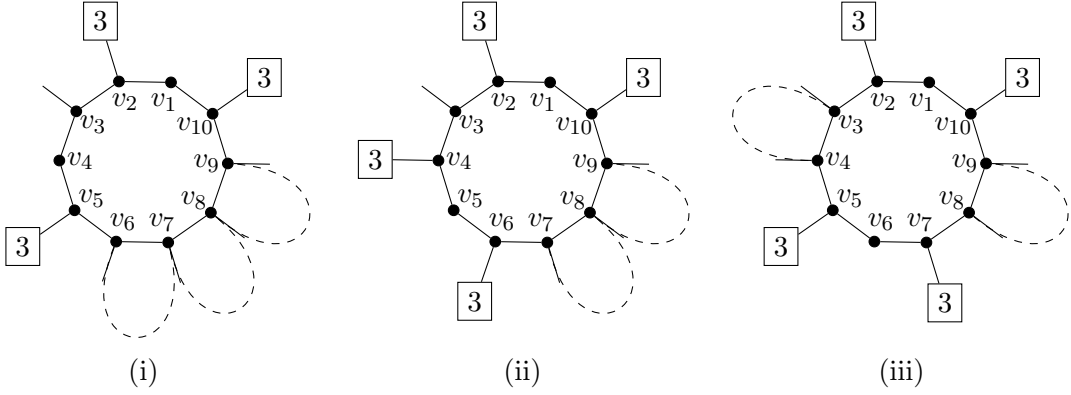
				
				If $\deg_G(v_5)=2$, then by Lemma \ref{lem:2-3-2}, all other neighbors of $v_4$ and $v_6$ are 3-vertices and thus, the only possible contributing pairs are $\{v_7, v_8\}, \{v_8, v_9\}$, so by Lemma \ref{lem:adj-bad-faces}, $y \leq 1$, which is a contradiction; see Figure \ref{fig:good-10-3b}(ii). 
				
				If $\deg_G(v_6)=2$, then by Lemma \ref{lem:2-3-2}, all other neighbors of $v_5$ and $v_7$ are 3-vertices and thus, the only possible contributing pairs are $\{v_3, v_4\}, \{v_8, v_9\}$, so by Lemma \ref{lem:adj-bad-faces}, $y \leq 2$, which is a contradiction; see Figure \ref{fig:good-10-3b}(iii).
				
			\end{description}
			
			\item[Case 4.] $x = 1$ and $y \geq 4$.\\
			Without loss of generality, let $v_1$ be the only 2-vertex of $f$ (since $x = 1$). By Lemma \ref{lem:good-3333}, the pairs $\{v_2, v_3\}$ and $\{v_9, v_{10}\}$ cannot be pairs of contributing vertices; see Figure \ref{fig:good-10-4}. Thus, the only possible pairs of contributing vertices are the six pairs $\{v_i, v_{i+1}\}$, $i \in \{3,\ldots, 8\}$. Since by Lemma \ref{lem:adj-bad-faces}, no two bad 8-faces share a contributing vertex, we have $y \leq 3$, which is a contradiction.
			
			\begin{figure}[h!!]
				\centering
				
				\begin{tikzpicture}
					[scale=0.7,
					vert/.style={circle,fill=black,draw=black, inner sep=0.05cm},
					s/.style={fill=black!15!white, draw=black!15!white, rounded corners},
					s2/.style={fill=white, draw=black!50!white, rounded corners},
					outvert/.style={rectangle,draw=black},
					outedge/.style={line width=1.5pt},
					dom_vert/.style={circle,draw=black,fill=white, inner sep=0.05cm}
					] 
					\pgfmathtruncatemacro{\N}{10}
					\pgfmathtruncatemacro{\R}{2}
					\pgfmathtruncatemacro{\D}{3}
					\pgfmathtruncatemacro{\r}{35}
					
					\clip (-3.5,-3.5) rectangle (3.5, 2.5); % to remove the white space above and below
					
					\begin{scope}[scale=0.8]
						
						\foreach \x in {1,...,\N}
						\node[vert] (\x) at (\r + \x*360/\N:\R cm) {};
						\foreach \x [remember=\x as \lastx (initially 1)] in {1,...,\N,1}
						\path (\x) edge (\lastx);
						
						\draw \foreach \x in {1,...,\N} {(\r + \x*360/\N:1.5) node {$v_{\x}$}};
						
						\foreach \y in {2,3,4,5,6,7,8,9,10}
						\draw (\y) -- (\r + \y*360/\N:\D cm);
						
						\draw[dashed] (3) to[out=146,in=180,distance=3cm] (4);
						\draw[dashed] (4) to[out=180,in=212,distance=3cm] (5);
						\draw[dashed] (5) to[out=212,in=248,distance=3cm] (6);
						\draw[dashed] (6) to[out=248,in=284,distance=3cm] (7);
						\draw[dashed] (7) to[out=284,in=320,distance=3cm] (8);
						\draw[dashed] (8) to[out=320,in=356,distance=3cm] (9);
						
					\end{scope}
				\end{tikzpicture}
				
				\caption{The contradictory configuration from Case 4 of Lemma \ref{lem:good-10}.}
				\label{fig:good-10-4}
			\end{figure}
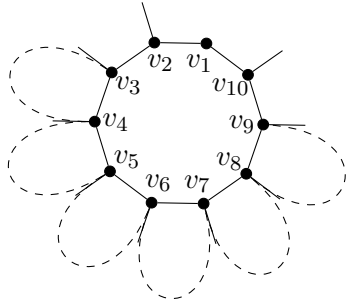
			
			\item[Case 5.] $x = 0$ and $y \geq 5$.\\
			Since $x = 0$, all vertices of $f$ are 3-vertices. By Lemma \ref{lem:adj-bad-faces}, no two bad 8-faces share a vertex on $f$. Thus $y \leq 5$, which imples $y = 5$, and without loss of generality, $\{v_1, v_2\}, \{v_3, v_4\}, \{v_5, v_6\}, \{v_7, v_8\}, \{v_9, v_{10}\}$ are pairs of contributing vertices; see Figure \ref{fig:good-10-5}. But this contradicts Lemma \ref{lem:good-3333}, since for example, both neighbors of the pair $\{v_1,v_2\}$ on $f$ are adjacent to one 2-vertex.
			
			\begin{figure}[h!!]
				\centering
				\begin{tikzpicture}
					[scale=0.7,
					vert/.style={circle,fill=black,draw=black, inner sep=0.05cm},
					s/.style={fill=black!15!white, draw=black!15!white, rounded corners},
					s2/.style={fill=white, draw=black!50!white, rounded corners},
					outvert/.style={rectangle,draw=black},
					outedge/.style={line width=1.5pt},
					dom_vert/.style={circle,draw=black,fill=white, inner sep=0.05cm}
					] 
					\pgfmathtruncatemacro{\N}{10}
					\pgfmathtruncatemacro{\R}{2}
					\pgfmathtruncatemacro{\D}{3}
					\pgfmathtruncatemacro{\r}{35}
					
					%\draw[gray,step=0.25] (-3.5,-3) grid (3.5, 3.5);
					\clip (-3.5,-3) rectangle (3.5, 3.5); % to remove the white space above and below
					
					\begin{scope}[scale=0.8]
						\filldraw[s2, rotate=0, shift={(-2.4,1.2)}] (0,0) rectangle (0.7,0.7){};
						\filldraw[s2, rotate=0, shift={(1.8,1.1)}] (0,0) rectangle (0.7,0.7){};
						
						\foreach \x in {1,...,\N}
						\node[vert] (\x) at (\r + \x*360/\N:\R cm) {};
						\foreach \x [remember=\x as \lastx (initially 1)] in {1,...,\N,1}
						\path (\x) edge (\lastx);
						
						\draw \foreach \x in {1,...,\N} {(\r + \x*360/\N:1.5) node {$v_{\x}$}};
						
						\foreach \y in {3,10}
						\node[vert] (o\y) at (\r + \y*360/\N:2.5 cm) {};
						
						\draw    (1) to[out=72,in=108,distance=3cm] (2);
						\node (l1) at (0,3) {bad};
						\draw (3) to[out=146,in=180,distance=3cm] (4);
						\node (l3) at (-3,1) {bad};
						\draw (5) to[out=212,in=248,distance=3cm] (6);
						\node (l5) at (-2,-2.5) {bad};
						\draw (7) to[out=284,in=320,distance=3cm] (8);
						\node (l7) at (1.8,-2.5) {bad};
						\draw (9) to[out=0,in=36,distance=3cm] (10);
						\node (l9) at (3,1) {bad};
						
					\end{scope}
				\end{tikzpicture}
				
				\caption{The contradictory configuration from Case 5 of Lemma \ref{lem:good-10}.}
				\label{fig:good-10-5}
			\end{figure}
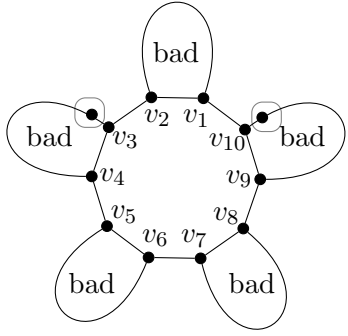
			
		\end{description}
		This concludes the proof of Lemma~\ref{lem:good-10}.
	\end{proof}

	\subsection{Discharging}
	\label{sec:discharging8}
	
	For $x \in V(G) \cup F(G)$, we define the following initial charge:
	$$\mu(x) = \begin{cases}
		2 \deg_G(x) - 6 & \text{if } x \in V(G),\\
		\ell(x) - 6 & \text{if } x \in F(G).
	\end{cases}$$
	Since $G$ is planar, Euler's formula implies that the initial charge of $G$ is
	$$\mu(G) = \sum_{x \in V(G) \cup F(G)} \mu(x) = -12.$$
	
	We use the following two discharging rules: 
	\begin{itemize}
		\item Every face $f$ sends charge 1 to each of the 2-vertices on the facial walk of $f$, counted with multiplicities.
		
		\item Every face sends charge $1$ to each of the bad 8-faces incident with it.
	\end{itemize}
	Let $\mu^*(x)$ denote the charge of vertices and faces after applying the discharging rules. 
	
	\begin{lemma}
		\label{lem:charge-positive-2}
		For every $x \in V(G) \cup F(G)$, $\mu^*(x) \geq 0$.
	\end{lemma}
	
	\begin{proof}
		By Lemma~\ref{lem:no-deg-1}, $G$ has no vertex of degree $1$.
		A 2-vertex that is incident with two different faces receives charge 1 from each of them. A 2-vertex that is incident with only one face $f$ appears on the facial walk of $f$ twice, so it receives charge 2 from $f$.
		Thus $\mu^*(v) = (2 \cdot 2 -6) + 2 \cdot 1 = 0$. If $v$ is a vertex of degree 3, then $\mu^*(v) = \mu(v) = 0$.
		
		If $f$ is a $k$-face, $k \geq 11$, then by Lemma \ref{lem:good-11+}, the facial walk of $f$ contains at most $ \left \lfloor \frac{k}{2} \right \rfloor$ 2-vertices and bad 8-faces. Thus $$\mu^*(f) \geq (k-6) - \left \lfloor \frac{k}{2} \right \rfloor = \left \lceil \frac{k}{2} \right \rceil - 6 \geq 0$$ since $k \geq 11$.
		
		If $f$ is a 10-face, then by Lemma \ref{lem:good-10}, $f$ is incident with at most four 2-vertices and bad 8-faces. Hence, $\mu^*(f) \geq (10-6) - 4 = 0$. If $f$ is a 9-face, then by Lemma \ref{lem:good-9}, $f$ is incident with at most three 2-vertices and bad 8-faces. Hence, $\mu^*(f) \geq (9-6) - 3 = 0$. If $f$ is a good 8-face, then by Lemma \ref{lem:good-8}, $f$ is incident with at most two 2-vertices and bad 8-faces. Hence, $\mu^*(f) \geq (8-6) - 2 = 0$.
		
		If $f$ is a bad 8-face, then it is incident with three 2-vertices and one contributing face. By Lemma \ref{lem:adj-bad-faces}, it follows that $f$ cannot be a contributing face of another bad 8-face. So $\mu^*(f) = (8-6) - 3 + 1 = 0$.
	\end{proof}
	
	Clearly it holds that $$-12 = \mu(G) = \mu^*(G) \geq 0,$$ which is a contradiction. This concludes the proof of Theorem~\ref{thm:main}.

	\section{Different weights}
	\label{sec:weights}
	
	It is not clear how to measure how good a certain bound for the domination number of subcubic planar graphs is. The result in Theorem \ref{thm:main} minimizes the weight of 3-vertices, but there are other bounds we can prove with the same method. One of them is the following.
	
	\begin{theorem}
		\label{thm:alternative}
		If $G$ is a planar subcubic graph with girth at least $9$, then 
		$$17 \gamma(G) < 17 n_0(G) + 13n_1(G) + 9 n_2(G) + 6n_3(G).$$
	\end{theorem}
	
	The proof of Theorem \ref{thm:alternative} uses the same ideas as the proof of Theorem \ref{thm:main} with the following modifications. Let $w_i$ be the coefficient of $n_i(G)$ in  $$\gamma(G) \le w_0n_0(G)+w_1n_1(G)+w_2n_2(G)+w_3n_3(G).$$
	\begin{enumerate}
		\item Lemmas \ref{lem:1-2} -- \ref{lem:2-2-2} (and their proofs) stay the same.
		\item To prove Lemma \ref{lem:no-deg-1}, additional configurations must be considered first, to deduce more information about the degree of the vertices in $N_G(S) \setminus S$. In particular, we first prove that the sequences $(2,2,3,3_3^1)$ and $(_2^2 3, 3_3^1)$ are reducible, and only then proceed to proving Lemma \ref{lem:no-deg-1}. The reason for this lies in the fact that in Theorem \ref{thm:alternative} $w_2 - w_3 < w_1 - w_2$, while in Theorem \ref{thm:main} these two differences are the same.
		\item Lemma \ref{lem:2-3-2} \textit{does not hold}. Instead, we first prove that there is no 3-vertex with all three neighbors of degree $2$, and then we prove that $(_2^3 3, 2, 2, 3_2^3)$, $(2,2,3,2,2,3)$, $(_2^2 3, 3_2^2)$, $(2,2,3,2,3,2)$, $(2,2,3,3,2,2)$, and $(3,2,3,2,3)$ are reducible, to conclude the result in Lemma \ref{lem:6-path}. 
		\item Lemma \ref{lem:short-path} is modified to consider only $9^+$-faces. Since we do not have Lemma \ref{lem:2-3-2}, the proof becomes much more technical.
		\item Lemma \ref{lem:11+face} stays the same.
		\item Lemmas \ref{lem:10-face} and \ref{lem:9-face} stay the same, but the proofs are slightly different (since we do not have Lemma \ref{lem:2-3-2}).
		\item Since in Theorem \ref{thm:alternative}, we only consider graphs with girth at least $9$, Lemmas \ref{lem:8-face-weak} -- \ref{lem:good-10} are not needed.
		\item The discharging remains the same, except that we do not consider bad 8-faces and we restrict our focus to $9^+$-faces only (dealing with 8-faces in this case seems to be more tricky than in Theorem \ref{thm:main}).
	\end{enumerate}
	
	Writing down the conditions needed to satisfy the steps in our proof using weights $w_0, \ldots, w_3$ instead of concrete values for weights, we obtain a family of bounds instead of only the bound from Theorem \ref{thm:main} (this can be easily done using a computer). As already mentioned, minimizing $w_3$ gives the set of weights from Theorem \ref{thm:main}, and to obtain the weights from Theorem \ref{thm:alternative}, we can for example minimize $w_2 + 6 w_3$.
	
	\section{Conclusion}
	\label{sec:further}
	
	For subcubic planar graphs with a smaller girth restriction, it is not clear what the best upper bound on the domination number is. The method of Theorem~\ref{thm:main} might be able to be pushed to include girths $6$ and $7$ with significantly more work, but there seem to be barriers of our method to handling cycles of length less than $6$.
	
	For subcubic planar graphs with no girth restriction, the upper bound of Theorem~\ref{thm:main} no longer holds.
	Through an exhaustive computer search of subcubic planar graphs on at most $19$ vertices (using SageMath~\cite{sagemath} and a customized version of \texttt{geng} from the \texttt{nauty} package~\cite{nauty}),
	we found two graphs (with girth $3$ or $4$) that exceed the upper bound of Theorem~\ref{thm:main} (see 
	Figure~\ref{fig:exceed}).
	Thus any theorem proving the upper bound of Theorem~\ref{thm:main} must require at least girth $5$.
	
	\begin{figure}[h!!]
		\centering
		
		\begin{tikzpicture}
			[
			scale=0.7,
			vert/.style={circle, fill=black, draw=black, inner sep=0.05cm}, 
			s/.style={fill=black!15!white, draw=black!15!white, rounded corners},
			outvert/.style={rectangle, draw=black},
			outedge/.style={line width=1.5pt},
			dom_vert/.style={circle,draw=black,fill=white, inner sep=0.05cm}
			]
			\begin{scope}
				\node[vert] (0) at (0,0) {};
				\node[vert] (1) at (2,0) {};
				\node[dom_vert] (2) at (0,2) {};
				\node[dom_vert] (3) at (1,-1) {};
				\node[vert] (4) at (2,2) {};
				\draw (0) -- (2) -- (4) -- (1) -- (3) -- (0) -- (4);
				\draw (2) -- (1);
				\node (n1) at (1,0) {$G_1$};
			\end{scope}
			
			\begin{scope}[xshift = 5cm]
				\node[vert] (0) at (0,0) {};
				\node[vert] (1) at (2,0) {};
				\node[dom_vert] (2) at (0,2) {};
				\node[dom_vert] (3) at (1,-1) {};
				\node[dom_vert] (4) at (2,2) {};
				
				\node[vert] (8') at (0,3) {};
				\node[vert] (9') at (2,3) {};
				\node[vert] (0') at (0,5) {};
				\node[vert] (1') at (2,5) {};
				\node[dom_vert] (5') at (1,6) {};
				
				\draw (0) -- (2);
				\draw (4) -- (1) -- (3) -- (0) -- (4);
				\draw (2) -- (1);
				
				\draw (2) -- (8') -- (0') -- (5') -- (1') -- (9') -- (4);
				\draw (8') -- (1');
				\draw (0') -- (9');
				
				\node (n2) at (1,2.5) {$G_2$};
			\end{scope}

		\end{tikzpicture}
		
		\caption{The only subcubic planar graphs with no girth restriction and at most $19$ vertices that exceed the upper bound of Theorem~\ref{thm:main}. The graph on the left has girth $3$, domination number $\gamma(G_1) = 2$, and the upper bound of Theorem~\ref{thm:main} gives $\frac{39}{20}$.
			The graph on the right has girth $4$, domination number $\gamma(G_2)=4$, and the upper bound gives $\frac{39}{10}$. 
			Examples of minimal dominating sets are marked with white vertices.
			Note that both graphs are planar.}
		\label{fig:exceed}
	\end{figure}
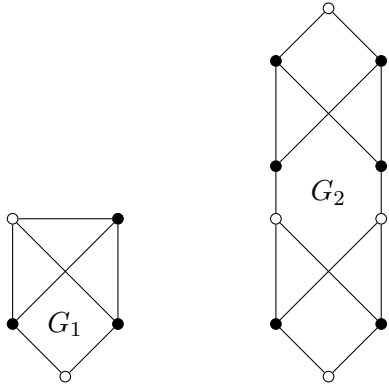
	
	However, we believe that these are sporadic small examples.
	The question remains of what the best asymptotic upper bound is for subcubic planar graphs with no girth restriction.
	By considering disjoint copies of $K_1$, $K_2$, $C_4$, or $K_3 \,\square\, K_2$, no upper bound can improve the individual coefficients in $\gamma(G) \leq n_0(G) + \frac12 n_1(G) + \frac12 n_2(G) + \frac13 n_3(G)$.
	However, this bound is not possible in general, as can be seen by the following two infinite families.
	
	For $n\geq 3$, let $C_n^\star$ be the graph obtained by attaching a leaf to each vertex of $C_n$. The graph $C_n^\star$ has $2n$ vertices ($n$ 1-vertices and $n$ 3-vertices) and $\gamma(C_n^\star) = n$. Then the family $\{C_n^\star\}_{n=3}^\infty$ shows that an upper bound on the domination number cannot simultaneously have a coefficient of $\frac{1}{2}$ on $n_1(G)$ and a coefficient of $\frac{1}{3}$ on $n_3(G)$. More precisely, the coefficient on $n_1(G)$ has to be at least $1 - w_3$, where $w_3$ is the coefficient on $n_3(G)$. For example, if $w_3 = \frac{7}{20}$, the best we can hope for is a coefficient of $\frac{13}{20}$ on $n_1(G)$.
	
	For $n \geq 3$, let $H_n$ be the graph obtained from the cycle $C_n$ by attaching the graph $G_2$ from Figure~\ref{fig:exceed} to each vertex with an edge (see Figure~\ref{fig:family} for a picture of $H_6$). The graph $H_n$ has $11n$ vertices: $n$ $2$-vertices and $10n$ $3$-vertices. 
	Since $\gamma(G_2) = 4$ (see Figure~\ref{fig:exceed}) and vertices adjacent to the cycle $C_n$ can be in the dominating set, we have $\gamma(H_n) \leq 4n$.
	To see that $\gamma(H_n) \geq 4n$, notice that each gadget consisting of $G_2$ and the adjacent vertex on $C_n$ contains at least four vertices from every minimal dominating set of $H_n$.
	Thus, $\gamma(H_n) = 4n$.
	Note that the family $\{H_n\}_{n=3}^{\infty}$ shows that an upper bound on the domination number cannot simultaneously have a coefficient of $\frac{1}{2}$ on $n_2(G)$ and a coefficient of $\frac{1}{3}$ on $n_3(G)$.
	
	\begin{figure}[h!!]
		\centering
		
		\begin{tikzpicture}
			[
			scale=0.5,
			vert/.style={circle, fill=black, draw=black, inner sep=0.05cm}, 
			s/.style={fill=black!15!white, draw=black!15!white, rounded corners},
			outvert/.style={rectangle, draw=black},
			outedge/.style={line width=1.5pt},
			dom_vert/.style={circle,draw=black,fill=white, inner sep=0.05cm}
			]
			
			\node[vert] (a0) at (0,0) {};
			\node[vert] (a1) at (2,0) {};
			\node[vert] (a2) at (0,2) {};
			\node[vert] (a3) at (1,-1) {};
			\node[vert] (a4) at (2,2) {};
			
			\node[vert] (a8') at (0,3) {};
			\node[vert] (a9') at (2,3) {};
			\node[vert] (a0') at (0,5) {};
			\node[vert] (a1') at (2,5) {};
			\node[vert] (a5') at (1,6) {};
			
			\draw (a0) -- (a2);
			\draw (a4) -- (a1) -- (a3) -- (a0) -- (a4);
			\draw (a2) -- (a1);
			
			\draw (a2) -- (a8') -- (a0') -- (a5') -- (a1') -- (a9') -- (a4);
			\draw (a8') -- (a1');
			\draw (a0') -- (a9');
			
			\node[vert] (b0) at (4,0) {};
			\node[vert] (b1) at (6,0) {};
			\node[vert] (b2) at (4,2) {};
			\node[vert] (b3) at (5,-1) {};
			\node[vert] (b4) at (6,2) {};
			
			\node[vert] (b8') at (4,3) {};
			\node[vert] (b9') at (6,3) {};
			\node[vert] (b0') at (4,5) {};
			\node[vert] (b1') at (6,5) {};
			\node[vert] (b5') at (5,6) {};
			
			\draw (b0) -- (b2);
			\draw (b4) -- (b1) -- (b3) -- (b0) -- (b4);
			\draw (b2) -- (b1);
			
			\draw (b2) -- (b8') -- (b0') -- (b5') -- (b1') -- (b9') -- (b4);
			\draw (b8') -- (b1');
			\draw (b0') -- (b9');
			
			\node[vert] (c0) at (8,0) {};
			\node[vert] (c1) at (10,0) {};
			\node[vert] (c2) at (8,2) {};
			\node[vert] (c3) at (9,-1) {};
			\node[vert] (c4) at (10,2) {};
			
			\node[vert] (c8') at (8,3) {};
			\node[vert] (c9') at (10,3) {};
			\node[vert] (c0') at (8,5) {};
			\node[vert] (c1') at (10,5) {};
			\node[vert] (c5') at (9,6) {};
			
			\draw (c0) -- (c2);
			\draw (c4) -- (c1) -- (c3) -- (c0) -- (c4);
			\draw (c2) -- (c1);
			
			\draw (c2) -- (c8') -- (c0') -- (c5') -- (c1') -- (c9') -- (c4);
			\draw (c8') -- (c1');
			\draw (c0') -- (c9');
			
			\node[vert] (a) at (1,-2) {};
			\node[vert] (b) at (5,-2) {};
			\node[vert] (c) at (9,-2) {};
			
			\draw (a) -- (a3);
			\draw (b) -- (b3);
			\draw (c) -- (c3);
			
			\draw (a) -- (b) -- (c);

			\node[vert] (Da0) at (0,-11) {};
			\node[vert] (Da1) at (2,-11) {};
			\node[vert] (Da2) at (0,-9) {};
			\node[vert] (Da3) at (1,-12) {};
			\node[vert] (Da4) at (2,-9) {};
			
			\node[vert] (Da8') at (0,-8) {};
			\node[vert] (Da9') at (2,-8) {};
			\node[vert] (Da0') at (0,-6) {};
			\node[vert] (Da1') at (2,-6) {};
			\node[vert] (Da5') at (1,-5) {};
			
			\draw (Da0) -- (Da2);
			\draw (Da4) -- (Da1) -- (Da3) -- (Da0) -- (Da4);
			\draw (Da2) -- (Da1);
			
			\draw (Da2) -- (Da8') -- (Da0') -- (Da5') -- (Da1') -- (Da9') -- (Da4);
			\draw (Da8') -- (Da1');
			\draw (Da0') -- (Da9');
			
			\node[vert] (Db0) at (4,-11) {};
			\node[vert] (Db1) at (6,-11) {};
			\node[vert] (Db2) at (4,-9) {};
			\node[vert] (Db3) at (5,-12) {};
			\node[vert] (Db4) at (6,-9) {};
			
			\node[vert] (Db8') at (4,-8) {};
			\node[vert] (Db9') at (6,-8) {};
			\node[vert] (Db0') at (4,-6) {};
			\node[vert] (Db1') at (6,-6) {};
			\node[vert] (Db5') at (5,-5) {};
			
			\draw (Db0) -- (Db2);
			\draw (Db4) -- (Db1) -- (Db3) -- (Db0) -- (Db4);
			\draw (Db2) -- (Db1);
			
			\draw (Db2) -- (Db8') -- (Db0') -- (Db5') -- (Db1') -- (Db9') -- (Db4);
			\draw (Db8') -- (Db1');
			\draw (Db0') -- (Db9');
			
			\node[vert] (Dc0) at (8,-11) {};
			\node[vert] (Dc1) at (10,-11) {};
			\node[vert] (Dc2) at (8,-9) {};
			\node[vert] (Dc3) at (9,-12) {};
			\node[vert] (Dc4) at (10,-9) {};
			
			\node[vert] (Dc8') at (8,-8) {};
			\node[vert] (Dc9') at (10,-8) {};
			\node[vert] (Dc0') at (8,-6) {};
			\node[vert] (Dc1') at (10,-6) {};
			\node[vert] (Dc5') at (9,-5) {};
			
			\draw (Dc0) -- (Dc2);
			\draw (Dc4) -- (Dc1) -- (Dc3) -- (Dc0) -- (Dc4);
			\draw (Dc2) -- (Dc1);
			
			\draw (Dc2) -- (Dc8') -- (Dc0') -- (Dc5') -- (Dc1') -- (Dc9') -- (Dc4);
			\draw (Dc8') -- (Dc1');
			\draw (Dc0') -- (Dc9');
			
			\node[vert] (Da) at (1,-4) {};
			\node[vert] (Db) at (5,-4) {};
			\node[vert] (Dc) at (9,-4) {};
			
			\draw (Da) -- (Da5');
			\draw (Db) -- (Db5');
			\draw (Dc) -- (Dc5');
			
			\draw (a) -- (Da) -- (Db) -- (Dc) -- (c);
			
		\end{tikzpicture}
		
		\caption{The graph $H_6$. Note that the graph is planar; it is drawn with edge-crossings only to emphasize the symmetry present.}
		\label{fig:family}
	\end{figure}
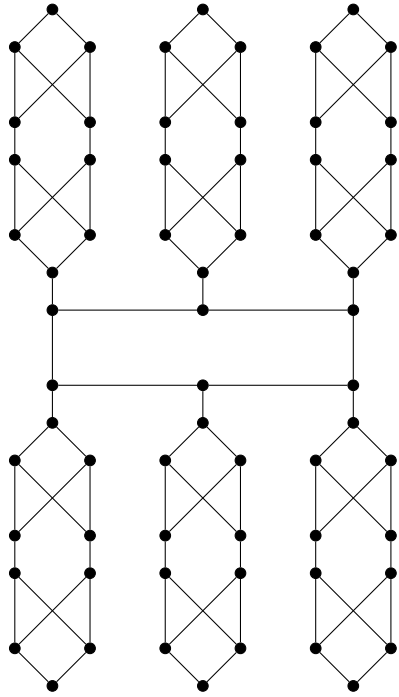

	The upper bound from Theorem~\ref{thm:main} gives $\gamma(H_n) < \frac{81n}{20}$.
	The bound would become tight for $H_n$ if the coefficient on $n_2(G)$ is lowered to $\frac{1}{2}$.
	We conjecture that this is optimal:
	\begin{conjecture}
		There exists a constant $N$ such that any connected
		subcubic planar graph with at least $N$ vertices satisfies
		$$ \gamma(G) \le n_0(G) + \frac{13}{20} n_1(G) + \frac{1}{2} n_2(G) + \frac{7}{20} n_3(G).$$
	\end{conjecture}
	
	Since there are no cubic planar graphs with girth at least $8$, the bound in Theorem~\ref{thm:main} does not apply to cubic planar graphs. To make progress towards Conjectures~\ref{conj:2concubicplanar} and \ref{conj:theother}, and considering the new result \cite{dorbec-henning-1/3}, we are left to consider cubic planar graphs with small girth.  Perhaps a partial result could be obtained by allowing small cycles of certain lengths (such as $3$-cycles) and forbidding other small cycles (such as cycle lengths $4$ through $8$).

	\section*{Acknowledgements}
	
	This work was started at the 2022 Graduate Research Workshop in Combinatorics, which was supported in part by NSF grant 1953985 and a generous award from the Combinatorics Foundation.
	We thank Alvaro Carbonero and Rebecca Robinson for initial discussions.
	
	Eun-Kyung Cho was supported by the Basic Science Research Program through the National Research Foundation of Korea (NRF) funded by the Ministry of Education (No. RS-2023-00244543).
	Vesna Ir\v{s}i\v{c} acknowledges the financial support from the Slovenian Research Agency (P1-0297, J1-2452, N1-0285 and Z1-50003) and the European Union (ERC, KARST, project number 101071836).
	
	\bibliographystyle{abbrv}

\begin{thebibliography}{99}
		
		\bibitem{abrishami+2019}
		G.~Abrishami, M.~A. Henning, and F.~Rahbarnia.
		\newblock On independent domination in planar cubic graphs.
		\newblock {\em Discuss. Math. Graph Theory}, 39(4):841--853, 2019.
		
		\bibitem{blank1973estimate}
		M.~Blank.
		\newblock An estimate of the external stability number of a graph without
		suspended vertices.
		\newblock {\em Prikl. Mat. i Programmirovanie}, 10:3--11, 1973.
		
		\bibitem{cho2021tight}
		E.-K. Cho, I.~Choi, H.~Kwon, and B.~Park.
		\newblock A tight bound for independent domination of cubic graphs without
		4-cycles.
		\newblock {\em J. Graph Theory}, 104(2):372--386, 2023.
		
		\bibitem{dorbec+2019}
		P.~Dorbec, A.~Gonz\'{a}lez, and C.~Pennarun.
		\newblock Power domination in maximal planar graphs.
		\newblock {\em Discrete Math. Theor. Comput. Sci.}, 21(4):Paper No. 18, 24,
		2019.
		
		\bibitem{dorbec-henning-1/3}
		P.~Dorbec and M.~Henning.
		\newblock The $\frac13$-conjectures for domination in cubic graphs.
		\newblock {\em Abstracts of the 10th Slovenian Conference on Graph Theory,
			Kranjska Gora, Slovenia, June 18 -- 24, 2023.}
		
		\bibitem{dorfling+2006}
		M.~Dorfling, W.~Goddard, and M.~A. Henning.
		\newblock Domination in planar graphs with small diameter. {II}.
		\newblock {\em Ars Combin.}, 78:237--255, 2006.
		
		\bibitem{enciso+2008}
		R.~I. Enciso and R.~D. Dutton.
		\newblock Global domination in planar graphs.
		\newblock {\em J. Combin. Math. Combin. Comput.}, 66:273--278, 2008.
		
		\bibitem{favaron92}
		O.~Favaron.
		\newblock A bound on the independent domination number of a tree.
		\newblock {\em Vishwa Internat. J. Graph Theory}, 1(1):19--27, 1992.
		
		\bibitem{goddard+2002}
		W.~Goddard and M.~A. Henning.
		\newblock Domination in planar graphs with small diameter.
		\newblock {\em J. Graph Theory}, 40(1):1--25, 2002.
		
		\bibitem{goddard+2012}
		W.~Goddard, M.~A. Henning, J.~Lyle, and J.~Southey.
		\newblock On the independent domination number of regular graphs.
		\newblock {\em Ann. Comb.}, 16(4):719--732, 2012.
		
		\bibitem{henning+2009}
		M.~A. Henning and J.~McCoy.
		\newblock Total domination in planar graphs of diameter two.
		\newblock {\em Discrete Math.}, 309(21):6181--6189, 2009.
		
		\bibitem{Kawarabayashi+2006}
		K.-i. Kawarabayashi, M.~D. Plummer, and A.~Saito.
		\newblock Domination in a graph with a 2-factor.
		\newblock {\em J. Graph Theory}, 52(1):1--6, 2006.
		
		\bibitem{kelmans2006counterexamples}
		A.~Kelmans.
		\newblock Counterexamples to the cubic graph domination conjecture.
		\newblock {\em arXiv preprint arXiv:math/0607512}, 2006.
		
		\bibitem{Kostochka+2009b}
		A.~V. Kostochka and C.~Stocker.
		\newblock A new bound on the domination number of connected cubic graphs.
		\newblock {\em Sib. \`Elektron. Mat. Izv.}, 6:465--504, 2009.
		
		\bibitem{kostochka2009upper}
		A.~V. Kostochka and B.~Stodolsky.
		\newblock An upper bound on the domination number of $n$-vertex connected cubic
		graphs.
		\newblock {\em Discrete Mathematics}, 309(5):1142--1162, 2009.
		
		\bibitem{Kostochka+2005}
		A.~V. Kostochka and B.~Y. Stodolsky.
		\newblock On domination in connected cubic graphs.
		\newblock {\em Discrete Math.}, 304(1-3):45--50, 2005.
		
		\bibitem{lauri+2020}
		J.~Lauri and C.~Mitillos.
		\newblock Perfect {I}talian domination on planar and regular graphs.
		\newblock {\em Discrete Appl. Math.}, 285:676--687, 2020.
		
		\bibitem{Lowenstein+2008}
		C.~L\"{o}wenstein and D.~Rautenbach.
		\newblock Domination in graphs of minimum degree at least two and large girth.
		\newblock {\em Graphs Combin.}, 24(1):37--46, 2008.
		
		\bibitem{macgillivray+1996}
		G.~MacGillivray and K.~Seyffarth.
		\newblock Domination numbers of planar graphs.
		\newblock {\em J. Graph Theory}, 22(3):213--229, 1996.
		
		\bibitem{McCuaig+1989}
		W.~McCuaig and B.~Shepherd.
		\newblock Domination in graphs with minimum degree two.
		\newblock {\em J. Graph Theory}, 13(6):749--762, 1989.
		
		\bibitem{nauty}
		B.~D. McKay and A.~Piperno.
		\newblock Practical graph isomorphism, ii.
		\newblock {\em Journal of Symbolic Computation}, 60:94--112, 2014.
		
		\bibitem{ore1962theory}
		O.~Ore.
		\newblock {\em Theory of graphs}, volume Vol. XXXVIII of {\em American
			Mathematical Society Colloquium Publications}.
		\newblock American Mathematical Society, Providence, RI, 1962.
		
		\bibitem{Reed1996}
		B.~Reed.
		\newblock Paths, stars and the number three.
		\newblock {\em Combin. Probab. Comput.}, 5(3):277--295, 1996.
		
		\bibitem{Stodolsky2008}
		B.~Y. Stodolsky.
		\newblock On domination in 2-connected cubic graphs.
		\newblock {\em Electron. J. Combin.}, 15(1):Note 38, 5, 2008.
		
		\bibitem{sagemath}
		{The Sage Developers}.
		\newblock {\em {S}ageMath, the {S}age {M}athematics {S}oftware {S}ystem
			({V}ersion 9.5)}, 2022.
		\newblock {\tt https://www.sagemath.org}.
		
		\bibitem{zhu+2015}
		T.~Zhu and B.~Wu.
		\newblock Domination of maximal {$K_4$}-minor free graphs and maximal
		{$K_{2,3}$}-minor free graphs, and disproofs of two conjectures on planar
		graphs.
		\newblock {\em Discrete Appl. Math.}, 194:147--153, 2015.
		
	\end{thebibliography}

\end{document}